\documentclass[a4paper]{article}

\usepackage{amsmath, amsfonts, amsthm, mathrsfs, amssymb}
\usepackage{graphicx, subfigure, threeparttable, booktabs, tabularx}
\usepackage{xcolor}                                  
\usepackage{dsfont}
\usepackage{hyperref}
\usepackage{float}
\usepackage{bbm}
\usepackage[a4paper, left=2.5cm,right=2.5cm,top=2.5cm,bottom=3cm]{geometry}

\newtheorem{theorem}{Theorem}\numberwithin{theorem}{section}
\newtheorem{corollary}[theorem]{Corollary}
\newtheorem{lemma}{Lemma}\numberwithin{lemma}{section}
\numberwithin{definition}{section}
\newtheorem{remark}{Remark}\numberwithin{remark}{section}
\newtheorem{example}{Example}\numberwithin{example}{section}
\newtheorem{assumption}{Assumption}\numberwithin{assumption}{section}

\numberwithin{equation}{section}

\def\d{{\rm d}}
\def\<{\langle}
\def\>{\rangle}

\allowdisplaybreaks[4]
\begin{document}

\title{Uniform-in-time error estimates for McKean-Vlasov SDEs with common noise and stochastic algorithms}
%one-dimensional

\author{Yuhang Zhang and Minghui Song*}

\date{}

%==============================================================

\maketitle

\begin{abstract}
In this work, by construct an asymptotic coupling by reflection, we first explore the uniform-in-time estimate on probability distance for two measure-valued processes induced by a McKean-Vlasov SDE with common noise and an interacting particle system, where the drift terms are dissipative merely in the long distance. As direct applications of this estimate, we establish the uniform-in-time error estimates for the numerical solutions derived via backward/tamed/adaptive Euler-Maruyama methods. Moreover, as another direct application, the uniform-in-time conditional propagation of chaos is quantified. 

\end{abstract}

\section{Introduction}
\subsection{Background}
A McKean-Vlasov stochastic differential equation (SDE) for a $d$-dimensional process $X$ is an SDE where the coefficients not only depend on the current state $X_t$, but also depend on the law of $X_t$, i.e.,
\begin{align}\label{MV-SDE001}
{\rm d}X_t=b(X_t,\mathcal{L}(X_t)){\rm d}t+\sigma(X_t,\mathcal{L}(X_t)){\rm d}B_t,\quad X_0=\xi,
\end{align}
where $B=(B_t)_{t\ge 0}$ is a standard Brownian motion, $\mathcal{L}(X_t)$ denotes the marginal law or distribution of the process $X$ at time $t\ge 0$ and $\xi$ is an $\mathbb{R}^d$-valued random variable. The work on McKean-Vlasov SDEs was originated in \cite{MR221595} and have been the target of much research. 
The existence and uniqueness of the solution of McKean-Vlasov SDEs in strong and weak sense has been thoroughly investigated in the literature, see for example, \cite{MR3752669, MR3753660, MR3914550, MR4260494,MR4421344, MR1108185} and references therein. So far, McKean-Vlasov SDEs have been investigated considerably on behaviors in a finite-time horizon (e.g. propagation of chaos (PoC for abbreviation, \cite{MR3752669, Chaintron_2022, MR1431299}) and numerical approximations (\cite{MR3871063, MR4860959, MR4367675, MR4568438, MR4293705})), and long-time asymptotics (e.g. ergodicity (\cite{MR3403022, MR4291453, MR3739509, MR4550214}) and long-time numerical analysis (\cite{MR4908055, yuanping2025explicitnumericalapproximationsmckeanvlasov, WOS:000712217500001, liu2025longtimestrongconvergence, MR4815916})).

%介绍根据混沌传播理论，上述MV-SDEs是IPS的宏观极限，进而引入IPS
Based on the theory of PoC, McKean-Vlasov SDE \eqref{MV-SDE001} is the macroscopic limit of the following interacting particle system (IPS) as the number of particles $N$ goes to infinity, 
 \begin{align}\label{IPS001}
{\rm d}X_t^{i,N}=b(X_t^{i,N},\mu_t^N){\rm d}t+\sigma(X_t^{i,N},\mu_t^N){\rm d}B_t^i,\quad X_0^{i,N}=\xi^i,
\end{align}
where $i\in\mathbb{S}_N:=\{1,\cdots, N\}$, $(B^i,\xi^i)_{i\in\mathbb{S}_N}$ are the independent copies of $(B,\xi)$, $\mu_t^N=\frac{1}{N}\sum_{j=1}^N \delta_{X_t^{j,N}}$, and $\delta_x$ denotes the Dirac measure at point $x$.

%介绍不只受独立噪声影响，还会受共同的环境噪声的影响，进而引入带有共同噪声的交互粒子系统
%再根据条件混沌传播理论引出带有共同噪声的MVSDEs
Note that in some circumstances, the particles in a mean-field model may be influenced not only by their own idiosyncratic noise but also by random shocks that affect all particles simultaneously. In such scenarios, the particle dynamics can no longer be represented by the classical mean-field system, and instead be described by a mean-field system with common noise:
 \begin{align}\label{IPS002}
{\rm d}X_t^{i,N}=b(X_t^{i,N},\mu_t^N){\rm d}t+\sigma(X_t^{i,N},\mu_t^N){\rm d}B_t^i+\overline\sigma(X_t^{i,N},\mu_t^N){\rm d}W_t,\quad X_0^{i,N}=\xi^i,
\end{align}
where $(B^i,\xi^i)_{i\in\mathbb{S}_N}$ are defined exactly as in \eqref{IPS001}, $W=(W_t)_{\ge 0}$ is also a Brownian motion, independent of $B^1,\cdots, B^N$. In \eqref{IPS002}, $(B^i)_{i\in\mathbb{S}_N}$ are also called the idiosyncratic noise (or individual noise) as in \eqref{IPS001}, and $W$ is named as a common noise capturing environmental randomness shared across the particle system. In this setting, all particles are not asymptotically independent any more and the empirical measure no longer converges to a deterministic distribution as the number of particles tends to infinity. However, under appropriate conditions, the phenomenon on conditional PoC (see e.g. \cite[Theorem 2.12]{MR3753660}) indicates that, conditioned on the $\sigma$-algebra generated by the common noise, the particles become asymptotically independent and their empirical distribution converges to the common conditional distribution of each particle. And the limiting dynamics can be characterized by the following McKean-Vlasov SDE with common noise:
\begin{align}\label{CMV-SDE001}
{\rm d}X_t=b(X_t,\mu_t){\rm d}t+\sigma(X_t,\mu_t){\rm d}B_t+\overline\sigma(X_t,\mu_t){\rm d}W_t,\quad X_0=\xi,
\end{align}
where $\mu_t:=\mathcal{L}(X_t|\mathcal{F}_t^{W})$ denotes the conditional distribution of $X_t$ given the $\sigma$-algebra $\mathcal{F}_t^{W}$, $B$ and $W$ are mutually independent Brownian motions. Such equations are also referred to as conditional McKean-Vlasov SDEs  and have been widely applied in stochastic optimal control and mean-field game theory.

%介绍关于带有共同噪声的MVSDEs的现有工作，引出数值方法，进而引出本文的研究内容
In recent years, there are some progresses on qualitative and quantitative analyses for McKean-Vlasov SDEs with common noise—such as well-posedness (\cite{MR4583663, MR4255126, MR4525160}), conditional PoC (\cite{MR4925916, MR3513594, MR4814027, liu2025adaptiveemschemesmckeanvlasov, MR4525160, MR4690564, MR4491513}) and numerical approximations (\cite{MR4728171, MR4497846}), et al. Nevertheless, compared with classical McKean-Vlasov SDEs, the literature on models incorporating common noise remains relatively limited, and most existing studies focus on finite-time behavior, asymptotic analysis in the infinite-time setting remains largely unexplored. Regarding the long-time dynamics of the exact solution, \cite{maillet2025notelongtimebehaviourstochastic} investigated the exponential ergodicity of the measure-valued process $(\mu_t)_{t>0}$ associated with the conditional McKean-Vlasov SDE on $\mathbb{R}$:
\begin{align*}
{\rm d}X_t=-\left(V'(X_t)+\int_{\mathbb{R}}W'(X_t-y)\mu_t({\rm d}y)\right){\rm d}t+\sigma{\rm d}B_t+\sigma_0{\rm d}W_t,
\end{align*}
where $V'$ and $W'$ are set to be globally Lipschitz continuous. Inspired by \cite{maillet2025notelongtimebehaviourstochastic}, \cite{MR4925916} handled the exponential ergodicity of the conditional McKean-Vlasov SDE on $\mathbb{R}$:
\begin{align*}
{\rm d}X_t=b(X_t,\mu_t){\rm d}t+\sigma(X_t){\rm d}B_t+\sigma_0{\rm d}W_t,
\end{align*}
in which the drift $b$ is allowed to be much more general and of polynomial growth, the idiosyncratic noise is permited to be of multiplicative type. While for long-time behavior of numerical solutions, to the best of our knowledge, \cite{liu2025adaptiveemschemesmckeanvlasov} is the only work that establishes the strong convergence of the numerical algorithm (the adaptive Euler-Maruyama method) in infinite horizon. In this work, we will establish the long-time boundedness and convergence of the numerical solutions derived by backward/tamed/adaptive Euler-Maruyama methods, respectively. The purpose and framework of this study are presented in the following subsection.

\subsection{Framework and motivation}
In this paper, to distinguish two sources of randomness, we introduce the following two probability spaces. Let $(\Omega^B,\mathcal{F}^B,\mathbb{P}^B)$ and $(\Omega^W,\mathcal{F}^W,\mathbb{P}^W)$ be two complete probability space with filtrations $\left(\mathcal{F}_t^B\right)_{t\ge 0}$ and $\left(\mathcal{F}_t^W\right)_{t\ge 0}$ satisfying the usual conditions, respectively. Define
\begin{align*}
\Omega=\Omega^B\times\Omega^W, ~\mathbb{P}= \mathbb{P}^B\times\mathbb{P}^W,
\end{align*}
and $\mathcal{F}$ stands for the completion of $ \mathcal{F}^B\times\mathcal{F}^W$, $\mathcal{F}_t$ stands for the completion of $ \mathcal{F}_t^B\times\mathcal{F}_t^W$ for $t\ge 0$. Element $\omega=(\omega^B,\omega^W)$ denotes a generic element in $\Omega$ with $\omega^B\in\Omega^B$ and $\omega^W\in\Omega^W$. $\mathbb{E}$, $\mathbb{E}^B$ and $\mathbb{E}^W$ stands for taking expectation with respect to $\mathbb{P}$, $\mathbb{P}^B$, and $\mathbb{P}^W$ respectively. Given a random variable $X$ on $\Omega$, $\mathcal{L}(X)$ denotes the law of $X$, $\mathcal{L}(X\vert \mathcal{F}^W)$ denotes the conditional law of $X$ given $\mathcal{F}^W$. 

In this work, we intend to extend the framework proposed for McKean-Vlasov SDEs in \cite{MR4908055} to McKean-Vlasov SDEs with common noise. For $i\in\mathbb{S}_N$, let \( B^i = (B_t^i)_{t \ge 0} \) and \( \overline{B}^i = (\overline{B}_t^i)_{t \ge 0} \) be mutually independent Brownian motions of dimensions \( d \) and \( m_1 \), respectively, defined on the filtered probability space $(\Omega^B,\mathcal{F}^B, (\mathcal{F}_t^B)_{t\ge 0}, \mathbb{P}^B)$. Let $W:=(W_t)_{t\ge 0}$ and $\overline{W}:=(\overline{W}_t)_{t\ge 0}$ be mutually independent Brownian motions of dimensions \( d \) and \( m_2 \), respectively, defined on $(\Omega^W,\mathcal{F}^W, (\mathcal{F}_t^W)_{t\ge 0}, \mathbb{P}^W)$. $(B^i)_{i\in\mathbb{S}_N}$, $(\overline B^i)_{i\in\mathbb{S}_N}$, $W$ and $\overline{W}$ are mutually independent. Consider the following stochastic particle systems: for $i\in\mathbb{S}_N$, $t\ge 0$,
\begin{align}\label{NIPS}
{\rm d}X_t^i=b(X_t^i,\mu_t^i){\rm d}t+\sigma_0{\rm d}B_t^i+\sigma(X_t^i,\mu_t^i){\rm d}\overline B_t^i+\sigma_1{\rm d}W_t+\overline\sigma(X_t^i,\mu_t^i){\rm d}\overline{W}_t
\end{align}
with the initial data $X_0^i$, and
\begin{align}\label{IPS}
{\rm d}X_t^{i,N}=\widetilde{b}(X_{\theta_t}^{i,N},\widetilde{\mu}_{\overline{\theta}_t}^N){\rm d}t+\sigma_0{\rm d}B_t^i+\sigma(X_{\overline{\theta}_t}^{i,N},\widetilde{\mu}_{\overline{\theta}_t}^N){\rm d}\overline B_t^i+\sigma_1{\rm d}W_t+\overline\sigma(X_{\overline{\theta}_t}^{i,N},\widetilde{\mu}_{\overline{\theta}_t}^N){\rm d}\overline{W}_t
\end{align}
with the initial data $X_0^{i,N}$, we aim to establish the uniform-in-time quantitative estimate on probability distance for two measure-valued processes $\mu_t^i$ and $\nu_t^{i,N}$:
\begin{align}\label{main aim}
\mathcal{W}_1(\mathcal{L}(\mu_t^i),\mathcal{L}(\nu_t^{i,N}))\le \varphi(t,N),~t\ge 0,~i\in\mathbb{S}_N,
\end{align}
where $\mu_t^i=\mathcal{L}(X_t^i\vert \mathcal{F}_t^W)$, $\nu_t^{i,N}=\mathcal{L}(X_t^{i,N}\vert \mathcal{F}_t^W)$, $\varphi:[0,\infty)\times \mathbb{N}_+\to (0,\infty)$. In \eqref{NIPS} and \eqref{IPS}, $b, \widetilde{b}:\mathbb{R}^d\times\mathcal{P}(\mathbb{R}^d)\to \mathbb{R}^d$, $\sigma:\mathbb{R}^d\times\mathcal{P}(\mathbb{R}^d)\to \mathbb{R}^{d\times m_1}$, $\overline\sigma:\mathbb{R}^d\times\mathcal{P}(\mathbb{R}^d)\to \mathbb{R}^{d\times m_2}$, $\sigma_0, \sigma_1\in \mathbb{R}$, $\theta, \overline\theta: [0,\infty)\to [0,\infty)$. 
 $\widetilde{\mu}_t^N=\frac{1}{N}\sum_{j=1}^N\delta_{X_t^{j,N}}$. $(B^i)_{i\in\mathbb{S}_N}$ and $(\overline B^i)_{i\in\mathbb{S}_N}$  are referred to as idiosyncratic noise, $W$ and $\overline{W}$ are referred to as common noise.

As described in \cite{MR4908055}, if we focus on the framework \eqref{NIPS} and \eqref{IPS}, and explore the uniform-in-time estimate \eqref{main aim}, we can get the following result as the immediate by-products of the quantitative estimate.
\begin{itemize}
\item[(1)]{\textit{Uniform-in-time conditional propagation of chaos}}

In \eqref{IPS}, once we take $\widetilde{b}(x,\mu)=b(x,\mu)$ and $\theta_t=\overline\theta_t=t$, \eqref{IPS} subsequently becomes
\begin{align}\label{IPS-22}
{\rm d}X_t^{i,N}=b(X_{t}^{i,N},\widetilde{\mu}_{t}^N){\rm d}t+\sigma_0{\rm d}B_t^i+\sigma(X_{t}^{i,N},\widetilde{\mu}_{t}^N){\rm d}\overline B_t^i+\sigma_1{\rm d}W_t+\overline\sigma(X_{t}^{i,N},\widetilde{\mu}_{t}^N){\rm d}\overline{W}_t, ~\forall i\in\mathbb{S}_N.
\end{align}
Hence, \eqref{NIPS} and \eqref{IPS-22} are the respective non-interacting particle system and interacting
particle system corresponding to the following McKean-Vlasov SDE with common noise:
\begin{align*}
{\rm d}X_t=b(X_t,\mu_t){\rm d}t+\sigma_0{\rm d}B_t+\sigma(X_t,\mu_t){\rm d}\overline B_t+\sigma_1{\rm d}W_t+\overline\sigma(X_t,\mu_t){\rm d}\overline{W}_t,
\end{align*}
where $\mu_t=\mathcal{L}(X_t\vert \mathcal{F}_t^W)$, $(B_t)_{t\ge 0}$(a copy of $(B_t^i)_{t\ge 0}$ for $i\in\mathbb{S}_N$) and $(\overline B_t)_{t\ge 0}$(a copy of $(\overline B_t^i)_{t\ge 0}$ for $i\in\mathbb{S}_N$) are mutually independent Brownian motions, and are all independent of the Brownian motions $W$ and $\overline{W}$. PoC is a hot topic in the McKean-Vlasov frame, the convergence is typically discussed in the sense of Wasserstein distance, total variation, or relative entropy. We only introduce and explore the results in Wasserstein metric in this paper. The theory on conditional PoC in a finite time horizon has been investigated in many recent works, for instance, \cite{MR3513594, MR4814027, MR4525160, MR4690564, MR4491513}. In \cite{MR3513594}, the authors proved conditional PoC when $\sigma_0=\sigma= 0$, $b(x,\mu)=\mu(f(x-\cdot))$ for some Lipschitz function $f$. \cite{MR4814027} established the quantitative conditional PoC for the McKean-Vlasov SDEs with additive common noise. In \cite{MR4491513}, the quantitative conditional PoC is studied for stochastic spatial epidemic model, where the evolution of infection states are driven by the Poisson point processes and the displacement of individuals contains a common noise. In addition, there are also some progresses on the issue of uniform-in-time conditional PoC, see, for example, \cite{MR4925916, liu2025adaptiveemschemesmckeanvlasov}. As an immediate by-product of the quantitative estimate \eqref{main aim}, the uniform-in-time conditional PoC will be reproduced right away, where the underlying drift is dissipative merely in the long distance.

%\cite{MR4235471} investigated conditional PoC for one dimensional SDEs driven by Poisson random measure and common Brownian motion noise. %无穷时间，没有阶
%进一步解释

\item[(2)]{\textit{Uniform-in-time error estimates for stochastic algorithms}}

As another direct application of the estimate \eqref{main aim}, the uniform-in-time error estimates of numerical solutions for the McKean-Vlasov SDE \eqref{NIPS} will be established. As is known to all, the Euler-Maruyama (EM for brevity) scheme works merely for SDEs with coefficients of linear growth. In this work, we deal with equations where the drift term exhibits dissipative property with respect to the state component and may grow superlinearly. Hence, 
several variants of the EM scheme are employed, all of which can be encompassed by \eqref{IPS}.
\begin{itemize}
\item[(i)] The backward EM scheme related to the McKean-Vlasov SDE \eqref{NIPS}: for a step size $h>0$,
\begin{align*}
{\rm d}X_t^{i,N}=b(X_{t_h+h}^{i,N},\widetilde{\mu}_{t_h}^{N}){\rm d}t+\sigma_0{\rm d}B_t^i+\sigma(X_{t_h}^{i,N},\widetilde{\mu}_{t_h}^{N}){\rm d}{\overline B}_t^i+\sigma_1{\rm d}W_t+\overline\sigma(X_{t_h}^{i,N},\widetilde{\mu}_{t_h}^{N}){\rm d}\overline{W}_t,
\end{align*}
is included in \eqref{IPS} by taking $\widetilde{b}(x,\mu)=b(x,\mu)$, $\theta_t=t_h+h, \overline\theta_t=t_h$, in which $t_h=\lfloor t/h \rfloor h$ with $\lfloor t/h \rfloor$ being the integer part of $t/h$.
\item[(ii)] The tamed EM method, constructed in \cite{MR2985171,MR3070913}, can also efficiently approximate SDEs with coefficients exhibiting superlinear growth in a finite time horizon. But, for the step size $h>0$, since the modified drifts proposed in \cite{MR2985171,MR3070913} are uniformly bounded, the schemes are not adequate to approximate the exact solution in an infinite-time horizon. To derive a uniform-in-time estimate between the distributions of the exact solution and the numerical counterpart, we adopt the following tamed EM scheme to solve the McKean-Vlasov SDE \eqref{NIPS}:
  \begin{align*}%\label{BEM}
{\rm d}X_t^{i,N}=\frac{b(X_{t_h}^{i,N},\widetilde{\mu}_{t_h}^{N})}{1+\alpha\sqrt{h}|X_{t_h}^{i,N}|^l}{\rm d}t+\sigma_0{\rm d}B_t^i+\sigma(X_{t_h}^{i,N},\widetilde{\mu}_{t_h}^{N}){\rm d}{\overline B}_t^i+\sigma_1{\rm d}W_t+\overline\sigma(X_{t_h}^{i,N},\widetilde{\mu}_{t_h}^{N}){\rm d}\overline{W}_t,
\end{align*}
in which $\alpha\in [0,1]$, the drift term $b$ satisfies Assumption \ref{S1} below, $l$ is also given in Assumption \ref{S1}. This scheme can also be incorporated into \eqref{IPS} by taking $\widetilde{b}(x,\mu)=\frac{b(x,\mu)}{1+\alpha\sqrt{h}|x|^l}$, $\theta_t=\overline\theta_t=t_h$. Moreover, if $l=0$, i.e., the drift coefficient $b$ is of linear growth, then we can take $\alpha=0$, and the scheme above will degenerate to the standard EM method.
\item[(iii)] The modified EM scheme with an adaptive step size is also considered. In the spirit of \cite{MR3828142}, a continuous adaptive EM scheme associated with \eqref{NIPS} can be described as below:
\begin{align*}
{\rm d}X_t^{i,N}=b(X_{\underline t}^{i,N},\widetilde{\mu}_{\underline t}^{N}){\rm d}t+\sigma_0{\rm d}B_t^i+\sigma(X_{\underline t}^{i,N},\widetilde{\mu}_{\underline t}^{N}){\rm d}\overline B_t^i+\sigma_1{\rm d}W_t+\overline\sigma(X_{\underline t}^{i,N},\widetilde{\mu}_{\underline t}^{N}){\rm d}\overline{W}_t,
\end{align*}
where $\underline t=\max\{t_n, n\in\mathbb{N}:t_n\le t\}$, $t_{n+1}=t_n+h_n$ with $h_n$ is an adaptive time step-size (see section \ref{Adaptive EM method} for more details). Then \eqref{IPS} can cover this scheme once we set $\widetilde{b}(x,\mu)=b(x,\mu)$ and $\theta_t=\overline \theta_t=\underline t$.
\end{itemize}
\end{itemize}
 
 {\bf{\textit{Our contributions.}}}
 \begin{itemize}
\item In this work, we extend the framework proposed for standard McKean-Vlasov SDEs in \cite{MR4908055} to McKean-Vlasov SDEs with common noise. Different from the McKean-Vlasov frame, the conditional distribution with respect to the common noise is a function of common noise so that we have to overcome difficulties produced by this crucial difference. For instance, in the procedure of constructing coupling processes, the conditional distribution with respect to the common noise is typically treated as a known function of the common noise so that the common noise need also be fixed. Hence, only new idiosyncratic noise can be constructed in coupling process. Moreover, since the idiosyncratic noise and the common noise are independent, when we calculate the expectation, we should firstly take conditional expectation with respect to the common noise in which the common noise can be viewed as a constant and then use the tower property of conditional expectation to realize this goal.
\item We establish the long-time boundedness and convergence of the numerical solutions derived by backward/tamed/adaptive Euler-Maruyama methods, respectively. If both measure-valued processes associated with the conditional McKean-Vlasov SDEs and its numerical methods possess unique invariant measures, the convergence results reveal the quantitative estimate between the invariant probability measure and its numerical version for the McKean-Vlasov SDEs with common noise with partially dissipative drifts. Moreover, \cite{MR4908055} only established uniform error estimates of numerical algorithms for McKean-Vlasov SDEs with additive noise. The model in the present paper is more general than that in \cite{MR4908055}, even in the absence of common noise.
\end{itemize}

{\bf{\textit{Organization.}}}
The rest of the paper is organized as follows. The last part of the current section is a summary of the notation used throughout the paper. In section \ref{Main estimate}, we focus on the framework \eqref{NIPS} and \eqref{IPS}, and explore the uniform-in-time estimate \eqref{main aim} via constructing an asymptotic coupling by reflection. In section \ref{stochastic algorithms}, we establish the uniform-in-time error bounds for the numerical solutions derived by backward/tamed/adaptive Euler-Maruyama methods, respectively. Finally, we numerical test long-time errors between a McKean-Vlasov SDEs with the drift term is dissipative merely in the long distance and its corresponsing interacting particle systems in Section \ref{Numerical example}.

\subsection{Notations}

Let $\mathbb{N}_+=\{1, 2,\cdots\}$ be the set of natural numbers starting at $1$. For $N\in \mathbb{N}_+$, let $\mathbb{S}_N:=\{1, \cdots, N\}$.
Let $\langle x, y\rangle$ denote the inner product of vectors $x, y$.
If $A$ is a vector or matrix, let $A^{\rm T}$ denote its transpose, $\vert A\vert$ denote its Hilbert-Schmidt norm, i.e., $\vert A\vert = \sqrt{{\rm trace}(A^{\rm T} A)}$. For $a,b\in\mathbb{R}$, set $a\vee b=\max\{a,b\}$, $a\wedge b=\min\{a,b\}$.  

Given the metric space $(\mathbb{R}^d, \vert \cdot\vert )$, we use $\mathcal{P}(\mathbb{R}^d)$ to denote the family of all probability measures on $(\mathbb{R}^d, \mathcal{B}(\mathbb{R}^d))$, where $\mathcal{B}(\mathbb{R}^d)$ denotes the Borel $\sigma$-field over $\mathbb{R}^d$, and define the subset of probability measures with finite $p$-th moment by 
\begin{align*}
\mathcal{P}_p(\mathbb{R}^d):=\left\{\mu\in\mathcal{P}(\mathbb{R}^d)\Big\vert \mu(|\cdot|^p):=\int_{\mathbb{R}^d} \vert x\vert^p\mu(dx)<\infty\right\}.
\end{align*}
It is well-known that $\mathcal{P}_p(\mathbb{R}^d), p\ge1$ is a Polish space under the $L^p$-Wasserstein metric given by
\begin{align*}
\mathbb{W}_p(\mu,\nu):=\inf_{\pi\in\mathcal{C}(\mu,\nu)}\left(\int_{\mathbb{R}^d\times \mathbb{R}^d}  \vert x-y\vert^p\pi(dx,dy)\right)^{1/p},~\mu, \nu \in\mathcal{P}_p(\mathbb{R}^d),
\end{align*}
where $\mathcal{C}(\mu,\nu)$ means all the couplings of $\mu$ and $\nu$. 
Moreover, since we consider the conditional McKean-Vlasov SDEs in this work, we further introduce following notations. Let $p\ge 1$, 
\begin{align*}
L_p(\mathcal{P}(\mathbb{R}^d)):=\left\{\nu\in\mathcal{P}(\mathcal{P}(\mathbb{R}^d))\Big\vert \int_{\mathbb{R}^d} \mu(|\cdot|^p)\nu(d\mu)<\infty\right\},
\end{align*}
and
\begin{align*}
\mathcal{W}_p(\mu,\nu):=\inf_{\pi\in\mathcal{C}(\mu,\nu)}\int_{\mathcal{P}(\mathbb{R}^d)\times \mathcal{P}(\mathbb{R}^d)}  \mathbb{W}_p(\overline\mu,\overline\nu)\pi(d\overline\mu,d\overline\nu), ~\mu, \nu \in L_p(\mathcal{P}(\mathbb{R}^d)).
\end{align*}

\section{Main Estimate}\label{Main estimate}
In the present work, for the sake of simplicity, we assume that the initial data $(X_0^i, X_0^{i,N})_{i\in\mathbb{S}_N}$ are supported by $(\Omega^B,\mathcal{F}^B, (\mathcal{F}_t^B)_{t\ge 0}, \mathbb{P}^B)$.
%$X_0^{i,N}:=(X_t^{i,N})_{t\in[-r_0,0]}$ for some $r_0\ge 0$, where, $X_0^{i,N}$ is a $\mathcal{C}:=C([-r_0,0],\mathbb{R}^d)$-valued random variables. 

\begin{assumption}\label{A1}
$b(\cdot, \delta_0):\mathbb{R}^d\to\mathbb{R}^d$ is local Lipschitz continuous, and there exist constants $\lambda_1>0$ and $l_0\ge 1$ such that for all $x,y\in\mathbb{R}^d$ and $\mu\in\mathcal{P}_1(\mathbb{R}^d)$,
\begin{align*}
\<x-y, b(x,\mu)-b(y,\mu)\>\le \vert x-y\vert\phi(|x-y|)\mathbbm{1}_{\{|x-y|\le l_0\}}-\lambda_1|x-y|^2\mathbbm{1}_{\{|x-y|> l_0\}},
\end{align*}
where $\phi:[0,\infty)\to [0,\infty)$ with $\phi(0)=0$ is increasing and continuous. Moreover, there exists a constant $\lambda_2>0$ such that for all $x\in\mathbb{R}^d$ and $\mu,\nu\in\mathcal{P}_1(\mathbb{R}^d)$,
\begin{align*}
|b(x,\mu)-b(x,\nu)|\le \lambda_2\mathbb{W}_1(\mu,\nu).
\end{align*}
\end{assumption}

\begin{assumption}\label{A2}
There exists a constant $L>0$ such that for all $x,y\in\mathbb{R}^d$ and $\mu,\nu\in\mathcal{P}_1(\mathbb{R}^d)$,
\begin{align*}
|\sigma(x,\mu)-\sigma(y,\nu)|^2\vee |\overline\sigma(x,\mu)-\overline\sigma(y,\nu)|^2\le L(|x-y|^2+\mathbb{W}_1^2(\mu,\nu)).
\end{align*}
\end{assumption}

\begin{remark}\label{Remark-1.1}
Based on Assumptions \ref{A1} and \ref{A2}, it follows that for all $x\in\mathbb{R}^d$ and $\mu\in\mathcal{P}_1(\mathbb{R}^d)$, 
 \begin{align}\label{B1}
 \<x,b(x,\mu)\>&\le \vert x\vert\phi(|x|)\mathbbm{1}_{\{|x|\le l_0\}}-\lambda_1|x|^2\mathbbm{1}_{\{|x|> l_0\}}+\lambda_2|x|\mathbb{W}_1(\mu,\delta_0)+|x||b(0,\delta_0)|\notag\\
 &\le -K_1|x|^2+\frac{\lambda_2}{2}\mathbb{W}_1^2(\mu,\delta_0)+K_2,
 \end{align}
where $K_1=\frac{2\lambda_1-3\lambda_2}{2}$, $K_2=l_0\phi(l_0)+\lambda_1 l_0^2+\frac{|b(0,\delta_0)|^2}{4\lambda_2}$, and 
 \begin{align}
 |\sigma(x,\mu)|^2&\le 2L(|x|^2+\mathbb{W}_1^2(\mu,\delta_0))+2|\sigma(0,\delta_0)|^2,\label{B2}\\
|\overline\sigma(x,\mu)|^2&\le 2L(|x|^2+\mathbb{W}_1^2(\mu,\delta_0))+2|\overline\sigma(0,\delta_0)|^2.\label{B3}
 \end{align}
 \end{remark}
%\begin{assumption}\label{A3}
%There exists a constant $\overline L>0$ such that
%\begin{align*}
%|b(0,\delta_0)|\vee |\sigma(0,\delta_0)|\vee |\overline\sigma(0,\delta_0)|\le \overline L.
%\end{align*}
%\end{assumption}

Under Assumptions \ref{A1} and \ref{A2}, there exists a unique strong solution of \eqref{NIPS}, and the uniform-in-time estimate on probability distance for two measure-valued processes induced by the McKean-Vlasov SDE with common noise \eqref{NIPS} and the interacting particle system \eqref{IPS} can be portrayed precisely stated as the following theorem.
\begin{theorem}\label{main theorem}
Suppose that Assumptions \ref{A1} and \ref{A2} hold with $\lambda_1>2\lambda_2$ and that $\sigma_0\sigma_1\neq 0$. Then there exist constants $\lambda_2^*,L^*,C,\lambda_0>0$ such that for any $\lambda_2\in [0,\lambda_2^*), L\in[0,L^*]$,
\begin{align*}
\mathcal{W}_1\big(\mathcal{L}(\mu_t^i),\mathcal{L}(\nu_t^{i,N})\big)\le C\left({\rm e}^{-\lambda_0t}\mathbb{W}_1(\mu,\nu)+\psi(N)+\int_0^t{\rm e}^{-\lambda_0(t-s)}\mathbb{E}\vert b(X_s^{i,N},\widetilde{\mu}_s^N)-\widetilde{b}(X_{\theta_s}^{i,N},\widetilde{\mu}_{\overline{\theta}_s}^N)\vert{\rm d}s\right),
\end{align*}
in the case of $\mu\in\mathcal{P}_q(\mathbb{R}^d)$ for some $q>1$ and $\nu\in\mathcal{P}_1(\mathbb{R}^d)$, where $t>0, i\in\mathbb{S}_N$, $\mu_t^i=\mathcal{L}(X_t^i\vert \mathcal{F}_t^W)$ and $\nu_t^{i,N}=\mathcal{L}(X_t^{i,N}\vert \mathcal{F}_t^W)$ stand for the regular conditional distributions of $X_t^i$ and $X_t^{i,N}$ with initial distributions $\mathcal{L}(X_0^i)=\mu$ and $\mathcal{L}(X_0^{i,N})=\nu$, respectively,
\begin{align*}
\psi(N)=\begin{cases}
      N^{-1/2}+N^{-(q-1)/q}, &{\rm if}~~d<2~~{\rm and}~~q\neq 2,\\
       N^{-1/2}{\rm log}(1+N)+N^{-(q-1)/q}, &{\rm if}~~d=2~~{\rm and}~~q\neq 2\\
       N^{-1/d}+N^{-(q-1)/q}, &{\rm if}~~d>2~~{\rm and}~~q\neq d/(d-1).
\end{cases}
\end{align*}

% $\nu_0({\rm d}x):=\nu(\{\eta\in\mathcal{C}\}:\eta_0\in{\rm d}x)$.
\end{theorem}

The proof will be given in section \ref{proof of the key estimate}, we first prepare a lemma and present a corollary.

%%%========================================================================

\begin{lemma}\label{CPoC}
Consider the McKean-Vlasov SDE \eqref{NIPS} with $\mathbb{E}|X_0^i|^q<\infty$ for all $i\in\mathbb{S}_N$ and some $q>1$. Let $\overline{\mu}_t^{N}=\frac{1}{N}\sum_{j=1}^N\delta_{X_t^j}$. Suppose that Assumptions \ref{A1} and \ref{A2} hold with $\lambda_1>2\lambda_2$. Then, there exist constants $L^*, C>0$ such that for any $L\in[0,L^*]$, $t>0$, and $i\in\mathbb{S}_N$,
\begin{align}\label{W_1_le_CN}
\mathbb{E}\mathbb{W}_1(\mu_t^i, \overline{\mu}_t^{N})\le C\left(1+(\mathbb{E}|X_0^i|^q)^{1/q}\right)\psi(N).
\end{align}
\end{lemma}
\begin{proof}
According to \cite[Theorem 1]{Fournier2014}, there exists a constant $C>0$ depending on $q$ and $d$ such that
\begin{align}\label{W_1_mu}
\mathbb{E}\mathbb{W}_1(\mu_t^i, \overline{\mu}_t^{N})\le C\left(\mathbb{E}|X_t^i|^q\right)^{1/q}\psi(N).
\end{align}
Next, we verify that there exists a constant $C_0>0$ such that for all $i\in\mathbb{S}_N$ and all $t>0$,
\begin{align}\label{X_t_1}
\mathbb{E}|X_t^i|^q\le \mathbb{E}|X_0^i|^q+C_0.
\end{align}

Applying It\^o's formula to the function ${\rm e}^{\gamma_0t}|X_t^i|^q$ for $q>2$ with 
$$\gamma_0=q\left(K_1-\frac{\lambda_2}{2}-4qL\right),$$%=q\left(\lambda_1-2\lambda_2-4qL\right)
where $K_1=\frac{1}{2}(2\lambda_1-3\lambda_2)$ was defined in Remark \ref{Remark-1.1}. Note that since $\lambda_1-2\lambda_2>0$, we can choose $L^*>0$ sufficiently small such that $\lambda_1-2\lambda_2-4qL^*>0$, then for any $L\in[0,L^*]$, it holds $\gamma_0>0$. According to Remark \ref{Remark-1.1} and the Young inequality, we deduce that
\begin{align*}
\mathbb{E}\left({\rm e}^{\gamma_0t}|X_t^i|^q\right)\le &\mathbb{E}|X_0^i|^q+\mathbb{E}\int_0^t{\rm e}^{\gamma_0s}\Bigg\{ \gamma_0|X_s^i|^q+q|X_s^i|^{q-2}\<X_s^i, b(X_s^i,\mu_s^i)\>\\
&\quad\qquad\qquad\qquad\qquad+\frac{q(q-1)}{2}|X_s^i|^{q-2}\left(d\sigma_0^2+d\sigma_1^2+|\sigma(X_s^i,\mu_s^i)|^2+|\overline\sigma(X_s^i,\mu_s^i)|^2\right)\Bigg\}{\rm d}s\notag\\
%\le &\mathbb{E}|X_0^i|^q+\mathbb{E}\int_0^t{\rm e}^{\lambda_0s}\Bigg\{ \lambda_0|X_s^i|^q+q|X_s^i|^{q-2}\bigg(-K_1|X_s^i|^2+\frac{\lambda_3}{2}\mathbb{W}_1^2(\mu_s^i,\delta_0)+K_2\\
%&\qquad\qquad\qquad\qquad\qquad\qquad+\frac{q-1}{2}\left(\sigma_0^2+(K_4+K_5)(1+|X_s^i|^2+\mathbb{W}_1^2(\mu_s^i,\delta_0))\right)\bigg)\Bigg\}{\rm d}s\notag\\
\le &\mathbb{E}|X_0^i|^q+\mathbb{E}\int_0^t{\rm e}^{\gamma_0s}\Bigg\{ \left(\gamma_0-qK_1+2q(q-1)L\right)|X_s^i|^q\\
&\quad\qquad\qquad\qquad\qquad+q\left(\lambda_2/2+2(q-1)L\right)|X_s^i|^{q-2}\mathbb{W}_1^2(\mu_s^i,\delta_0)\notag\\
&\quad\qquad\qquad\qquad\qquad+q\left(K_2+\frac{q-1}{2}\left(d\sigma_0^2+d\sigma_1^2+2|\sigma(0,\delta_0)|^2+2|\overline\sigma(0,\delta_0)|^2\right)\right)|X_s^i|^{q-2}\Bigg\}{\rm d}s\notag\\
\le &\mathbb{E}|X_0^i|^q+\int_0^t{\rm e}^{\gamma_0s} \Big(\gamma_0-qK_1+q\lambda_2/2+4q^2L\Big)\mathbb{E}|X_s^i|^q{\rm d}s+C_0\int_0^t{\rm e}^{\gamma_0s}{\rm d}s\notag\\
<&\mathbb{E}|X_0^i|^q+C_0{\rm e}^{\gamma_0t},
\end{align*}
where the constant $C_0=2\left(K_2+\frac{q-1}{2}\left(d\sigma_0^2+d\sigma_1^2+2|\sigma(0,\delta_0)|^2+2|\overline\sigma(0,\delta_0)|^2\right)\right)^{q/2}(4L)^{-(q-2)/2}$. Hence, 
\begin{align*}
\mathbb{E}|X_t^i|^q\le \mathbb{E}|X_0^i|^q{\rm e}^{-\gamma_0t}+C_0\le \mathbb{E}|X_0^i|^q+C_0, ~\forall t\ge 0,~ i\in\mathbb{S}_N.
\end{align*}
Combining \eqref{W_1_mu} and \eqref{X_t_1}, we can derive \eqref{W_1_le_CN}.
\end{proof}

%=====================================================
As an immediate by-product of the Theorem \ref{main theorem}, the uniform-in-time conditional PoC will be reproduced right away.
\begin{corollary}[Conditional propagation of chaos]
Let $X_0^{i,N}=X_0^i$ for all $i\in\mathbb{S}_N$. Suppose that Assumptions \ref{A1} and \ref{A2} hold and that $\sigma_0\sigma_1\neq 0$. Then there exists a constant $C>0$ such that for any $t>0$ and $i\in\mathbb{S}_N$, 
\begin{align*}
\mathcal{W}_1\big(\mathcal{L}(\mu_t^i),\mathcal{L}(\nu_t^{i,N})\big)\le C\psi(N),
\end{align*}
in the case of $\mathbb{E}|X_0^i|^q<\infty$ for some $q>1$, where $\mu_t^i=\mathcal{L}(X_t^i\vert \mathcal{F}_t^W)$ stands for the regular conditional distribution of $X_t^i$, determined by \eqref{NIPS}, $\nu_t^{i,N}=\mathcal{L}(X_t^{i,N}\vert \mathcal{F}_t^W)$ stands for the conditional distribution of $X_t^{i,N}$, determined by \eqref{IPS} with $\widetilde{b}(x,\mu)=b(x,\mu), \theta_t=\overline\theta_t=t$.
\end{corollary}
%\subsection{Conditional propagation of chaos}
%{\color{red}The qualitative convergence and quantitative convergence rate of the conditional propagation of chaos(PoC) has been studied in many recent works,

 %By invoking \cite[Theorem 1]{Fournier2014}, the convergence rate of conditional propagation of chaos can be established once the initial distribution enjoys the high-order moment.}
 
\begin{remark}[Sharp rate for conditional propagation of chaos]%关于PoC最优收敛阶
Note that in the present case, the coefficients are only assumed to be Lipschitz continuous in $\mathbb{W}_1$-distance with respect to the measure variable, so that \cite[Theorem 1]{Fournier2014} for $p=1$ is used to estimate the convergence rate of conditional propagation of chaos, which depends on the dimension $d$ and seems a little complicated. However, if we use the following assumption replacing the Lipschitz continuity in $\mathbb{W}_1$-distance, the convergence rate can be improved to be $\mathcal{O}(N^{-1/2})$. The assumption is: There exists a constant $L>0$ and a Lipschitz continuous function $\phi:\mathbb{R}^d\to\mathbb{R}^d$ with the Lipschitz constant $L_\phi\le 1$ such that for all $x\in\mathbb{R}^d$ and $\mu,\nu\in\mathcal P(\mathbb{R}^d)$.
{\setlength\abovedisplayskip{6pt}
\setlength\belowdisplayskip{6pt}
  \begin{align*}
|(b,\sigma,\overline\sigma)(x,\mu)-(b,\sigma,\overline\sigma)(x,\nu)|\le Ld_{\phi}(\mu,\nu),
\end{align*}}
where $d_\phi(\mu,\nu):=\left\vert \int_{\mathbb{R}^d} \phi(x) \mu({\rm d}x)-\int_{\mathbb{R}^d} \phi(x) \nu({\rm d}x)\right\vert\le\mathbb{W}_1(\mu,\nu)$. The proof of this result makes use of the conditional Rosenthal inequality (see \cite{MR4914659,soni2025tamedeulerapproximationfully,MR4988090}).

\end{remark}
%%==================================
\subsection{Asymptotic coupling by reflection}
The proof of Theorem \ref{main theorem} relies on the coupling process we present below. In order to describe the asymptotic coupling by reflection, we need to introduce some notations. For $\varepsilon> 0$, define the cut-off function $h_{\varepsilon}$ by
\begin{align*}
h_{\varepsilon}(r)=\begin{cases}
      0,&0\le r\le \varepsilon, \\
      1-{\rm exp}\left(\frac{r-\varepsilon}{r-2\varepsilon}\right),& \varepsilon< r< 2\varepsilon, \\
     1,  & r\ge 2\varepsilon,
\end{cases}
\end{align*}
and $h_{\varepsilon}^*(r):=\sqrt{1-h_{\varepsilon}^2(r)}, r\ge 0$. Set 
\[\Pi(x):=I_{d\times d}-2\mathbf{e}(x)\otimes\mathbf{e}(x),\quad x\in\mathbb{R}^d,\]
where $I_{d\times d}$ is the $d\times d$ identity matrix, 
\[\mathbf{e}(x):=\frac{x}{|x|}\mathbbm{1}_{\{x\neq 0\}}+(1, 0,\dots, 0)^{\rm T}\mathbbm{1}_{\{x=0\}}\]
and $\mathbf{e}(x)\otimes\mathbf{e}(x)$ means the tensor between $\mathbf{e}(x)$ and $\mathbf{e}(x)$. 
%{\color{red}For $x\in\mathbb{R}^N$, define $$\|x\|_1=\frac{1}{N}\sum_{j=1}^N |x^j|.$$}
Furthermore, we assume that $B^{1,1}:=(B^{1,1}_t)_{t\ge 0},\cdots,B^{1,N}:=(B^{1,N}_t)_{t\ge 0}$ (resp. $B^{2,1}:=(B^{2,1}_t)_{t\ge 0}, \cdots$, $B^{2,N}:=(B^{2,N}_t)_{t\ge 0}$) are mutually independent $d$-dimensional Brownian motions carried on the same probability space as that of $B^1,\cdots, B^N$, $W^{1}:=(W^{1}_t)_{t\ge 0}$ and $W^{2}:=(W^{2}_t)_{t\ge 0}$ are mutually independent $d$-dimensional Brownian motions carried on the same probability space as that of $W$. Assume also that $(B^{1,i})_{i\in\mathbb{S}_N}$, $(B^{2,i})_{i\in\mathbb{S}_N}$,  $(\overline{B}^i)_{i\in\mathbb{S}_N}$, $W^{1}$, $W^{2}$ and $\overline{W}$ are mutually independent. 

With the proceeding notations at hand, we can construct the asymptotic coupling by reflection associated with \eqref{NIPS} and \eqref{IPS},
\begin{align}\label{coupling process}
\begin{cases}
\d  Y_t^{i,\varepsilon}=b(Y_t^{i,\varepsilon},\widehat{\mu}_t^{i,\varepsilon}) \d t+ h_\varepsilon(|Z_t^{i,N,\varepsilon}|)(\sigma_0{\rm d}B_t^{1,i}+\sigma_1{\rm d}W_t^{1})+h_\varepsilon^*(|Z_t^{i,N,\varepsilon}|)(\sigma_0{\rm d}B_t^{2,i}+\sigma_1{\rm d}W_t^{2})\\
~\qquad\qquad+ \sigma(Y_t^{i,\varepsilon},\widehat{\mu}_t^{i,\varepsilon}){\rm d}\overline B_t^i +\overline\sigma(Y_t^{i,\varepsilon},\widehat{\mu}_t^{i,\varepsilon}){\rm d}\overline{W}_t,\\
\d   Y_t^{i,N,\varepsilon}=\widetilde{b}(Y_{\theta_t}^{i,N,\varepsilon},\widetilde{\mu}_{\overline{\theta}_t}^{N,\varepsilon}){\rm d}t+h_\varepsilon(|Z_t^{i,N,\varepsilon}|)\Pi(Z_t^{i,N,\varepsilon})(\sigma_0{\rm d}B_t^{1,i}+\sigma_1{\rm d}W_t^{1})+ h_\varepsilon^*(|Z_t^{i,N,\varepsilon}|)(\sigma_0{\rm d}B_t^{2,i}+\sigma_1{\rm d}W_t^{2})\\
~\qquad\qquad+ \sigma(Y_{\overline{\theta}_t}^{i,N,\varepsilon},\widetilde{\mu}_{\overline{\theta}_t}^{N,\varepsilon}){\rm d}\overline B_t^i+\overline\sigma(Y_{\overline{\theta}_t}^{i,N,\varepsilon},\widetilde{\mu}_{\overline{\theta}_t}^{N,\varepsilon}){\rm d}\overline{W}_t,
\end{cases}
\end{align}
with the initial condition $(Y_0^{i,\varepsilon}, Y_0^{i,N,\varepsilon})_{i\in\mathbb{S}_N}=(X_0^i, X_0^{i,N})_{i\in\mathbb{S}_N}$, which are i.i.d. $\mathcal{F}_0^B$-measurable random variables. The quantities $\widehat{\mu}_t^{i,\varepsilon}, \widetilde{\mu}_t^{N,\varepsilon}$, and $Z_t^{i,N,\varepsilon}$ are defined respectively by
$$\widehat{\mu}_t^{i,\varepsilon}=\mathcal{L}(Y_t^{i,\varepsilon}\vert \mathcal{F}_t^W), \quad \widetilde{\mu}_t^{N,\varepsilon}=\frac{1}{N}\sum_{j=1}^N\delta_{Y_t^{j,N,\varepsilon}}, \quad Z_t^{i,N,\varepsilon}=Y_t^{i,\varepsilon}-Y_t^{i,N,\varepsilon}.$$

Furthermore, for the notational brevity, we set for any $t\ge 0$,
$${\bf X}_t^{N}:=(X_t^1,\cdots,X_t^N), \quad {\bf X}_t^{N,N}:=(X_t^{1,N},\cdots,X_t^{N,N}),$$
and
$${\bf Y}_t^{N, \varepsilon}:=(Y_t^{1, \varepsilon},\cdots,Y_t^{N, \varepsilon}), \quad {\bf Y}_t^{N,N, \varepsilon}:=(Y_t^{1,N, \varepsilon},\cdots,Y_t^{N,N, \varepsilon}).$$

Now, we claim that for any $\varepsilon>0$, $({\bf Y}_t^{N, \varepsilon},{\bf Y}_t^{N,N, \varepsilon})_{t\ge 0}$ is a coupling of $({\bf X}_t^{N})_{t\ge 0}$ and $({\bf X}_t^{N,N})_{t\ge 0}$ given $\mathcal{F}^W$.
\begin{lemma}\label{coupling process-lemma}
Assume that the McKean-Vlasov SDEs \eqref{NIPS} and \eqref{IPS} are weakly well-posed. Then, for any $\varepsilon>0$, the path-valued processes $({\bf Y}_t^{N, \varepsilon})_{t\ge 0}$ and $({\bf Y}_t^{N, N, \varepsilon})_{t\ge 0}$ share the common conditional distributions given $\mathcal{F}^W$ as those of $({\bf X}_t^{N})_{t\ge 0}$ and $({\bf X}_t^{N, N})_{t\ge 0}$, respectively.
\end{lemma}
\begin{proof}
For any $i\in\mathbb{S}_N$, we first show that the conditional distributions of $(Y_t^{i,\varepsilon})_{t\ge 0}$ and $(X_t^i)_{t\ge 0}$ given $\mathcal{F}^W$ are identical. Set for $i\in\mathbb{S}_N$ and $t\ge 0$,
\begin{align*}
\widetilde B_t^{i}=&\int_0^t h_\varepsilon(|Z_s^{i,N,\varepsilon}|){\rm d} B_s^{1,i}+\int_0^t h_\varepsilon^*(|Z_s^{i,N,\varepsilon}|){\rm d}B_s^{2,i},\notag\\
\widetilde W_t=&\int_0^t h_\varepsilon(|Z_s^{i,N,\varepsilon}|){\rm d} W_s^{1}+\int_0^t h_\varepsilon^*(|Z_s^{i,N,\varepsilon}|){\rm d}W_s^{2}.
\end{align*}
Recall that $h_\varepsilon(r)^2+h_\varepsilon^*(r)^2=1$ for all $r\ge 0$, since $B^{1,i}$ and $B^{2,i}$ (resp. $W^{1}$ and $W^{2}$) are mutually independent, L\'evy's characterization implies that $\widetilde B^i:=(\widetilde B_t^i)_{t\ge 0}$ and $\widetilde W:=(\widetilde W_t)_{t\ge 0}$ are still Brownian motions. Then, the McKean-Vlasov SDE with common noise solved by $(Y_t^{i,\varepsilon})_{t\ge 0}$ can be rewritten as
\begin{align}\label{(2.2)}
\d   Y_t^{i,\varepsilon}=&b(Y_t^{i,\varepsilon},\widehat{\mu}_t^{i,\varepsilon}){\rm d}t+\sigma_0{\rm d}\widetilde B_t^i+ \sigma(Y_t^{i,\varepsilon},\widehat{\mu}_t^{i,\varepsilon}){\rm d}\overline B_t^i+\sigma_1{\rm d}\widetilde W_t+\overline\sigma(Y_t^{i,\varepsilon},\widehat{\mu}_t^{i,\varepsilon}){\rm d}\overline{W}_t.
\end{align}
In order to prove that, for any $\varepsilon>0$, the conditional distribution $(\widehat{\mu}_t^{i,\varepsilon})_{t\ge 0}$ is equal to $(\mu_t^i)_{t\ge 0}$, it remains to verify that, for any $i\neq j$, $\widetilde B_t^i$ and $\widetilde B_t^j$ are mutually independent. For any $u,v\in\mathbb{R}^d, i, j\in\mathbb{S}_N$ with $i\neq j$, and $t>0$, consider $\mathbb{E}(\<u,\widetilde B_t^i\>\<v,\widetilde B_t^j\>).$ Indeed, by applying It\^o's formula, it follows that for $u,v\in\mathbb{R}^d, i, j\in\mathbb{S}_N$ with $i\neq j$, and $t>0$,
\begin{align}\label{(2.3)}
{\rm d}(\<u,\widetilde B_t^i\>\<v,\widetilde B_t^j\>)=&\<u,\widetilde B_t^i\>{\rm d}\<v,\widetilde B_t^j\>+\<v,\widetilde B_t^j\>{\rm d}\<u,\widetilde B_t^i\>+{\rm d}[\<u,\widetilde B_t^i\>,\<v,\widetilde B_t^j\>]\notag\\
=&\<u,\widetilde B_t^i\>{\rm d}\<v,\widetilde B_t^j\>+\<v,\widetilde B_t^j\>{\rm d}\<u,\widetilde B_t^i\>,
\end{align}
where the second identity holds true due to the quadratic variation $[\<u,\widetilde B_t^i\>,\<v,\widetilde B_t^j\>]=0$, which is valid since $B^{1,1},\cdots, B^{1,N}$ (resp. $B^{2,1},\cdots, B^{2,N}$) are independent and $(B^{1,1},\cdots, B^{1,N})$ is also independent of $(B^{2,1},\cdots, B^{2,N})$. Obviously, \eqref{(2.3)} manifests that $(\<u,\widetilde B_t^i\>\<v,\widetilde B_t^j\>)_{t\ge 0}$ is a martingale. This results in that the covariance matrix $\mathbb{E}(\widetilde B_t^i\otimes \widetilde B_t^j), i\neq j$, is a $d\times d$ zero matrix. Hence, we conclude that, for any $i\neq j$, $\widetilde B_t^i$ and $\widetilde B_t^j$ are mutually independent. Moreover, we can also deduce that $(\widetilde B^i)_{i\in\mathbb{S}_N}$, $({\overline B}^i)_{i\in\mathbb{S}_N}$, $\widetilde W$, and $\overline{W}$ are mutually independent. Thus, thanks to the weak uniqueness of \eqref{NIPS}, we conclude that $({\bf X}_t^N)_{t\ge 0}$ and $({\bf Y}_t^{N,\varepsilon})_{t\ge 0}$ posses the same conditional distribution given $\mathcal{F}^W$.

For $i\in\mathbb{S}_N$ and $t\ge 0$, let
$$\widehat B_t^{1,i}=\int_0^t\Pi(Z_s^{i,N,\varepsilon})\d B_s^{1,i},~\widehat W_t^{1}=\int_0^t\Pi(Z_s^{i,N,\varepsilon})\d W_s^{1},$$
with this notation, the SDE solved by $(Y_t^{i,N,\varepsilon})_{t\ge 0}$ can be reformulated as
\begin{align}\label{(2.1)}
\d   Y_t^{i,N,\varepsilon}=&\widetilde{b}(Y_{\theta_t}^{i,N,\varepsilon},\widetilde{\mu}_{\overline{\theta}_t}^{N,\varepsilon}){\rm d}t+\sigma_0h_\varepsilon(|Z_t^{i,N,\varepsilon}|){\rm d}\widehat B_t^{1,i}+ \sigma_0h_\varepsilon^*(|Z_t^{i,N,\varepsilon}|){\rm d}B_t^{2,i}+ \sigma(Y_{\overline{\theta}_t}^{i,N,\varepsilon},\widetilde{\mu}_{\overline{\theta}_t}^{N,\varepsilon}){\rm d}\overline B_t^i\notag\\
&+\sigma_1h_\varepsilon(|Z_t^{i,N,\varepsilon}|){\rm d}\widehat W_t^{1}+\sigma_1h_\varepsilon^*(|Z_t^{i,N,\varepsilon}|){\rm d}W_t^{2}+\overline\sigma(Y_{\overline{\theta}_t}^{i,N,\varepsilon},\widetilde{\mu}_{\overline{\theta}_t}^{N,\varepsilon}){\rm d}\overline{W}_t,
\end{align}
with the initial value $Y_0^{i,N,\varepsilon}=X_0^{i,N}$. Then, by following an analogous procedure above, we deduce that the conditional distributions of $({\bf Y}_t^{N,N,\varepsilon})_{t\ge 0}$ and $({\bf X}_t^{N,N})_{t\ge 0}$ given $\mathcal{F}^W$ are identical. 
\end{proof}

\begin{lemma}\label{boundedness of Y_t^i,varepsilon}
Suppose that $\mathbb{E}|X_0^i|^q<\infty$ for all $i\in\mathbb{S}_N$ and some $q>1$. Let Assumptions \ref{A1} and \ref{A2}  hold with $\lambda_1>2\lambda_2$, then there exists $C>0$ such that for any $L\in[0,L^*]$, $i\in\mathbb{S}_N$ and $\varepsilon>0$,
 \begin{align*}
\sup_{t\ge 0} \mathbb{E}|Y_t^{i,\varepsilon}|^q\le \mathbb{E}|X_0^i|^q+C,
 \end{align*}
 where $L^*$ was defined in Lemma \ref{CPoC}.
\end{lemma}
\begin{proof}
Note that $h_{\varepsilon}(r)^2+h_{\varepsilon}^*(r)^2=1$ for all $r>0$, by following an analogous procedure as the proof of \eqref{X_t_1}, we can get the assertation.
\end{proof}

%========================================================================
\subsection{Proof of Theorem \ref{main theorem}}\label{proof of the key estimate}

\begin{proof}[Proof of Theorem \ref{main theorem}]
For any $\mu,\nu\in\mathcal{P}_1(\mathbb{R}^d)$, we choose $(X_0^i, X_0^{i,N})_{i\in\mathbb{S}_N}$ such that
\begin{align}\label{W_1(mu,nu_0)}
\mathbb{W}_1(\mu,\nu)=\mathbb{E}\vert X_0^i-X_0^{i,N}\vert,\quad i\in\mathbb{S}_N.
\end{align}
Since $Z_t^{i,N,\varepsilon}=Y_t^{i,\varepsilon}-Y_t^{i,N,\varepsilon}$, by \eqref{coupling process},
\begin{align*}
\d   Z_t^{i,N,\varepsilon}=&\left(b(Y_t^{i,\varepsilon},\widehat{\mu}_t^{i,\varepsilon})-\widetilde{b}(Y_{\theta_t}^{i,N,\varepsilon},\widetilde{\mu}_{\overline{\theta}_t}^{N,\varepsilon})\right){\rm d}t+h_\varepsilon(|Z_t^{i,N,\varepsilon}|)\left(I_{d\times d}-\Pi(Z_t^{i,N,\varepsilon})\right)\left(\sigma_0{\rm d}B_t^{1,i}+\sigma_1{\rm d}W_t^{1}\right)\\
&+\left(\sigma(Y_t^{i,\varepsilon},\widehat{\mu}_t^{i,\varepsilon})-\sigma(Y_{\overline{\theta}_t}^{i,N,\varepsilon},\widetilde{\mu}_{\overline{\theta}_t}^{N,\varepsilon})\right){\rm d}\overline B_t^i+\left(\overline\sigma(Y_t^{i,\varepsilon},\widehat{\mu}_t^{i,\varepsilon})-\overline\sigma(Y_{\overline{\theta}_t}^{i,N,\varepsilon},\widetilde{\mu}_{\overline{\theta}_t}^{N,\varepsilon})\right){\rm d}\overline{W}_t.
\end{align*}
Applying It\^o's formula, for any $i\in\mathbb{S}_N$, we have
\begin{align}\label{(3.11)}
\d   |Z_t^{i,N,\varepsilon}|^2=&\left(2\left\<Z_t^{i,N,\varepsilon},b(Y_t^{i,\varepsilon},\widehat{\mu}_t^{i,\varepsilon})-\widetilde{b}(Y_{\theta_t}^{i,N,\varepsilon},\widetilde{\mu}_{\overline{\theta}_t}^{N,\varepsilon})\right\>+4(\sigma_0^2+\sigma_1^2)h_\varepsilon(|Z_t^{i,N,\varepsilon}|)^2\right){\rm d}t\notag\\
&+\left(|\sigma(Y_t^{i,\varepsilon},\widehat{\mu}_t^{i,\varepsilon})-\sigma(Y_{\overline{\theta}_t}^{i,N,\varepsilon},\widetilde{\mu}_{\overline{\theta}_t}^{N,\varepsilon})|^2+|\overline\sigma(Y_t^{i,\varepsilon},\widehat{\mu}_t^{i,\varepsilon})-\overline\sigma(Y_{\overline{\theta}_t}^{i,N,\varepsilon},\widetilde{\mu}_{\overline{\theta}_t}^{N,\varepsilon})|^2\right){\rm d}t\notag\\
&+2\left\<Z_t^{i,N,\varepsilon}, h_\varepsilon(|Z_t^{i,N,\varepsilon}|)\left(I_{d\times d}-\Pi(Z_t^{i,N,\varepsilon})\right)\left(\sigma_0{\rm d}B_t^{1,i}+\sigma_1{\rm d}W_t^{1}\right)\right\>\notag\\
&+2\left\<Z_t^{i,N,\varepsilon},\sigma(Y_t^{i,\varepsilon},\widehat{\mu}_t^{i,\varepsilon})-\sigma(Y_{\overline{\theta}_t}^{i,N,\varepsilon},\widetilde{\mu}_{\overline{\theta}_t}^{N,\varepsilon})\right\>{\rm d}\overline B_t^i\notag\\
&+2\left\<Z_t^{i,N,\varepsilon},\overline\sigma(Y_t^{i,\varepsilon},\widehat{\mu}_t^{i,\varepsilon})-\overline\sigma(Y_{\overline{\theta}_t}^{i,N,\varepsilon},\widetilde{\mu}_{\overline{\theta}_t}^{N,\varepsilon})\right\>{\rm d}\overline{W}_t.
\end{align}
Below, we define a $C^2$-function 
$$f(r)=1-{\rm e}^{-c_1r}+c_2r,\quad r\ge 0,$$
where $c_1=\frac{2\lambda_1l_0}{\sigma_0^2+\sigma_1^2}, c_2=c_1{\rm e}^{-c_1l_0}$. 

Choose $\lambda_2^*= \frac{(2\phi(l_0)\wedge\lambda_1l_0)c_1l_0}{(1+{\rm e}^{c_1l_0})({\rm e}^{c_1l_0}-1+c_2l_0{\rm e}^{c_1l_0})}>0$, we can arrive at for any $\lambda_2\in[0,\lambda_2^*)$, 
$$\lambda_0:=\frac{(2\phi(l_0)\wedge\lambda_1l_0)l_0c_2}{1-{\rm e}^{-c_1l_0}+c_2l_0}-\frac{\lambda_2(c_1+c_2)}{c_2}>0.$$
Moreover, we define $F(r)=f(r^{1/2}), r\ge 0$. Applying It\^o-Tanaka's  formula(\cite[Theorem 29.5]{Kallenberg2021book}) yields that for all $t>0$,
\begin{align*}
{\rm e}^{\lambda_0t}F(|Z_t^{i,N,\varepsilon}|^2)=&F(|Z_0^{i,N,\varepsilon}|^2)+\lambda_0\int_0^t{\rm e}^{\lambda_0s}F(|Z_s^{i,N,\varepsilon}|^2){\rm d}s\\
&+\int_0^t{\rm e}^{\lambda_0s}F'_-(|Z_s^{i,N,\varepsilon}|^2){\rm d}|Z_s^{i,N,\varepsilon}|^2+\frac{1}{2}\int_0^t{\rm e}^{\lambda_0s}\int_0^{\infty}{\rm d}L_s^{i,N,\varepsilon,x}\mu_F({\rm d}x),
\end{align*}
where $F'_-$ means the left derivative of $F$, $(L_t^{i,N,\varepsilon,x})_{t\ge 0}$ is the local time of $(|Z_t^{i,N,\varepsilon}|^2)_{t\ge 0}$ at point $x$, and $\mu_F$ denotes the Lebesgue-Stieltjes measure associated with the left derivative $F'_-$ (i.e., $\mu_F([a,b))=F'_-(b)-F'_-(a)$ for $a\le b$). Denote by $F''$ the almost everywhere defined second derivative of the function $F$. Since $\mu_F({\rm d}x)\le F''(x){\rm d}x$ (thanks to the fact that $F'_-$ is non-increasing) and $t\mapsto L_t^{i,N,\varepsilon,x}$ is increasing, we infer that
\begin{align}\label{(3.12)}
{\rm e}^{\lambda_0t}F(|Z_t^{i,N,\varepsilon}|^2)\le&F(|Z_0^{i,N,\varepsilon}|^2)+\lambda_0\int_0^t{\rm e}^{\lambda_0s}F(|Z_s^{i,N,\varepsilon}|^2){\rm d}s\notag\\
&+\int_0^t{\rm e}^{\lambda_0s}F'_-(|Z_s^{i,N,\varepsilon}|^2){\rm d}|Z_s^{i,N,\varepsilon}|^2+\frac{1}{2}\int_0^t{\rm e}^{\lambda_0s}\int_0^{\infty}{\rm d}L_s^{i,N,\varepsilon,x}F''(x){\rm d}x,
\end{align}
Next, by the chain rule and Fubini's theorem, in addition to the occupation time formula(see \cite[Theorem 29.5]{Kallenberg2021book}), we can deduce that
\begin{align*}
&\int_0^t{\rm e}^{\lambda_0s}\int_0^{\infty}{\rm d}L_s^{i,N,\varepsilon,x}F''(x){\rm d}x\\
=&\int_0^{\infty}\left(\int_0^t{\rm e}^{\lambda_0s}{\rm d}L_s^{i,N,\varepsilon,x}\right)F''(x){\rm d}x\\
=&\int_0^{\infty}\left({\rm e}^{\lambda_0t}L_t^{i,N,\varepsilon,x}-\lambda_0\int_0^t{\rm e}^{\lambda_0s}L_s^{i,N,\varepsilon,x}{\rm d}s\right)F''(x){\rm d}x\\
=&{\rm e}^{\lambda_0t}\int_0^t F''(|Z_s^{i,N,\varepsilon}|^2){\rm d}\<|Z^{i,N,\varepsilon}|^2\>_s-\lambda_0\int_0^t {\rm e}^{\lambda_0s}\left(\int_0^s F''(|Z_u^{i,N,\varepsilon}|^2){\rm d}\<|Z^{i,N,\varepsilon}|^2\>_u\right){\rm d}s\\
=&{\rm e}^{\lambda_0t}\int_0^t F''(|Z_s^{i,N,\varepsilon}|^2){\rm d}\<|Z^{i,N,\varepsilon}|^2\>_s-\lambda_0\int_0^t \left(\int_u^t{\rm e}^{\lambda_0s}{\rm d}s \right)F''(|Z_u^{i,N,\varepsilon}|^2){\rm d}\<|Z^{i,N,\varepsilon}|^2\>_u\\
=&\int_0^t {\rm e}^{\lambda_0s}F''(|Z_s^{i,N,\varepsilon}|^2){\rm d}\<|Z^{i,N,\varepsilon}|^2\>_s,
\end{align*}
where $(\<|Z^{i,N,\varepsilon}|^2\>_t)_{t\ge 0}$ stands for the quadratic variation process of $(|Z_t^{i,N,\varepsilon}|^2)_{t\ge 0}$. Plugging the preceding identity into \eqref{(3.12)} enables us to deduce that
\begin{align}\label{(3.13)}
{\rm e}^{\lambda_0t}F(|Z_t^{i,N,\varepsilon}|^2)=&F(|Z_0^{i,N,\varepsilon}|^2)+\lambda_0\int_0^t{\rm e}^{\lambda_0s}F(|Z_s^{i,N,\varepsilon}|^2){\rm d}s\notag\\
&+\int_0^t{\rm e}^{\lambda_0s}F'_-(|Z_s^{i,N,\varepsilon}|^2){\rm d}|Z_s^{i,N,\varepsilon}|^2+\frac{1}{2}\int_0^t {\rm e}^{\lambda_0s}F''(|Z_s^{i,N,\varepsilon}|^2){\rm d}\<|Z^{i,N,\varepsilon}|^2\>_s.
\end{align}
Note that for $r>0$,
$$F'(r)=\frac{1}{2}f'(r^{1/2})r^{-1/2},\quad F''(r)=\frac{1}{4}f''(r^{1/2})r^{-1}-\frac{1}{4}f'(r^{1/2})r^{-3/2},$$%=\frac{1}{4}f''(r^{1/2})r^{-1}-\frac{1}{2}F'(r)r^{-1}
and that
\begin{align*}
{\rm d}\<|Z^{i,N,\varepsilon}|^2\>_t=&\Bigg[16(\sigma_0^2+\sigma_1^2)h_\varepsilon(|Z_t^{i,N,\varepsilon}|)^2\left|(Z_t^{i,N,\varepsilon})^{\rm T}{\bf e}(Z_t^{i,N,\varepsilon})\otimes{\bf e}(Z_t^{i,N,\varepsilon})\right|^2\\
&+4\left|(Z_t^{i,N,\varepsilon})^{\rm T}\left(\sigma(Y_t^{i,\varepsilon},\widehat{\mu}_t^{i,\varepsilon})-\sigma(Y_{\overline{\theta}_t}^{i,N,\varepsilon},\widetilde{\mu}_{\overline{\theta}_t}^{N,\varepsilon})\right)\right|^2\\
&+4\left|(Z_t^{i,N,\varepsilon})^{\rm T}\left(\overline\sigma(Y_t^{i,\varepsilon},\widehat{\mu}_t^{i,\varepsilon})-\overline\sigma(Y_{\overline{\theta}_t}^{i,N,\varepsilon},\widetilde{\mu}_{\overline{\theta}_t}^{N,\varepsilon})\right)\right|^2\Bigg]{\rm d}t\\
\le&4|Z_t^{i,N,\varepsilon}|^2\bigg[4(\sigma_0^2+\sigma_1^2)h_\varepsilon(|Z_t^{i,N,\varepsilon}|)^2+\left|\sigma(Y_t^{i,\varepsilon},\widehat{\mu}_t^{i,\varepsilon})-\sigma(Y_{\overline{\theta}_t}^{i,N,\varepsilon},\widetilde{\mu}_{\overline{\theta}_t}^{N,\varepsilon})\right|^2\\
&\quad\qquad\qquad+\left|\overline\sigma(Y_t^{i,\varepsilon},\widehat{\mu}_t^{i,\varepsilon})-\overline\sigma(Y_{\overline{\theta}_t}^{i,N,\varepsilon},\widetilde{\mu}_{\overline{\theta}_t}^{N,\varepsilon})\right|^2\bigg]{\rm d}t.
\end{align*}
Whence, along with \eqref{(3.11)} and \eqref{(3.13)} as well as $f''(r)=-c_1^2{\rm e}^{-c_1r}<0$, we derive that
\begin{align}\label{(3.14)}
&{\rm e}^{\lambda_0t}f(|Z_t^{i,N,\varepsilon}|)\notag\\
&=f(|Z_0^{i,N,\varepsilon}|)+\lambda_0\int_0^t{\rm e}^{\lambda_0s}f(|Z_s^{i,N,\varepsilon}|){\rm d}s+\frac{1}{2}\int_0^t{\rm e}^{\lambda_0s}f'(|Z_s^{i,N,\varepsilon}|)|Z_s^{i,N,\varepsilon}|^{-1}{\rm d}|Z_s^{i,N,\varepsilon}|^2\notag\\
&\quad+\frac{1}{2}\int_0^t {\rm e}^{\lambda_0s}\left(f''(|Z_s^{i,N,\varepsilon}|)-f'(|Z_s^{i,N,\varepsilon}|)|Z_s^{i,N,\varepsilon}|^{-1}\right)\notag\\
&\quad\qquad\times\Big(4(\sigma_0^2+\sigma_1^2)h_\varepsilon(|Z_s^{i,N,\varepsilon}|)^2+\left|\sigma(Y_s^{i,\varepsilon},\widehat{\mu}_s^{i,\varepsilon})-\sigma(Y_{\overline{\theta}_s}^{i,N,\varepsilon},\widetilde{\mu}_{\overline{\theta}_s}^{N,\varepsilon})\right|^2+\left|\overline\sigma(Y_s^{i,\varepsilon},\widehat{\mu}_s^{i,\varepsilon})-\overline\sigma(Y_{\overline{\theta}_s}^{i,N,\varepsilon},\widetilde{\mu}_{\overline{\theta}_s}^{N,\varepsilon})\right|^2\Big){\rm d}s\notag\\
&\le f(|Z_0^{i,N,\varepsilon}|)+\int_0^t{\rm e}^{\lambda_0s}\Big(\lambda_0f(|Z_s^{i,N,\varepsilon}|)+2f''(|Z_s^{i,N,\varepsilon}|)(\sigma_0^2+\sigma_1^2)h_\varepsilon(|Z_s^{i,N,\varepsilon}|)^2\Big){\rm d}s\notag\\
&\quad+\int_0^t{\rm e}^{\lambda_0s}f'(|Z_s^{i,N,\varepsilon}|)|Z_s^{i,N,\varepsilon}|^{-1}\left\<Z_s^{i,N,\varepsilon},b(Y_s^{i,\varepsilon},\widehat{\mu}_s^{i,\varepsilon})-\widetilde{b}(Y_{\theta_s}^{i,N,\varepsilon},\widetilde{\mu}_{\overline{\theta}_s}^{N,\varepsilon})\right\>{\rm d}s+M_t^{i,N,\varepsilon},
%&+2\int_0^t{\rm e}^{\lambda_0s}f'(|Z_s^{i,N,\varepsilon}|)|Z_s^{i,N,\varepsilon}|^{-1}h_\varepsilon(\|{\bf Z}_s^{N,N,\varepsilon}\|_1)^2\sigma_0^2{\rm d}s\notag\\
%&+\frac{1}{2}\int_0^t{\rm e}^{\lambda_0s}f'(|Z_s^{i,N,\varepsilon}|)|Z_s^{i,N,\varepsilon}|^{-1}|(\sigma\ast\widehat{\mu}_t^{i,\varepsilon})(Y_t^{i,\varepsilon})-(\sigma\ast\widetilde{\mu}_t^{N,\varepsilon})(Y_t^{i,N,\varepsilon})|^2{\rm d}t\notag\\
%&+\frac{1}{2}\int_0^t {\rm e}^{\lambda_0s}\left(f''(|Z_s^{i,N,\varepsilon}|)-f'(|Z_s^{i,N,\varepsilon}|)|Z_s^{i,N,\varepsilon}|^{-1}\right)\left|(\sigma\ast\widehat{\mu}_s^{i,\varepsilon})(Y_s^{i,\varepsilon})-(\sigma\ast\widetilde{\mu}_s^{N,\varepsilon})(Y_s^{i,N,\varepsilon})\right|^2{\rm d}s\notag\\
\end{align}
where 
\begin{align*}
M_t^{i,N,\varepsilon}=&2\int_0^t{\rm e}^{\lambda_0s}f'(|Z_s^{i,N,\varepsilon}|)|Z_s^{i,N,\varepsilon}|^{-1}\left\<Z_s^{i,N,\varepsilon},h_\varepsilon(|Z_s^{i,N,\varepsilon}|){\bf e}(Z_s^{i,N,\varepsilon})\otimes {\bf e}(Z_s^{i,N,\varepsilon})(\sigma_0{\rm d}B_s^{1,i}+\sigma_1{\rm d}W_s^{1})\right\>\notag\\
&+\int_0^t{\rm e}^{\lambda_0s}f'(|Z_s^{i,N,\varepsilon}|)|Z_s^{i,N,\varepsilon}|^{-1}\left\<Z_s^{i,N,\varepsilon},\sigma(Y_s^{i,\varepsilon},\widehat{\mu}_s^{i,\varepsilon})-\sigma(Y_{\overline{\theta}_s}^{i,N,\varepsilon},\widetilde{\mu}_{\overline{\theta}_s}^{N,\varepsilon})\right\>{\rm d}\overline B_s^i\notag\\
&+\int_0^t{\rm e}^{\lambda_0s}f'(|Z_s^{i,N,\varepsilon}|)|Z_s^{i,N,\varepsilon}|^{-1}\left\<Z_s^{i,N,\varepsilon},\overline\sigma(Y_s^{i,\varepsilon},\widehat{\mu}_s^{i,\varepsilon})-\overline\sigma(Y_{\overline{\theta}_s}^{i,N,\varepsilon},\widetilde{\mu}_{\overline{\theta}_s}^{N,\varepsilon})\right\>{\rm d}\overline{W}_s.
\end{align*}
Let $\overline\mu_t^{N,\varepsilon}=\frac{1}{N}\sum_{j=1}^N\delta_{Y_t^{j,\varepsilon}}$, using Assumption \ref{A1}, and recall that $\widetilde{\mu}_t^{N,\varepsilon}=\frac{1}{N}\sum_{j=1}^N\delta_{Y_t^{j,N,\varepsilon}}$, it follows 
\begin{align}\label{(3.15)}
&|Z_t^{i,N,\varepsilon}|^{-1}\left\<Z_t^{i,N,\varepsilon},b(Y_t^{i,\varepsilon},\widehat{\mu}_t^{i,\varepsilon})-\widetilde{b}(Y_{\theta_t}^{i,N,\varepsilon},\widetilde{\mu}_{\overline{\theta}_t}^{N,\varepsilon})\right\>\notag\\
&\le |Z_t^{i,N,\varepsilon}|^{-1}\left\<Z_t^{i,N,\varepsilon},b(Y_t^{i,\varepsilon},\widehat{\mu}_t^{i,\varepsilon})-b(Y_t^{i,N,\varepsilon},\widehat{\mu}_t^{i,\varepsilon})\right\>+\left|b(Y_t^{i,N,\varepsilon},\widehat{\mu}_t^{i,\varepsilon})-b(Y_t^{i,N,\varepsilon},\overline{\mu}_t^{N,\varepsilon})\right|\notag\\
&\quad+\left|b(Y_t^{i,N,\varepsilon},\overline{\mu}_t^{N,\varepsilon})-b(Y_t^{i,N,\varepsilon},\widetilde{\mu}_t^{N,\varepsilon})\right|+\left|b(Y_t^{i,N,\varepsilon},\widetilde{\mu}_t^{N,\varepsilon})-\widetilde{b}(Y_{\theta_t}^{i,N,\varepsilon},\widetilde{\mu}_{\overline{\theta}_t}^{N,\varepsilon})\right|\notag\\
&\le \phi(\vert Z_t^{i,N,\varepsilon}\vert)\mathbbm{1}_{\{|Z_t^{i,N,\varepsilon}|\le l_0\}}-\lambda_1|Z_t^{i,N,\varepsilon}|\mathbbm{1}_{\{|Z_t^{i,N,\varepsilon}|> l_0\}}+\lambda_2\mathbb{W}_1(\widehat{\mu}_t^{i,\varepsilon},\overline{\mu}_t^{N,\varepsilon})+\frac{\lambda_2}{N}\sum_{j=1}^N|Z_t^{j,N,\varepsilon}|\notag\\
&\quad+|b(Y_t^{i,N,\varepsilon},\widetilde{\mu}_t^{N,\varepsilon})-\widetilde{b}(Y_{\theta_t}^{i,N,\varepsilon},\widetilde{\mu}_{\overline{\theta}_t}^{N,\varepsilon})|.
\end{align}
Plugging \eqref{(3.15)} back into \eqref{(3.14)} yields that
\begin{align*}
&{\rm d}\left({\rm e}^{\lambda_0t}f(|Z_t^{i,N,\varepsilon}|)\right)\notag\\
&\le {\rm e}^{\lambda_0t}\Bigg\{\lambda_0f(|Z_t^{i,N,\varepsilon}|)+\varphi(|Z_t^{i,N,\varepsilon}|)h_\varepsilon(|Z_t^{i,N,\varepsilon}|)^2+\Upsilon(|Z_t^{i,N,\varepsilon}|)\notag\\
&\qquad+f'(|Z_t^{i,N,\varepsilon}|)\left(\lambda_2\mathbb{W}_1(\widehat{\mu}_t^{i,\varepsilon},\overline{\mu}_t^{N,\varepsilon})+\frac{\lambda_2}{N}\sum_{j=1}^N|Z_t^{j,N,\varepsilon}|+|b(Y_t^{i,N,\varepsilon},\widetilde{\mu}_t^{N,\varepsilon})-\widetilde{b}(Y_{\theta_t}^{i,N,\varepsilon},\widetilde{\mu}_{\overline{\theta}_t}^{N,\varepsilon})|\right)\Bigg\}{\rm d}t\notag\\
&\quad+\d M_t^{i,N,\varepsilon},
\end{align*}
where 
\begin{align*}
\varphi(r):=&2f''(r)(\sigma_0^2+\sigma_1^2)+f'(r)\left(\phi(r)\mathbbm{1}_{\{r\le l_0\}}-\lambda_1r\mathbbm{1}_{\{r> l_0\}}\right)r,\notag\\
\Upsilon(r):=&f'(r)\left(\phi(r)\mathbbm{1}_{\{r\le l_0\}}-\lambda_1r\mathbbm{1}_{\{r> l_0\}}\right)\left(1-h_\varepsilon(r)^2\right)r.
\end{align*}
Below, for the case $0\le r\le l_0$ and the case $r>l_0$, we aim to verify respectively that 
\begin{align}\label{varphi^*(r)}
\varphi(r)\le -\lambda^*f(r), \quad r\ge 0,
\end{align}
where $\lambda^*=\frac{(2\phi(l_0)\wedge\lambda_1l_0)l_0c_2}{1-{\rm e}^{-c_1l_0}+c_2l_0}>0$. By virtue of 
 $$f'(r)=c_1{\rm e}^{-c_1r}+c_2, \quad f''(r)=-c_1^2{\rm e}^{-c_1r},\quad r\ge 0,$$
 we have
 $$\varphi(r)=-2c_1^2(\sigma_0^2+\sigma_1^2){\rm e}^{-c_1r}+(c_1{\rm e}^{-c_1r}+c_2)\left[\phi(r)\mathbbm{1}_{\{r\le l_0\}}-\lambda_1r\mathbbm{1}_{\{r> l_0\}}\right]r.$$
 For the case $0\le r\le l_0$, note that $c_1=\frac{2l_0\phi(l_0)}{\sigma_0^2+\sigma_1^2}, c_2=c_1{\rm e}^{-c_1l_0}\le c_1{\rm e}^{-c_1r}$, 
 \begin{align}\label{varphi^*(r)_1}
 \varphi(r)\le&-2c_1^2(\sigma_0^2+\sigma_1^2){\rm e}^{-c_1r}+(c_1{\rm e}^{-c_1r}+c_2)l_0\phi(l_0)\notag\\
 \le&-c_1{\rm e}^{-c_1r}\left(2c_1(\sigma_0^2+\sigma_1^2)-2l_0\phi(l_0)\right)\notag\\
 =&-c_1^2{\rm e}^{-c_1r}(\sigma_0^2+\sigma_1^2)\notag\\
 \le &-\frac{c_1^2{\rm e}^{-c_1l_0}(\sigma_0^2+\sigma_1^2)}{f(l_0)}f(r)\notag\\
 =&-\frac{2c_2l_0\phi(l_0)}{1-{\rm e}^{-c_1l_0}+c_2l_0}f(r),
\end{align}
where the last inequality is valid due to the fact that $r\mapsto f(r)$ is increasing on $[0,\infty)$.
On the other hand, for $r>l_0$, using the fact that $r\mapsto \frac{r}{f(r)}$ is increasing on the interval $[0,\infty)$, we obtain that
\begin{align}\label{varphi^*(r)_2}
\varphi(r)=&-2c_1^2(\sigma_0^2+\sigma_1^2){\rm e}^{-c_1r}-\lambda_1r(c_1{\rm e}^{-c_1r}+c_2)r\le -\lambda_1c_2r^2\le -\frac{\lambda_1c_2l_0^2}{1-{\rm e}^{-c_1l_0}+c_2l_0}f(r).
\end{align}
  Hence, \eqref{varphi^*(r)} follows by combining \eqref{varphi^*(r)_1} with \eqref{varphi^*(r)_2}.  In the sequel, invoking \eqref{varphi^*(r)} and taking $c_2\le f'(r)\le c_1+c_2, r\ge 0$ and $f(0)=0$ into account leads to
  \begin{align*}
&{\rm e}^{\lambda_0t}\mathbb{E}f(|Z_t^{i,N,\varepsilon}|)\\
&\le\mathbb{E}f(|Z_0^{i,N,\varepsilon}|)+\int_0^t{\rm e}^{\lambda_0s}\Bigg\{-C_1\left(\mathbb{E}f(|Z_s^{i,N,\varepsilon}|)-\frac{1}{N}\sum_{j=1}^N\mathbb{E}f(|Z_s^{j,N,\varepsilon}|)\right)\notag\\
&\qquad\qquad\qquad\qquad\qquad+\lambda^*\mathbb{E}\left[f(|Z_s^{i,N,\varepsilon}|)(1-h_\varepsilon(|Z_s^{i,N,\varepsilon}|)^2)\right]+\mathbb{E}\Upsilon(|Z_s^{i,N,\varepsilon}|)\notag\\
&\qquad\qquad\qquad\qquad\qquad+(c_1+c_2)\left(\lambda_2\mathbb{E}\mathbb{W}_1(\widehat{\mu}_s^{i,\varepsilon},\overline{\mu}_s^{N,\varepsilon})+\mathbb{E}|b(Y_s^{i,N,\varepsilon},\widetilde{\mu}_s^{N,\varepsilon})-\widetilde{b}(Y_{\theta_s}^{i,N,\varepsilon},\widetilde{\mu}_{\overline{\theta}_s}^{N,\varepsilon})|\right)\Bigg\}{\rm d}s,
\end{align*}
  where $C_1=\frac{\lambda_2(c_1+c_2)}{c_2}>0$, and then
\begin{align*}
&{\rm e}^{\lambda_0t}\frac{1}{N}\sum_{i=1}^N\mathbb{E}f(|Z_t^{i,N,\varepsilon}|)\notag\\
\le &\frac{1}{N}\sum_{i=1}^N\mathbb{E}f(|Z_0^{i,N,\varepsilon}|)+\int_0^t{\rm e}^{\lambda_0s}\mathbb{E}\left[\lambda^*\frac{1}{N}\sum_{i=1}^Nf(|Z_s^{i,N,\varepsilon}|)(1-h_\varepsilon(|Z_s^{i,N,\varepsilon}|)^2)+\frac{1}{N}\sum_{i=1}^N\Upsilon(|Z_s^{i,N,\varepsilon}|)\right]{\rm d}s\notag\\
&+(c_1+c_2)\frac{1}{N}\sum_{i=1}^N\int_0^t{\rm e}^{\lambda_0s}\left(\lambda_2\mathbb{E}\mathbb{W}_1(\widehat{\mu}_s^{i,\varepsilon},\overline{\mu}_s^{N,\varepsilon})+\mathbb{E}|b(Y_s^{i,N,\varepsilon},\widetilde{\mu}_s^{N,\varepsilon})-\widetilde{b}(Y_{\theta_s}^{i,N,\varepsilon},\widetilde{\mu}_{\overline{\theta}_s}^{N,\varepsilon})|\right){\rm d}s.
\end{align*}
By invoking $c_2\le f'(r)\le c_1+c_2, r\ge 0$ and $f(0)=0$, in addition to $h_\varepsilon\in[0,1]$, we deduce that for all $\varepsilon>0$,
\begin{align*}
&\lambda^*\frac{1}{N}\sum_{i=1}^Nf(|Z_s^{i,N,\varepsilon}|)\big(1-h_\varepsilon(|Z_s^{i,N,\varepsilon}|)^2\big)+\frac{1}{N}\sum_{i=1}^N\Upsilon(|Z_s^{i,N,\varepsilon}|)\notag\\
\le &2\left(\lambda^*+\phi(l_0)\right)(c_1+c_2)\frac{1}{N}\sum_{i=1}^N|Z_s^{i,N,\varepsilon}|\big(1-h_\varepsilon(|Z_s^{i,N,\varepsilon}|)\big)\notag\\
\le &4\left(\lambda^*+\phi(l_0)\right)(c_1+c_2)\varepsilon,
\end{align*}
where in the last inequality we used the fact that 
$$r(1-h_\varepsilon(r))\le 2\varepsilon,\quad r\ge 0.$$
On the other hand, according to Lemmas \ref{CPoC} and \ref{boundedness of Y_t^i,varepsilon}, it follows that for all $s\ge 0$ and $i\in \mathbb{S}_N$,
\begin{align*}
\mathbb{E}\mathbb{W}_1(\widehat{\mu}_s^{i,\varepsilon},\overline{\mu}_s^{N,\varepsilon})\le C\psi(N),
\end{align*}
Thus,
\begin{align*}
{\rm e}^{\lambda_0t}\frac{1}{N}\sum_{i=1}^N\mathbb{E}f(|Z_t^{i,N,\varepsilon}|)\le &\frac{1}{N}\sum_{i=1}^N\mathbb{E}f(|Z_0^{i,N,\varepsilon}|)+4\left(\lambda^*+\phi(l_0)\right)(c_1+c_2)\varepsilon{\rm e}^{\lambda_0t}+C\psi(N){\rm e}^{\lambda_0t}\notag\\
&+(c_1+c_2)\frac{1}{N}\sum_{i=1}^N\int_0^t{\rm e}^{\lambda_0s}\mathbb{E}|b(Y_s^{i,N,\varepsilon},\widetilde{\mu}_s^{N,\varepsilon})-\widetilde{b}(Y_{\theta_s}^{i,N,\varepsilon},\widetilde{\mu}_{\overline{\theta}_s}^{N,\varepsilon})|{\rm d}s.
\end{align*}
Once more, with the help of $c_2r\le f(r)\le (c_1+c_2)r, r\ge 0$, along with \eqref{W_1(mu,nu_0)}, there is a constant 
\begin{align}\label{(2.17)}
\frac{1}{N}\sum_{i=1}^N\mathbb{E}|Z_t^{i,N,\varepsilon}|\le &C{\rm e}^{-\lambda_0t}\mathbb{W}_1(\mu,\nu)+C\varepsilon+C\psi(N)\notag\\
&+C\frac{1}{N}\sum_{i=1}^N\int_0^t{\rm e}^{-\lambda_0(t-s)}\mathbb{E}|b(Y_s^{i,N,\varepsilon},\widetilde{\mu}_s^{N,\varepsilon})-\widetilde{b}(Y_{\theta_s}^{i,N,\varepsilon},\widetilde{\mu}_{\overline{\theta}_s}^{N,\varepsilon})|{\rm d}s.
\end{align}
By the aid of Lemma \ref{coupling process-lemma}, $\mu_t^i=\mathcal{L}(Y_t^{i,\varepsilon}\vert \mathcal{F}_t^W)$ and $\nu_t^{i,N}=\mathcal{L}(Y_t^{i,N,\varepsilon}\vert \mathcal{F}_t^W)$ for any $i\in\mathbb{S}_N$. Moreover, recall that $(Y_0^{i,\varepsilon}, Y_0^{i,N,\varepsilon})_{i\in\mathbb{S}_N}=(X_0^i, X_0^{i,N})_{i\in\mathbb{S}_N}$ are independent and identically distributed. Therefore, we derive that 
$$\mathcal{W}_1\big(\mathcal{L}(\mu_t^i),\mathcal{L}(\nu_t^{i,N})\big)\le\mathbb{E}^W\mathbb{W}_1(\mu_t^i,\nu_t^{i,N})\le \mathbb{E}^W\mathbb{E}^B|Z_t^{i,N,\varepsilon}|=\mathbb{E}|Z_t^{i,N,\varepsilon}|=\frac{1}{N}\sum_{j=1}^N\mathbb{E}|Z_t^{j,N,\varepsilon}|,~i\in\mathbb{S}_N.$$
Thus, \eqref{(2.17)} enables us to derive that 
\begin{align*}
\mathcal{W}_1\big(\mathcal{L}(\mu_t^i),\mathcal{L}(\nu_t^{i,N})\big)\le &C{\rm e}^{-\lambda_0t}\mathbb{W}_1(\mu,\nu)+C\varepsilon+C\psi(N)\notag\\
&+C\frac{1}{N}\sum_{j=1}^N\int_0^t{\rm e}^{-\lambda_0(t-s)}\mathbb{E}|b(Y_s^{j,N,\varepsilon},\widetilde{\mu}_s^{N,\varepsilon})-\widetilde{b}(Y_{\theta_s}^{j,N,\varepsilon},\widetilde{\mu}_{\overline{\theta}_s}^{N,\varepsilon})|{\rm d}s.
\end{align*}
Subsequently, making use of Lemma \ref{coupling process-lemma} followed by approaching $\varepsilon\to 0$, and applying the prerequisite that $(X_t^{i,N})_{i\in\mathbb{S}_N}$ are identically distributed yields that
\begin{align*}
\mathcal{W}_1\big(\mathcal{L}(\mu_t^i),\mathcal{L}(\nu_t^{i,N})\big)\le &C\left({\rm e}^{-\lambda_0t}\mathbb{W}_1(\mu,\nu)+\psi(N)+\int_0^t{\rm e}^{-\lambda_0(t-s)}\mathbb{E}|b(X_s^{i,N},\widetilde{\mu}_s^N)-\widetilde{b}(X_{\theta_s}^{i,N},\widetilde{\mu}_{\overline{\theta}_s}^N)|{\rm d}s\right).
\end{align*}
\end{proof}

\section{Uniform-in-time error estimates for stochastic algorithms}\label{stochastic algorithms}
Based on Theorem \ref{main theorem}, we focus on the uniform-in-time discretization error bounds for stochastic algorithms for the McKean-Vlasov SDE with common noise \eqref{NIPS} in this section. In addition to Assumptions \ref{A1} and \ref{A2}, we further suppose that the drift coefficient $b$ is of linear or super-linear growth and the diffusion coefficients $\sigma$ and $\overline\sigma$ are bounded, which are stated precisely as below.
 
 \begin{assumption}\label{S1}
There exist constants $l\ge 0$ and $K>0$ such that for all $x,y\in\mathbb{R}^d$ and $\mu\in\mathcal{P}_1(\mathbb{R}^d)$,
\begin{align*}
|b(x,\mu)-b(y,\mu)|\le K(1+|x|^l+|y|^l)\vert x-y\vert.
\end{align*}
\end{assumption}

Combining Assumptions \ref{A1} and \ref{S1}, it yields from the Young inequality that 
\begin{align}\label{b_growth}
|b(x,\mu)|\le \left(2K+\lambda_2+|b(0,\delta_0)|\right)\left(1+|x|^{l+1}+\mathbb{W}_1(\mu,\delta_0)\right)=:\hat K\left(1+|x|^{l+1}+\mathbb{W}_1(\mu,\delta_0)\right).
\end{align}

\begin{assumption}\label{N1}
There exists a constant $\hat L>0$ such that for all $x\in\mathbb{R}^d$ and $\mu\in\mathcal{P}_1(\mathbb{R}^d)$,
\begin{align*}
|\sigma(x,\mu)|\vee|\overline\sigma(x,\mu)|\le \hat L.
\end{align*}
\end{assumption}

\subsection{Backward EM method}
For a stepsize $h>0$, the backward EM method related to the McKean-Vlasov SDE with common noise \eqref{NIPS} is defined as
\begin{align}
X_{(n+1)h}^{i,N,h}=&X_{nh}^{i,N,h}+b(X_{(n+1)h}^{i,N,h},\widetilde{\mu}_{nh}^{N,h})h+\sigma_0\Delta B_n^i+\sigma(X_{nh}^{i,N,h},\widetilde{\mu}_{nh}^{N,h})\Delta \overline B_n^i\notag\\
&+\sigma_1\Delta W_n+\overline\sigma(X_{nh}^{i,N,h},\widetilde{\mu}_{nh}^{N,h})\Delta \overline{W}_n\label{BEM-dis}
\end{align}
with initial data $X_0^{i,N,h}$ for all $i\in\mathbb{S}_N$, where $\Delta B_n^i=B_{(n+1)h}^i-B_{nh}^i$, $\Delta \overline B_n^i=\overline B_{(n+1)h}^i-\overline B_{nh}^i$, $\Delta W_n=W_{(n+1)h}-W_{nh}$, $\Delta \overline{W}_n=\overline{W}_{(n+1)h}-\overline W_{nh}$, $\widetilde{\mu}_{nh}^{N,h}=\frac{1}{N}\sum_{j=1}^N \delta_{X_{nh}^{j,N,h}}$, $n\in\mathbb{N}$.

Set $t_h=\lfloor t/h \rfloor h$ with $\lfloor t/h \rfloor$ being the integer part of $t/h$, the continuous version of the backward EM method is defined as
\begin{align}\label{BEM}
{\rm d}X_t^{i,N,h}=b(X_{t_h+h}^{i,N,h},\widetilde{\mu}_{t_h}^{N,h}){\rm d}t+\sigma_0{\rm d}B_t^i+\sigma(X_{t_h}^{i,N,h},\widetilde{\mu}_{t_h}^{N,h}){\rm d}\overline B_t^i+\sigma_1{\rm d}W_t+\overline\sigma(X_{t_h}^{i,N,h},\widetilde{\mu}_{t_h}^{N,h}){\rm d}\overline{W}_t,
\end{align}
 Obviously, \eqref{BEM} is included in \eqref{IPS} by taking $\widetilde{b}=b$, $\theta_t=t_h+h$ and $\overline{\theta}_t=t_h$.

\begin{lemma}\label{boundedness of BEM}
Suppose that Assumptions \ref{A1} and \ref{N1} hold with $\lambda_1>2\lambda_2$. Then for any $p\ge 1$ and $h\in (0,h_p]$ with$$h_p=1\wedge\frac{\lambda_1-2\lambda_2}{2p(1+\lambda_2)^p},$$ there exists a constant $C>0$ such that for all $n\ge 0$ and $i\in\mathbb{S_N}$,
\begin{align*}
\mathbb{E}|X_{nh}^{i,N, h}|^p\le {\rm e}^{-\gamma^*nh}\mathbb{E}|X_0^{i,N,h}|^p+C, 
\end{align*}
in case of $\mathbb{E}|X_0^{i,N,h}|^p<\infty$, where $\gamma^*=\frac{\lambda_1-2\lambda_2}{2+3(2\lambda_1-3\lambda_2)}$.
\end{lemma}
\begin{proof}
%Define $\mathcal{G}_t^{0,W}$ stands for the completion of $\mathcal{F}_0^{B}\times \mathcal{F}_t^W$

For any $n\in \mathbb{N}$, according to \eqref{BEM-dis}, we arrive at
\begin{align*}
|X_{(n+1)h}^{i,N, h}|^2=&|X_{nh}^{i,N, h}|^2-|X_{(n+1)h}^{i,N, h}-X_{nh}^{i,N, h}|^2+2\<X_{(n+1)h}^{i,N, h},X_{(n+1)h}^{i,N, h}-X_{nh}^{i,N, h}\>\notag\\
=&|X_{nh}^{i,N, h}|^2-|X_{(n+1)h}^{i,N, h}-X_{nh}^{i,N, h}|^2+2\<X_{(n+1)h}^{i,N, h}, b(X_{(n+1)h}^{i,N,h},\widetilde{\mu}_{nh}^{N,h})h\>\notag\\
&+2\<X_{(n+1)h}^{i,N, h}-X_{nh}^{i,N, h}, \sigma_0\Delta B_n^i+\sigma(X_{nh}^{i,N,h},\widetilde{\mu}_{nh}^{N,h})\Delta \overline B_n^i+\sigma_1\Delta W_n+\overline\sigma(X_{nh}^{i,N,h},\widetilde{\mu}_{nh}^{N,h})\Delta \overline{W}_n\>\notag\\
&+2\<X_{nh}^{i,N, h}, \sigma_0\Delta B_n^i+\sigma(X_{nh}^{i,N,h},\widetilde{\mu}_{nh}^{N,h})\Delta \overline B_n^i+\sigma_1\Delta W_n+\overline\sigma(X_{nh}^{i,N,h},\widetilde{\mu}_{nh}^{N,h})\Delta \overline{W}_n\>\notag\\
\le &|X_{nh}^{i,N, h}|^2+2h\<X_{(n+1)h}^{i,N, h}, b(X_{(n+1)h}^{i,N,h},\widetilde{\mu}_{nh}^{N,h})\>\notag\\
&+| \sigma_0\Delta B_n^i+\sigma(X_{nh}^{i,N,h},\widetilde{\mu}_{nh}^{N,h})\Delta \overline B_n^i+\sigma_1\Delta W_n+\overline\sigma(X_{nh}^{i,N,h},\widetilde{\mu}_{nh}^{N,h})\Delta \overline{W}_n|^2\notag\\
&+2\<X_{nh}^{i,N, h}, \sigma_0\Delta B_n^i+\sigma(X_{nh}^{i,N,h},\widetilde{\mu}_{nh}^{N,h})\Delta \overline B_n^i+\sigma_1\Delta W_n+\overline\sigma(X_{nh}^{i,N,h},\widetilde{\mu}_{nh}^{N,h})\Delta \overline{W}_n\>.
\end{align*}
By means of \eqref{B1} in Remark \ref{Remark-1.1}, together with Assumption \ref{N1}, we deduce that 
\begin{align*}
&(1+2K_1h)|X_{(n+1)h}^{i,N, h}|^2\\
&\le |X_{nh}^{i,N, h}|^2+\lambda_2h\mathbb{W}_1^2(\widetilde{\mu}_{nh}^{N,h},\delta_0)+2K_2h+4(\sigma_0^2+\sigma_1^2+\hat L^2)(|\Delta B_n^i|^2+|\Delta \overline B_n^i|^2+|\Delta W_n|^2+|\Delta \overline{W}_n|^2)\notag\\
&\quad+2\<X_{nh}^{i,N, h}, \sigma_0\Delta B_n^i+\sigma(X_{nh}^{i,N,h},\widetilde{\mu}_{nh}^{N,h})\Delta \overline B_n^i+\sigma_1\Delta W_n+\overline\sigma(X_{nh}^{i,N,h},\widetilde{\mu}_{nh}^{N,h})\Delta \overline{W}_n\>.
\end{align*}
%Since $\lambda_2>\lambda_3$, we arrive at
%\begin{align*}
%|X_{(n+1)h}^{i,N, h}|^2\le &|X_0^{i,N, h}|^2+2((n+1)h)(\lambda_1+\lambda_2)l_0^2+\lambda_3\int_0^{(n+1)h}\mathbb{W}_1^2(\widetilde{\mu}_{s_h}^{N,h},\delta_0){\rm d}s\notag\\
%&+h^{-1}\int_0^{(n+1)h}| (\sigma_0I_{d\times d}+\sigma(X_{s_h}^{i,N,h},\widetilde{\mu}_{s_h}^{N,h}))\Delta B_{s_h}^i+\overline\sigma(X_{s_h}^{i,N,h},\widetilde{\mu}_{s_h}^{N,h})\Delta W_{s_h}|^2{\rm d}s\notag\\
%&+2h^{-1}\int_0^{(n+1)h}\<X_{s_h}^{i,N, h}, (\sigma_0I_{d\times d}+\sigma(X_{s_h}^{i,N,h},\widetilde{\mu}_{s_h}^{N,h}))\Delta B_{s_h}^i+\overline\sigma(X_{s_h}^{i,N,h},\widetilde{\mu}_{s_h}^{N,h})\Delta W_{s_h}\>{\rm d}s.
%\end{align*}
%For any $p\ge 2$, using the Jensen inequality, yields
%\begin{align*}
%\mathbb{E}|X_{(n+1)h}^{i,N, h}|^p\le &C\mathbb{E}|X_0^{i,N, h}|^p+C((n+1)h)^{p/2}+C((n+1)h)^{p/2-1}\int_0^{(n+1)h}\mathbb{E}\mathbb{W}_1^p(\widetilde{\mu}_{s_h}^{N,h},\delta_0){\rm d}s\notag\\
%&+h^{-p/2}((n+1)h)^{p/2-1}\int_0^{(n+1)h}\mathbb{E}| (\sigma_0I_{d\times d}+\sigma(X_{s_h}^{i,N,h},\widetilde{\mu}_{s_h}^{N,h}))\Delta B_{s_h}^i+\overline\sigma(X_{s_h}^{i,N,h},\widetilde{\mu}_{s_h}^{N,h})\Delta W_{s_h}|^p{\rm d}s\notag\\
%&+Ch^{-p/2}((n+1)h)^{p/2-1}\int_0^{(n+1)h}\mathbb{E}\left(\int_{s_h}^{s_h+h}\<X_{s_h}^{i,N, h}, \sigma_0I_{d\times d}+\sigma(X_{s_h}^{i,N,h},\widetilde{\mu}_{s_h}^{N,h})\>{\rm d} B_s^i\right)^{p/2}{\rm d}s\notag\\
%&+Ch^{-p/2}((n+1)h)^{p/2-1}\int_0^{(n+1)h}\mathbb{E}\left(\int_{s_h}^{s_h+h}\<X_{s_h}^{i,N, h},\overline\sigma(X_{s_h}^{i,N,h},\widetilde{\mu}_{s_h}^{N,h})\>{\rm d} W_s\right)^{p/2}{\rm d}s.
%\end{align*}

For any integer $p\ge 1$, using the binomial theorem, we have
\begin{align}\label{lemma4.1-eq1}
(1+2K_1h)^{p}|X_{(n+1)h}^{i,N, h}|^{2p}\le &\left(|X_{nh}^{i,N, h}|^2+\lambda_2h\mathbb{W}_1^2(\widetilde{\mu}_{nh}^{N,h},\delta_0)\right)^{p}\notag\\
&+p\left(|X_{nh}^{i,N, h}|^2+\lambda_2h\mathbb{W}_1^2(\widetilde{\mu}_{nh}^{N,h},\delta_0)\right)^{p-1}U_{nh}^i\notag\\
&+{\mathbbm 1}_{\{p\ge 2\}}\sum_{k=0}^{p-2}C_{p}^k\left(|X_{nh}^{i,N, h}|^2+\lambda_2h\mathbb{W}_1^2(\widetilde{\mu}_{nh}^{N,h},\delta_0)\right)^k(U_{nh}^i)^{p-k}\notag\\
=&:\Gamma_{nh}^i({\bf X}_{nh}^{N,h})+\widetilde \Gamma_{nh}^i({\bf X}_{nh}^{N,h})+\overline \Gamma_{nh}^i({\bf X}_{nh}^{N,h}),
\end{align}
where 
\begin{align*}
U_{nh}^i=&2K_2h+4(\sigma_0^2+\sigma_1^2+\hat L^2)(|\Delta B_n^i|^2+|\Delta \overline B_n^i|^2+|\Delta W_n|^2+|\Delta \overline{W}_n|^2)\notag\\
&+2\<X_{nh}^{i,N, h}, \sigma_0\Delta B_n^i+\sigma(X_{nh}^{i,N,h},\widetilde{\mu}_{nh}^{N,h})\Delta \overline B_n^i+\sigma_1\Delta W_n+\overline\sigma(X_{nh}^{i,N,h},\widetilde{\mu}_{nh}^{N,h})\Delta \overline{W}_n\>,\notag\\
{\bf X}_{nh}^{N,h}=&(X_{nh}^{1,N,h},\cdots, X_{nh}^{N,N,h}).
\end{align*}
In the following, let us estimate the terms $\Gamma_{nh}^i, \widetilde \Gamma_{nh}^i$ and $\overline \Gamma_{nh}^i$ one by one. First, using the binomial theorem once again, combining the Jensen inequality and the Young inequality, yields
\begin{align*}
\Gamma_{nh}^i({\bf X}_{nh}^{N,h})=&\sum_{k=0}^{p}C_{p}^k |X_{nh}^{i,N, h}|^{2k}\left(\lambda_2h\mathbb{W}_1^2(\widetilde{\mu}_{nh}^{N,h},\delta_0)\right)^{p-k}\notag\\
\le &\sum_{k=0}^{p}C_{p}^k (\lambda_2h)^{p-k}\left(\frac{k}{p}|X_{nh}^{i,N, h}|^{2p}+\frac{p-k}{p}\mathbb{W}_1^{2p}(\widetilde{\mu}_{nh}^{N,h},\delta_0)\right).
\end{align*}
Hence, according to the fact that $X_{nh}^{1,N,h},\cdots, X_{nh}^{N,N,h}$ are identically distributed, we have
\begin{align}\label{lemma4.1-eq2}
\mathbb{E}[\Gamma_{nh}^i({\bf X}_{nh}^{N,h})|\mathcal F_0]\le (1+\lambda_2h)^{p}\mathbb{E}[|X_{nh}^{i,N, h}|^{2p}|\mathcal F_0].
\end{align}
By using the Young inequality, one can derive that
\begin{align}\label{lemma4.1-eq3}
\mathbb{E}[\widetilde \Gamma_{nh}^i({\bf X}_{nh}^{N,h})|\mathcal F_0]
\le&\frac{1}{8}(\lambda_1-2\lambda_2)ph\mathbb{E}[|X_{nh}^{i,N, h}|^{2p}|\mathcal F_0]+Ch.
\end{align}
Once again, we can also get that
\begin{align}\label{lemma4.1-eq4}
\mathbb{E}[\overline \Gamma_{nh}^i({\bf X}_{nh}^{N,h})|\mathcal F_0]\le&\frac{1}{8}(\lambda_1-2\lambda_2)ph\mathbb{E}[|X_{nh}^{i,N, h}|^{2p}|\mathcal F_0]+Ch.
\end{align}
Since $h\in(0,h_p]$, we can derive that $(1+\lambda_2h)^{p}\le 1+\lambda_2ph+\frac{1}{4}(\lambda_1-2\lambda_2)ph$. Substituting \eqref{lemma4.1-eq2}-\eqref{lemma4.1-eq4} into \eqref{lemma4.1-eq1}, it yields that
\begin{align}\label{lemma4.1-eq5}
(1+2pK_1h)\mathbb{E}[|X_{(n+1)h}^{i,N, h}|^{2p}|\mathcal F_0]\le&\left( (1+\lambda_2h)^{p}+\frac{1}{4}(\lambda_1-2\lambda_2)ph\right)\mathbb{E}[|X_{nh}^{i,N, h}|^{2p}|\mathcal F_0]+Ch\notag\\
\le& \left(1+\lambda_2ph+\frac{1}{2}(\lambda_1-2\lambda_2)ph\right)\mathbb{E}[|X_{nh}^{i,N, h}|^{2p}|\mathcal F_0]+Ch.
\end{align}
%{\color{red}where we need to set $h<\frac{2(\lambda_1-2\lambda_2)}{p(1+\lambda_3)^{p/2}}$. }
Hence,
\begin{align*}
\mathbb{E}[|X_{(n+1)h}^{i,N, h}|^{2p}|\mathcal F_0]
\le&\overline\gamma_h\mathbb{E}[|X_{nh}^{i,N, h}|^{2p}|\mathcal F_0]+Ch,
\end{align*}
which further implies 
\begin{align*}
\mathbb{E}[|X_{(n+1)h}^{i,N, h}|^{2p}|\mathcal F_0]
\le&(\overline\gamma_h)^{n+1}|X_0^{i,N, h}|^{2p}+\frac{Ch}{1-\overline\gamma_h},
\end{align*}
where
\begin{align}
\overline\gamma_h:= \frac{1+\lambda_1ph/2}{1+(2\lambda_1-3\lambda_2)ph}\in(0,1).
\end{align}
By employing the inequality: $a^r\le {\rm e}^{-(1-a)r}$ for $a,r>0$, yields
\begin{align*}
\mathbb{E}[|X_{(n+1)h}^{i,N, h}|^{2p}|\mathcal F_0]\le {\rm e}^{-\hat\gamma_p(n+1)h}|X_0^{i,N, h}|^{2p}+\frac{C}{\hat\gamma_{p}},
\end{align*}
where $\hat\gamma_p=\frac{1-\overline\gamma_h}{h}=\frac{3}{2}\cdot\frac{(\lambda_1-2\lambda_2)p}{1+(2\lambda_1-3\lambda_2)p}>0.$ 
For $p\in [1,2)$, it follows from the H\"older inequality and the inequality:$(a+b)^r\le a^r+b^r$ for $a,b>0$ and $r\in(0,1]$ that 
\begin{align*}
\mathbb{E}[|X_{(n+1)h}^{i,N, h}|^p|\mathcal F_0]\le \left(\mathbb{E}[|X_{(n+1)h}^{i,N, h}|^{2}|\mathcal F_0]\right)^{\frac{p}{2}}\le {\rm e}^{-\gamma_p^*(n+1)h}|X_0^{i,N, h}|^p+C.
\end{align*}
where $\gamma_p^*=\frac{3}{2}\cdot\frac{(\lambda_1-2\lambda_2)p}{2+(2\lambda_1-3\lambda_2)(p+2)}$. For $p> 2$ which is not an even number, 
\begin{align*}
\mathbb{E}[|X_{(n+1)h}^{i,N, h}|^p|\mathcal F_0]\le \left(\mathbb{E}[|X_{(n+1)h}^{i,N, h}|^{2\lceil p/2\rceil}|\mathcal F_0]\right)^{\frac{p}{2\lceil p/2\rceil}}
\le {\rm e}^{-\gamma_p^*(n+1)h}|X_0^{i,N, h}|^p+C,
\end{align*}
where $\lceil \cdot\rceil$ means the ceiling function. Hence, it holds that for any $p\ge 1$,
\begin{align}\label{lemma4.1-eq6}
\mathbb{E}[|X_{(n+1)h}^{i,N, h}|^p|\mathcal F_0]\le {\rm e}^{-\gamma^*(n+1)h}|X_0^{i,N, h}|^p+C, ~\forall n\in \mathbb{N},
\end{align}
where $\gamma^*=\frac{\lambda_1-2\lambda_2}{2+3(2\lambda_1-3\lambda_2)}>0$.
Consequently, recall that $X_0^{i,N, h}=X_0^i$, it follows from the property of conditional expectation that 
\begin{align*}
\mathbb{E}|X_{(n+1)h}^{i,N, h}|^p\le {\rm e}^{-\gamma^*(n+1)h}\mathbb{E}|X_0^{i,N,h}|^p+C, ~\forall n\in \mathbb{N}.
\end{align*}
\end{proof}

%\begin{lemma}\label{BEM-lemma2}
%Suppose that Assumptions \ref{A1}, \ref{A12}, \ref{S1}, \ref{N1} hold with $\lambda_1>2\lambda_2$ and that $\mathbb{E}|X_0^i|^{2(l+1)^2}<\infty$. Then for any $h\in (0,\overline h]$ with$$\overline h=1\wedge\frac{\lambda_1-2\lambda_2}{4(l+1)^2(1+\lambda_3)^{2(l+1)^2}},$$ there exists a constant $C>0$ such that for all $n\ge 0$,
%\begin{align*}
%\mathbb{E}|b(X_t^{i,N,h},\widetilde{\mu}_t^{N,h})-b(X_{t_h+h}^{i,N,h},\widetilde{\mu}_{t_h}^{N,h})|\le C\left(1+\mathbb{E}|X_0^{i,N,h}|^{2(l+1)^2}\right)h^{1/2}.
%\end{align*}
%\end{lemma}

As a direct application of Theorem \ref{main theorem}, we can get the uniform-in-time discretization error bounds for the backward EM method.
\begin{theorem}\label{BEM-error-bound}
Suppose that Assumptions \ref{A1}, \ref{A2}, \ref{S1}, \ref{N1} hold with $\lambda_1>2\lambda_2$ and that $\sigma_0\sigma_1\neq 0$. Let 
$$\overline h=1\wedge\frac{\lambda_1-2\lambda_2}{4(l+1)^2(1+\lambda_2)^{2(l+1)^2}},$$
then there exist constants $C>0$ such that for any $\lambda_2\in [0,\lambda_2^*), L\in[0,L^*]$, $h\in (0,\overline h]$, $t>0$ and $i\in\mathbb{S}_N$,
\begin{align*}
\mathcal{W}_1\big(\mathcal{L}(\mu_t^i),\mathcal{L}(\nu_t^{i,N,h})\big)\le C\left\{{\rm e}^{-\lambda_0t}\mathbb{W}_1(\mu,\nu)+\psi(N)+\left(1+\mathbb{E}|X_0^{i,N,h}|^{2(l+1)^2}\right)h^{1/2}\right\}
\end{align*}
in case of $\mathbb{E}|X_0^i|^{q}<\infty$ for some $q>1$ and $\mathbb{E}|X_0^{i,N,h}|^{2(l+1)^2}<\infty$, where $\mu_t^i=\mathcal{L}(X_t^i\vert \mathcal{F}_t^W)$ stands for the regular conditional distribution of $X_t$, determined by \eqref{NIPS}, with the initial distribution $\mathcal{L}(X_0^i)=\mu$, $\nu_t^{i,N,h}=\mathcal{L}(X_t^{i,N,h}\vert \mathcal{F}_t^W)$ stands for the regular conditional distribution of $X_t^{i,N,h}$, determined by \eqref{BEM}, with the initial distribution $\mathcal{L}(X_0^{i,N,h})=\nu$, $\lambda_2^*, L^*$, and $\lambda_0$ were given in Theorem  \ref{main theorem}.
\end{theorem}

\begin{proof}
Note that for any $t\ge 0$, there always exists $n\in \mathbb{N}$ such that $t\in[nh,(n+1)h)$, and then $t_h=nh$. For any $p>0$, from \eqref{BEM} and \eqref{b_growth}, using Lemma \ref{boundedness of BEM}, 
\begin{align}\label{lemma_4.2-1}
\mathbb{E}|X_t^{i,N, h}|^p\le &C\mathbb{E}|X_{nh}^{i,N, h}|^p+Ch^p\mathbb{E}|b(X_{(n+1)h}^{i,N,h},\widetilde{\mu}_{nh}^{N,h})|^p+C|\sigma_0|^p\mathbb{E}|B_t^i-B_{nh}^i|^p\notag\\
&+C\mathbb{E}\left[\left|\sigma(X_{nh}^{i,N,h},\widetilde{\mu}_{nh}^{N,h})\right|^p\mathbb{E}\left(|\overline B_t^i-\overline B_{nh}^i|^p\big|\mathcal{F}_{nh}\right)\right]+C|\sigma_1|^p\mathbb{E}|W_t-W_{nh}|^p\notag\\
&+C\mathbb{E}\left[|\overline\sigma(X_{nh}^{i,N,h},\widetilde{\mu}_{nh}^{N,h})|^p\mathbb{E}\left(|\overline{W}_t-\overline W_{nh}|^p\big|\mathcal{F}_{nh}\right)\right]\notag\\
\le &C\left(1+\mathbb{E}|X_{(n+1)h}^{i,N,h}|^{p(l+1)}+\mathbb{E}|X_{nh}^{i,N,h}|^{p}\right)\notag\\
\le &C\left(1+\mathbb{E}|X_0^{i,N,h}|^{p(l+1)}\right).
\end{align}
Based on \eqref{BEM}, combining \eqref{b_growth} and Assumption \ref{N1}, utilizing the fact that $X_t^{1,N,h}, \cdots, X_t^{N,N,h}$ are identically distributed, it yields that for $t\in(nh,(n+1)h)$,
\begin{align}\label{4.2.1}
\mathbb{E}\vert X_t^{i,N,h}-X_{nh}^{i,N,h}\vert^2\le& Ch\left(\mathbb{E}\vert b(X_{(n+1)h}^{i,N,h},\widetilde{\mu}_{nh}^{N,h})\vert^2h+\sigma_0^2+\sigma_1^2+\mathbb{E}\vert\sigma(X_{nh}^{i,N,h},\widetilde{\mu}_{nh}^{N,h})\vert^2+\mathbb{E}\vert\overline\sigma(X_{nh}^{i,N,h},\widetilde{\mu}_{nh}^{N,h})\vert^2\right)\notag\\
\le&Ch\left(1+\mathbb{E}|X_{(n+1)h}^{i,N,h}|^{2(l+1)}+\mathbb{E}|X_{nh}^{i,N,h}|^2\right),
\end{align}
according to the continuity of $X_t^{i,N,h}$, it implies
\begin{align}\label{4.2.2}
\mathbb{E}\vert X_{(n+1)h}^{i,N,h}-X_{nh}^{i,N,h}\vert^2\le&Ch\left(1+\mathbb{E}|X_{(n+1)h}^{i,N,h}|^{2(l+1)}+\mathbb{E}|X_{nh}^{i,N,h}|^2\right).
\end{align}
With the help of the triangle and H\"older inequalities, applying Assumptions \ref{A1} and \ref{S1}, it follows from \eqref{4.2.1} and \eqref{4.2.2} that
\begin{align*}
&\mathbb{E}|b(X_t^{i,N,h},\widetilde{\mu}_t^{N,h})-b(X_{(n+1)h}^{i,N,h},\widetilde{\mu}_{nh}^{N,h})|\\
&\le \mathbb{E}|b(X_t^{i,N,h},\widetilde{\mu}_t^{N,h})-b(X_{(n+1)h}^{i,N,h},\widetilde{\mu}_t^{N,h})|+\mathbb{E}|b(X_{(n+1)h}^{i,N,h},\widetilde{\mu}_t^{N,h})-b(X_{(n+1)h}^{i,N,h},\widetilde{\mu}_{nh}^{N,h})|\notag\\
&\le C\left(1+\mathbb{E}|X_t^{i,N,h}|^{2l}+\mathbb{E}|X_{(n+1)h}^{i,N,h}|^{2l}\right)^{1/2}\left(\mathbb{E}\vert X_t^{i,N,h}-X_{(n+1)h}^{i,N,h}\vert^2\right)^{1/2}+\lambda_2\mathbb{E}\mathbb{W}_1(\widetilde{\mu}_t^{N,h},\widetilde{\mu}_{nh}^{N,h})\notag\\
&\le C\left(1+\mathbb{E}|X_t^{i,N,h}|^{2l}+\mathbb{E}|X_{(n+1)h}^{i,N,h}|^{2l}\right)^{1/2}\left(\mathbb{E}\vert X_t^{i,N,h}-X_{nh}^{i,N,h}\vert^2+\mathbb{E}\vert X_{nh}^{i,N,h}-X_{(n+1)h}^{i,N,h}\vert^2\right)^{1/2}\notag\\
&\quad+\lambda_2\mathbb{E}|X_t^{i,N,h}-X_{nh}^{i,N,h}|\notag\\
&\le C\left(1+\mathbb{E}|X_t^{i,N,h}|^{2l}+\mathbb{E}|X_{(n+1)h}^{i,N,h}|^{2(l+1)}+\mathbb{E}|X_{nh}^{i,N,h}|^2\right)h^{1/2}.
\end{align*}
Using Lemma \ref{boundedness of BEM} and \eqref{lemma_4.2-1}, leads to
\begin{align}\label{BEM-b-b_tilde}
\mathbb{E}|b(X_t^{i,N,h},\widetilde{\mu}_t^{N,h})-b(X_{(n+1)h}^{i,N,h},\widetilde{\mu}_{nh}^{N,h})|\le C\left(1+\mathbb{E}|X_0^{i,N,h}|^{2(l+1)^2}\right)h^{1/2}.
\end{align}
Applying Theorem \ref{main theorem} with $X_t^{i,N}=X_t^{i,N,h}$, $\theta_t=t_h+h$, $\overline \theta_t=t_h$, together with \eqref{BEM-b-b_tilde}, yields that 
\begin{align*}
\mathcal{W}_1\big(\mathcal{L}(\mu_t^i),\mathcal{L}(\nu_t^{i,N,h})\big)\le C\left({\rm e}^{-\lambda_0t}\mathbb{W}_1(\mu,\nu)+\psi(N)+\left(1+\mathbb{E}|X_0^{i,N,h}|^{2(l+1)^2}\right)h^{1/2}\right).
\end{align*}
\end{proof}

%%%-------------------------------------------------------------------------------------------------------------------------------------------------------

\subsection{Tamed EM method}
For a stepsize $h>0$, the tamed EM method related to the McKean-Vlasov SDE with common noise \eqref{NIPS} is defined as
\begin{align}\label{TEM-dis}
\overline X_{(n+1)h}^{i,N,h}=&\overline X_{nh}^{i,N,h}+\frac{b(\overline X_{nh}^{i,N,h},\widetilde{\mu}_{nh}^{\overline X,N,h})}{1+\alpha\sqrt{h}|\overline X_{nh}^{i,N,h}|^l}h+\sigma_0\Delta B_n^i+\sigma(\overline X_{nh}^{i,N,h},\widetilde{\mu}_{nh}^{\overline X,N,h})\Delta \overline B_n^i\notag\\
&+\sigma_1\Delta W_n+\overline\sigma(\overline X_{nh}^{i,N,h},\widetilde{\mu}_{nh}^{\overline X,N,h})\Delta \overline{W}_n,
\end{align}
with initial data $\overline X_0^{i,N,h}$ for all $i\in\mathbb{S}_N$, where $\widetilde{\mu}_{nh}^{\overline X,N,h}=\frac{1}{N}\sum_{j=1}^N \delta_{\overline X_{nh}^{j,N,h}}$, $n\in\mathbb{N}$, $\alpha\in (0,1]$, $l$ being defined in Assumption \ref{S1}. If $l=0$, i.e., the drift coefficient $b$ is of linear growth, then we can take $\alpha=0$, and equation \eqref{TEM-dis} degenerates to the standard EM method.

For $t_h=\lfloor t/h \rfloor h$, the continuous version of the tamed EM method for the McKean-Vlasov SDE with common noise \eqref{NIPS} as
\begin{align}\label{TEM}
{\rm d}\overline X_t^{i,N,h}=&\frac{b(\overline X_{t_h}^{i,N,h},\widetilde{\mu}_{t_h}^{\overline X,N,h})}{1+\alpha\sqrt{h}|\overline X_{t_h}^{i,N,h}|^l}{\rm d}t+\sigma_0{\rm d}B_t^i+\sigma(\overline X_{t_h}^{i,N,h},\widetilde{\mu}_{t_h}^{\overline X,N,h}){\rm d}\overline B_t^i\notag\\
&+\sigma_1{\rm d}W_t+\overline\sigma(\overline X_{t_h}^{i,N,h},\widetilde{\mu}_{t_h}^{\overline X,N,h}){\rm d}\overline{W}_t.
\end{align}
Transparently, \eqref{TEM} can be incorporated into the framework of \eqref{IPS} by setting $\widetilde b(x,\mu)=\frac{b(x,\mu)}{1+\alpha\sqrt{h}|x|^l}$ and $\theta_t=\overline \theta_t=t_h$.

\begin{assumption}\label{N2-1}
There exist constants $\overline\lambda_1, \overline\lambda_2, \overline\lambda_3>0$ such that for all $x\in\mathbb{R}^d$ and $\mu\in\mathcal{P}_1(\mathbb{R}^d)$,
\begin{align*}
\<x, b(x,\mu)\>\le -\overline\lambda_1|x|^{l+2}+\overline\lambda_2\mathbb{W}_1^2(\mu,\delta_0)+\overline\lambda_3,
\end{align*}
where $l$ was proposed in Assumption \ref{S1}.
\end{assumption}

\begin{lemma}\label{TEM-lemma1}
Suppose that Assumptions \ref{A1}, \ref{S1}-\ref{N2-1} hold with $\overline\lambda_1>2\overline\lambda_2$. Let
\begin{align*}
h^*:=1\wedge \frac{1}{\overline\lambda_1}\wedge\frac{4\alpha^2(\overline\lambda_1-2\overline\lambda_2)^2}{9\hat K^4(1+2\alpha)^2},
\end{align*}
where $\hat K=2K+\lambda_2+|b(0,\delta_0)|$ was defined in \eqref{b_growth}. Then, for any $p\ge 1$ and $h\in(0,h^*]$, the solution of \eqref{TEM} satisfies
\begin{align}\label{TEM-lemma1-EQ}
\mathbb{E}|\overline X_{nh}^{i,N, h}|^p\le C(1+\mathbb{E}|\overline X_0^{i,N,h}|^p), ~n\in\mathbb{N},
\end{align}
in case of $\mathbb{E}|\overline X_0^{i,N,h}|^p<\infty$.
\end{lemma}
\begin{proof}
For $n\ge 0$, based on \eqref{TEM-dis}, 
\begin{align*}
&|\overline X_{(n+1)h}^{i,N,h}|^2\notag\\
=&|\overline X_{nh}^{i,N,h}|^2+\left(2\left\langle \overline X_{nh}^{i,N,h},\frac{b(\overline X_{nh}^{i,N,h},\widetilde{\mu}_{nh}^{\overline X,N,h})}{1+\alpha\sqrt{h}|\overline X_{nh}^{i,N,h}|^l}\right\rangle+\left|\frac{b(\overline X_{nh}^{i,N,h},\widetilde{\mu}_{nh}^{\overline X,N,h})}{1+\alpha\sqrt{h}|\overline X_{nh}^{i,N,h}|^l}\right|^2h\right)h\notag\\
&+\left|\sigma_0\Delta B_n^i+\sigma(\overline X_{nh}^{i,N,h},\widetilde{\mu}_{nh}^{\overline X,N,h})\Delta \overline B_n^i+\sigma_1\Delta W_n+\overline\sigma(\overline X_{nh}^{i,N,h},\widetilde{\mu}_{nh}^{\overline X,N,h})\Delta \overline{W}_n\right|^2\notag\\
&+2\left\langle \overline X_{nh}^{i,N,h}+\frac{b(\overline X_{nh}^{i,N,h},\widetilde{\mu}_{nh}^{\overline X,N,h})}{1+\alpha\sqrt{h}|\overline X_{nh}^{i,N,h}|^l}h,\sigma_0\Delta B_n^i+\sigma(\overline X_{nh}^{i,N,h},\widetilde{\mu}_{nh}^{\overline X,N,h})\Delta \overline B_n^i+\sigma_1\Delta W_n+\overline\sigma(\overline X_{nh}^{i,N,h},\widetilde{\mu}_{nh}^{\overline X,N,h})\Delta \overline{W}_n\right\rangle\notag\\
=:&|\overline X_{nh}^{i,N,h}|^2+\Lambda^i(\overline{\bf X}_{nh}^{N,h})h+\widetilde \Lambda^i(\overline{\bf X}_{nh}^{N,h})+\overline \Lambda^i(\overline{\bf X}_{nh}^{N,h}),
\end{align*}
where $\overline{\bf X}_{nh}^{N,h}=(\overline X_{nh}^{1,N,h},\cdots,\overline X_{nh}^{N,N,h})$.

By using Assumption \ref{N2-1} and \eqref{b_growth}, it holds that 
\begin{align*}
&2\left\langle x,\frac{b(x,\mu)}{1+\alpha\sqrt{h}|x|^l}\right\rangle+\left|\frac{b(x,\mu)}{1+\alpha\sqrt{h}|x|^l}\right|^2h\notag\\
%&\le  -\frac{|x|^l}{1+\sqrt{h}|x|^l}\left(2\overline\lambda_1-3\hat K^2\sqrt{h}\right)|x|^2+2\overline\lambda_2+3\hat K^2h+3\hat K^2h\mathbb{W}_1^2(\mu,\delta_0)\notag\\
&\le  -\frac{|x|^l}{1+\alpha\sqrt{h}|x|^l}\left(2\overline\lambda_1-\frac{3\hat K^2\sqrt{h}}{\alpha}\right)|x|^2\mathbbm{1}_{\{|x|\ge  1\}}+2\overline\lambda_2\mathbb{W}_1^2(\mu,\delta_0)+2\overline\lambda_3+3\hat K^2h+3\hat K^2h\mathbb{W}_1^2(\mu,\delta_0)\notag\\
&\le  -\frac{1}{1+\alpha\sqrt{h}}\left(2\overline\lambda_1-\frac{3\hat K^2\sqrt{h}}{\alpha}\right)|x|^2+2\overline\lambda_1+2\overline\lambda_3+(2\overline\lambda_2+3\hat K^2h)\mathbb{W}_1^2(\mu,\delta_0)+3\hat K^2h\notag\\
%&\le  -\frac{1}{1+\sqrt{h}}\left(2\overline\lambda_1-9\hat K^2\sqrt{h}\right)|x|^2+2\overline\lambda_2+3\hat K^2h(\mathbb{W}_1^2(\mu,\delta_0)-|x|^2)+2\overline\lambda_1\notag\\
&\le  -\frac{\kappa}{4}|x|^2+(2\overline\lambda_2+3\hat K^2h)(\mathbb{W}_1^2(\mu,\delta_0)-|x|^2)+2\overline\lambda_1+2\overline\lambda_3+3\hat K^2h,
\end{align*}
where $\kappa=2(2\overline\lambda_1-4\overline\lambda_2-\frac{3\hat K^2\sqrt{h}}{\alpha}-6\hat K^2h)>0$ for $h\in(0,h^*]$, which gives
\begin{align*}
\Lambda^i(\overline{\bf X}_{nh}^{N,h})&\le  -\frac{\kappa}{4}|\overline X_{nh}^{i,N,h}|^2+(2\overline\lambda_2+3\hat K^2h)\left(\mathbb{W}_1^2(\widetilde{\mu}_{nh}^{\overline X,N,h},\delta_0)-|\overline X_{nh}^{i,N,h}|^2\right)+2\overline\lambda_1+2\overline\lambda_3+3\hat K^2h.
\end{align*}
Then
\begin{align*}
|\overline X_{(n+1)h}^{i,N,h}|^2\le&\left(1-\frac{\kappa}{4}h-2\overline\lambda_2h-3\hat K^2h^2\right)|\overline X_{nh}^{i,N,h}|^2+(2\overline\lambda_2+3\hat K^2h)h\mathbb{W}_1^2(\widetilde{\mu}_{nh}^{\overline X,N,h},\delta_0)\notag\\
&+(2\overline\lambda_1+2\overline\lambda_3+3\hat K^2h)h+\widetilde \Lambda^i(\overline{\bf X}_{nh}^{N,h})+\overline \Lambda^i(\overline{\bf X}_{nh}^{N,h}),
\end{align*}
where the factor $1-\frac{\kappa}{4}h-2\overline\lambda_2h-3\hat K^2h^2$ is positive by taking $h\in(0,h^*]$. For any integer $p\ge 3$,
\begin{align}\label{lemma4.3111111}
|\overline X_{(n+1)h}^{i,N,h}|^{2p}\le&\left(\left(1-\frac{\kappa}{4}h-2\overline\lambda_2h-3\hat K^2h^2\right)|\overline X_{nh}^{i,N,h}|^2+(2\overline\lambda_2h+3\hat K^2h^2)\mathbb{W}_1^2(\widetilde{\mu}_{nh}^{\overline X,N,h},\delta_0)\right)^p\notag\\
&+p\left(\left(1-\frac{\kappa}{4}h-2\overline\lambda_2h-3\hat K^2h^2\right)|\overline X_{nh}^{i,N,h}|^2+(2\overline\lambda_2h+3\hat K^2h^2)\mathbb{W}_1^2(\widetilde{\mu}_{nh}^{\overline X,N,h},\delta_0)\right)^{p-1}\notag\\
&\quad\times\left((2\overline\lambda_1+2\overline\lambda_3+3\hat K^2h)h+\widetilde \Lambda^i(\overline{\bf X}_{nh}^{N,h})+\overline \Lambda^i(\overline{\bf X}_{nh}^{N,h})\right)\notag\\
&+\sum_{k=0}^{p-2}C_p^k\left(\left(1-\frac{\kappa}{4}h-2\overline\lambda_2h-3\hat K^2h^2\right)|\overline X_{nh}^{i,N,h}|^2+(2\overline\lambda_2h+3\hat K^2h^2)\mathbb{W}_1^2(\widetilde{\mu}_{nh}^{\overline X,N,h},\delta_0)\right)^k\notag\\
&\quad\times\left((2\overline\lambda_1+2\overline\lambda_3+3\hat K^2h)h+\widetilde \Lambda^i(\overline{\bf X}_{nh}^{N,h})+\overline \Lambda^i(\overline{\bf X}_{nh}^{N,h})\right)^{p-k}\notag\\
=&: \Theta^i(\overline{\bf X}_{nh}^{N,h})+\widehat\Theta^i(\overline{\bf X}_{nh}^{N,h})+\overline\Theta^i(\overline{\bf X}_{nh}^{N,h}).
\end{align}
Note that $\overline{X}_{nh}^{1,N,h},\dots, \overline{X}_{nh}^{N,N,h}$ are identical distribution, the binomial theorem and the Young inequality yield that
\begin{align}\label{lemma4.3111112}
&\mathbb{E}\left\{\Theta^i(\overline{\bf X}_{nh}^{N,h})\Big\vert\mathcal{F}_0\right\}\notag\\
&= \mathbb{E}\left\{\sum_{k=0}^{p}C_p^k\left(1-\frac{\kappa}{4}h-2\overline\lambda_2h-3\hat K^2h^2\right)^k|\overline X_{nh}^{i,N,h}|^{2k}\left((2\overline\lambda_2h+3\hat K^2h^2)\mathbb{W}_1^2(\widetilde{\mu}_{nh}^{\overline X,N,h},\delta_0)\right)^{p-k}\Big\vert\mathcal{F}_0\right\}\notag\\
&\le \sum_{k=0}^{p}C_p^k\left(1-\frac{\kappa}{4}h-2\overline\lambda_2h-3\hat K^2h^2\right)^k(2\overline\lambda_2h+3\hat K^2h^2)^{p-k}\mathbb{E}\left\{|\overline X_{nh}^{i,N,h}|^{2p}\vert\mathcal{F}_0\right\}\notag\\
&=\left(1-\frac{\kappa}{4}h\right)^{p}\mathbb{E}\left\{|\overline X_{nh}^{i,N,h}|^{2p}\vert\mathcal{F}_0\right\}\notag\\
&\le\left(1-\frac{\kappa}{4}h\right)\mathbb{E}\left\{|\overline X_{nh}^{i,N,h}|^{2p}\vert\mathcal{F}_0\right\},
\end{align}
where the last display is hold thanks to $1-\frac{\kappa}{4}h\in(0,1)$. Due to $$\mathbb{E}\left[\Delta B_n^i\vert\mathcal{F}_{nh}\right]=\mathbb{E}\left[\Delta \overline B_n^i\vert\mathcal{F}_{nh}\right]=\mathbb{E}\left[\Delta W_n\vert\mathcal{F}_{nh}\right]=\mathbb{E}\left[\Delta \overline W_n\vert\mathcal{F}_{nh}\right]=0,$$ $\mathbb{E}\left[|\Delta B_n^i|^2\vert\mathcal{F}_{nh}\right]=\mathbb{E}\left[|\Delta W_n|^2\vert\mathcal{F}_{nh}\right]=dh$, $\mathbb{E}\left[|\Delta W_n|^2\vert\mathcal{F}_{nh}\right]=m_1h$, and $\mathbb{E}\left[|\Delta \overline W_n|^2\vert\mathcal{F}_{nh}\right]=m_2h$, it follows from the tower property of the conditional expectation, Assumption \ref{N1} and the Young inequality that there is a constant $C>0$ such that
\begin{align}\label{lemma4.311113}
&\mathbb{E}\left\{\widehat\Theta^i(\overline{\bf X}_{nh}^{N,h})\Big\vert\mathcal{F}_0\right\}\notag\\&=p\mathbb{E}\Bigg\{\left(\left(1-\frac{\kappa}{4}h-2\overline\lambda_2h-3\hat K^2h^2\right)|\overline X_{nh}^{i,N,h}|^2+(2\overline\lambda_2h+3\hat K^2h^2)\mathbb{W}_1^2(\widetilde{\mu}_{nh}^{\overline X,N,h},\delta_0)\right)^{p-1}\notag\\
&\quad\qquad\times\left((2\overline\lambda_1+2\overline\lambda_3+3\hat K^2h)h+\mathbb{E}\left[\widetilde \Lambda^i(\overline{\bf X}_{nh}^{N,h})\vert\mathcal{F}_{nh}\right]+\mathbb{E}\left[\overline \Lambda^i(\overline{\bf X}_{nh}^{N,h})\vert\mathcal{F}_{nh}\right]\right)\Big\vert\mathcal{F}_0\Bigg\}\notag\\
&\le Ch\mathbb{E}\Bigg\{\left(\left(1-\frac{\kappa}{4}h-2\overline\lambda_2h-3\hat K^2h^2\right)|\overline X_{nh}^{i,N,h}|^2+(2\overline\lambda_2h+3\hat K^2h^2)\mathbb{W}_1^2(\widetilde{\mu}_{nh}^{\overline X,N,h},\delta_0)\right)^{p-1}\Big\vert\mathcal{F}_0\Bigg\}\notag\\
&\le Ch+\frac{\kappa h}{16}\mathbb{E}\left\{|\overline X_{nh}^{i,N,h}|^{2p}\vert\mathcal{F}_0\right\},
\end{align}
and similarly, together with the fact that the conditional expectations of the increments $|\Delta B_n^i|^{p-k}$, $|\Delta \overline B_n^i|^{p-k}$, $|\Delta W_n|^{p-k}$ and $|\Delta \overline W_n|^{p-k}$ contributes at least the order $h$ for $k\le p-2$, and the young inequality, leads to 
\begin{align}\label{lemma4.3111114}
\mathbb{E}\left\{\overline\Theta^i(\overline{\bf X}_{nh}^{N,h})\Big\vert\mathcal{F}_0\right\}\le Ch+\frac{\kappa h}{16}\mathbb{E}\left\{|\overline X_{nh}^{i,N,h}|^{2p}\vert\mathcal{F}_0\right\}.
\end{align}
Combining \eqref{lemma4.3111111}-\eqref{lemma4.3111114}, it follows that
\begin{align}\label{lemma4.3111115}
\mathbb{E}\left(|\overline X_{(n+1)h}^{i,N,h}|^{2p}\vert \mathcal{F}_0\right)\le\left(1-\frac{\kappa}{8}h\right)\mathbb{E}\left\{|\overline X_{nh}^{i,N,h}|^{2p}\vert\mathcal{F}_0\right\}+Ch,
\end{align}
hence, via an inductive argument, we arrive at
\begin{align*}
\mathbb{E}[|\overline X_{(n+1)h}^{i,N, h}|^{2p}|\mathcal F_0]
\le&\left(1-\frac{\kappa}{8}h\right)^{n+1}|\overline X_0^{i,N,h}|^{2p}+\frac{Ch}{\frac{\kappa}{8}h}.
\end{align*}
By employing the inequality: $a^r\le {\rm e}^{-(1-a)r}$ for $a,r>0$, yields
\begin{align*}
\mathbb{E}[|\overline X_{(n+1)h}^{i,N, h}|^{2p}|\mathcal F_0]\le {\rm e}^{-\kappa(n+1)h/8}|\overline X_0^{i,N,h}|^{2p}+C.
\end{align*}
For $p\in [1,6)$, it follows from the H\"older inequality and the inequality:$(a+b)^r\le a^r+b^r$ for $a,b>0$ and $r\in(0,1)$ that 
\begin{align*}
\mathbb{E}[|\overline X_{(n+1)h}^{i,N, h}|^p|\mathcal F_0]\le \left(\mathbb{E}[|\overline X_{(n+1)h}^{i,N, h}|^6|\mathcal F_0]\right)^{p/6}\le {\rm e}^{-\kappa(n+1)h/48}|\overline X_0^{i,N,h}|^p+C.
\end{align*}
For $p>6$ which is not an even number, 
\begin{align*}
\mathbb{E}[|\overline X_{(n+1)h}^{i,N, h}|^p|\mathcal F_0]\le \left(\mathbb{E}|\overline X_{(n+1)h}^{i,N, h}|^{2\lceil p/2\rceil}\right)^{\frac{p}{2\lceil p/2\rceil}}
\le {\rm e}^{-3\kappa(n+1)h/32}|\overline X_0^{i,N,h}|^p+C.
\end{align*}
Hence, it holds that for any $p\ge 1$,
\begin{align}\label{lemma4.1-eq6}
\mathbb{E}[|\overline X_{(n+1)h}^{i,N, h}|^p|\mathcal F_0]\le {\rm e}^{-\kappa(n+1)h/48}|\overline X_0^{i,N, h}|^p+C\le |\overline X_0^{i,N,h}|^p+C, ~\forall n\in \mathbb{N}.
\end{align}
Consequently, it follows from the property of conditional expectation that 
\begin{align*}
\mathbb{E}|\overline X_{(n+1)h}^{i,N, h}|^p\le C(1+\mathbb{E}|\overline X_0^{i,N,h}|^p), ~\forall n\in \mathbb{N}.
\end{align*}

\end{proof}

The uniform-in-time error estimate for the tamed EM method can also be obtained follows Theorem \ref{main theorem} and Lemma \ref{TEM-lemma1}.
\begin{theorem}\label{TEM-error-bound}
Let $\sigma_0\sigma_1\neq 0$. Suppose that Assumptions \ref{A1}, \ref{A2}, \ref{S1}-\ref{N2-1} hold with $\lambda_1>2\lambda_2$ and $\overline\lambda_1>2\overline\lambda_2$. Then there exists $C>0$ such that for any $\lambda_2\in [0,\lambda_2^*), L\in[0,L^*]$, $h\in (0,h^*]$,  $t>0$ and $i\in\mathbb{S}_N$,
\begin{align*}
\mathcal{W}_1\big(\mathcal{L}(\mu_t^i),\mathcal{L}(\overline \nu_t^{i,N,h})\big)\le C\left\{{\rm e}^{-\lambda_0t}\mathbb{W}_1(\mu,\nu)+\psi(N)+\left(1+\mathbb{E}|\overline X_0^{i,N,h}|^{2(l+1)^2}\right)h^{1/2}\right\}
\end{align*}
in case of $\mathbb{E}|X_0^i|^{q}<\infty$ for some $q>1$ and $\mathbb{E}|\overline X_0^{i,N,h}|^{2(l+1)^2}<\infty$, where $h^*$ was defined in Lemma \ref{TEM-lemma1}, $\mu_t^i=\mathcal{L}(X_t^i\vert \mathcal{F}_t^W)$ stands for the regular conditional distribution of $X_t$, determined by \eqref{NIPS}, with the initial distribution $\mathcal{L}(X_0^i)=\mu$, $\overline\nu_t^{i,N,h}=\mathcal{L}(\overline X_t^{i,N,h}\vert \mathcal{F}_t^W)$ stands for the regular conditional distribution of $\overline X_t^{i,N,h}$, determined by \eqref{TEM}, with the initial distribution $\mathcal{L}(\overline X_0^{i,N,h})=\nu$, $\lambda_2^*, L^*$, and $\lambda_0$ were given in Theorem  \ref{main theorem}, $h^*$ was defined in Lemma \ref{TEM-lemma1}.
\end{theorem}
\begin{proof}
For $t\ge 0$, there always exists a constant $n\in\mathbb{N}$ such that $t_h=\lfloor t/h \rfloor h=nh$ and $t\in[nh,(n+1)h)$, Note that $\overline X_t^{1,N,h}, \cdots, \overline X_t^{N,N,h}$ are identically distributed, and thereby $\mathbb{E}\mathbb{W}_1^p(\widetilde{\mu}_{nh}^{\overline X,N,h},\delta_0)\le \mathbb{E}|\overline X_{nh}^{i,N,h}|^p$ for any $i\in\mathbb{S}_N$, then for any $p>0$, from \eqref{TEM} and \eqref{b_growth}, using Lemma \ref{TEM-lemma1}, 
\begin{align}\label{Eq-lemma_4.4-1}
\mathbb{E}|\overline X_t^{i,N, h}|^p\le &C\mathbb{E}|\overline X_{nh}^{i,N, h}|^p+Ch^p\mathbb{E}|b(\overline X_{nh}^{i,N,h},\widetilde{\mu}_{nh}^{\overline X,N,h})|^p+C|\sigma_0|^p\mathbb{E}|B_t^i-B_{nh}^i|^p\notag\\
&+C\mathbb{E}\left[|\sigma(\overline X_{nh}^{i,N,h},\widetilde{\mu}_{nh}^{\overline X,N,h})\vert^p\mathbb{E}\left(|\overline B_t^i-\overline B_{nh}^i|^p\big|\mathcal{F}_{nh}\right)\right]+C|\sigma_1|^p\mathbb{E}|W_t-W_{nh}|^p\notag\\
&+C\mathbb{E}\left[|\overline\sigma(\overline X_{nh}^{i,N,h},\widetilde{\mu}_{nh}^{\overline X,N,h})|^p\mathbb{E}\left(|\overline{W}_t-\overline W_{nh}|^p\big|\mathcal{F}_{nh}\right)\right]\notag\\
\le &C\left(1+\mathbb{E}|\overline X_{nh}^{i,N,h}|^{p(l+1)}\right)\notag\\
\le &C\left(1+\mathbb{E}|\overline X_0^{i,N,h}|^{p(l+1)}\right).
\end{align}
Based on \eqref{TEM}, using \eqref{b_growth}, Assumption \ref{N1} and Lemma \ref{TEM-lemma1}, we have
\begin{align*}
&\mathbb{E}|\overline X_t^{i,N,h}-\overline X_{nh}^{i,N,h}|^2\notag\\
\le &C\mathbb{E}\vert b(\overline X_{nh}^{i,N,h},\widetilde{\mu}_{nh}^{\overline X,N,h})\vert^2(t-nh)^2+C\sigma_0^2\mathbb{E}|B_t^i-B_{nh}^i|^2+C\sigma_1^2\mathbb{E}|W_t-W_{nh}|^2\notag\\
&+C\mathbb{E}\left(\vert\sigma(\overline X_{nh}^{i,N,h},\widetilde{\mu}_{nh}^{\overline X,N,h})\vert^2|\overline B_t^i-\overline B_{nh}^i|^2\right)+C\mathbb{E}\left(|\overline\sigma(\overline X_{nh}^{i,N,h},\widetilde{\mu}_{nh}^{\overline X,N,h})|^2|\overline{W}_t-\overline W_{nh}|^2\right)\notag\\
\le &Ch\left(\mathbb{E}\vert b(\overline X_{nh}^{i,N,h},\widetilde{\mu}_{nh}^{\overline X,N,h})\vert^2+\sigma_0^2+\sigma_1^2+\mathbb{E}\vert \sigma(\overline X_{nh}^{i,N,h},\widetilde{\mu}_{nh}^{\overline X,N,h})\vert^2+\mathbb{E}|\overline\sigma(\overline X_{nh}^{i,N,h},\widetilde{\mu}_{nh}^{\overline X,N,h})|^2\right)\notag\\
\le &Ch\left(1+\mathbb{E}|\overline X_{nh}^{i,N,h}|^{2(l+1)}\right)\notag\\
\le &Ch\left(1+\mathbb{E}|\overline X_0^{i,N,h}|^{2(l+1)}\right).
\end{align*}
Recall that $\widetilde b(x,\mu)=\frac{b(x,\mu)}{1+\alpha\sqrt{h}|x|^l}$ with $l$ being defined in Assumption \ref{S1}, by using Assumption \ref{S1} and \ref{A1}, we arrive at
\begin{align*}
&\mathbb{E}|b(\overline X_t^{i,N,h},\widetilde{\mu}_t^{\overline X,N,h})-\widetilde{b}(\overline X_{t_h}^{i,N,h},\widetilde{\mu}_{t_h}^{\overline X,N,h})|\notag\\
\le& \mathbb{E}\left|b(\overline X_t^{i,N,h},\widetilde{\mu}_t^{\overline X,N,h})-b(\overline X_{nh}^{i,N,h},\widetilde{\mu}_{nh}^{\overline X,N,h})\right|+\mathbb{E}\left|b(\overline X_{nh}^{i,N,h},\widetilde{\mu}_{nh}^{\overline X,N,h})-\frac{b(\overline X_{nh}^{i,N,h},\widetilde{\mu}_{nh}^{\overline X,N,h})}{1+\alpha\sqrt{h}|\overline X_{nh}^{i,N,h}|^l}\right|\notag\\
\le&K\mathbb{E}\left[(1+|\overline X_t^{i,N,h}|^l+|\overline X_{nh}^{i,N,h}|^l)\vert \overline X_t^{i,N,h}-\overline X_{nh}^{i,N,h}\vert+\lambda_2\mathbb{W}_1(\widetilde{\mu}_t^{\overline X,N,h},\widetilde{\mu}_{nh}^{\overline X,N,h})\right]\notag\\
&+\alpha\sqrt{h}\mathbb{E}\left[|\overline X_{nh}^{i,N,h}|^l|b(\overline X_{nh}^{i,N,h},\widetilde{\mu}_{nh}^{\overline X,N,h})|\right]\notag\\
\le&K\left(\mathbb{E}(1+|\overline X_t^{i,N,h}|^l+|\overline X_{nh}^{i,N,h}|^l)^2\right)^{1/2}\left(\mathbb{E}\vert \overline X_t^{i,N,h}-\overline X_{nh}^{i,N,h}\vert^2\right)^{1/2}+\lambda_2\left(\mathbb{E}\vert \overline X_t^{i,N,h}-\overline X_{nh}^{i,N,h}\vert^2\right)^{1/2}\notag\\
&+\alpha\hat K\sqrt{h}\mathbb{E}\left[|\overline X_{nh}^{i,N,h}|^l\left(1+|\overline X_{nh}^{i,N,h}|^{l+1}+\mathbb{W}_1(\widetilde{\mu}_{nh}^{\overline X,N,h},\delta_0)\right)\right]
\notag\\
\le&Ch^{1/2}\left(\mathbb{E}(1+|\overline X_t^{i,N,h}|^l+|\overline X_{nh}^{i,N,h}|^l)^2\right)^{1/2}\left(1+\mathbb{E}|\overline X_0^{i,N,h}|^{2(l+1)}\right)^{1/2}+Ch^{1/2}\left(1+\mathbb{E}|\overline X_0^{i,N,h}|^{2(l+1)}\right)^{1/2}\notag\\
&+\alpha\hat K\sqrt{h}\mathbb{E}\left[|\overline X_{nh}^{i,N,h}|^l\left(1+|\overline X_{nh}^{i,N,h}|^{l+1}+\mathbb{W}_1(\widetilde{\mu}_{nh}^{\overline X,N,h},\delta_0)\right)\right],
\end{align*}
together with \eqref{Eq-lemma_4.4-1} and Lemma \ref{TEM-lemma1}, it follows that 
\begin{align}\label{TEM-b-b_tilde}
\mathbb{E}|b(\overline X_t^{i,N,h},\widetilde{\mu}_t^{\overline X,N,h})-\widetilde{b}(\overline X_{t_h}^{i,N,h},\widetilde{\mu}_{t_h}^{\overline X,N,h})|\le Ch^{1/2}\left(1+\mathbb{E}|\overline X_0^{i,N,h}|^{2(l+1)^2}\right).
\end{align}
Applying Theorem \ref{main theorem} with $X_t^{i,N}=\overline X_t^{i,N,h}$, $\widetilde b(x,\mu)=\frac{b(x,\mu)}{1+\alpha\sqrt{h}|x|^l}$, $\theta_t=\overline \theta_t=t_h$, together with \eqref{TEM-b-b_tilde}, the assertion follows. 
\end{proof}
%%%-------------------------------------------------------------------------------------------------------------------------------------------------------

\subsection{Adaptive EM method}\label{Adaptive EM method}

For $t_{n+1}=t_n+h_n$ with $h_n$ is an adaptive time step-size, the adaptive EM method for the McKean-Vlasov SDE with common noise \eqref{NIPS} is constructed as
\begin{align}\label{AEM-dis}
\widetilde X_{t_{n+1}}^{i,N}=\widetilde X_{t_n}^{i,N}+b(\widetilde X_{t_n}^{i,N},\widetilde{\mu}_{t_n}^{\widetilde X,N})h_n+\sigma_0\Delta B_{t_n}^i+\sigma(\widetilde X_{t_n}^{i,N},\widetilde{\mu}_{t_n}^{\widetilde X,N})\Delta \overline B_{t_n}^i+\sigma_1\Delta W_{t_n}+\overline\sigma(\widetilde X_{t_n}^{i,N},\widetilde{\mu}_{t_n}^{\widetilde X, N})\Delta \overline W_{t_n},
\end{align}
with initial data $\widetilde X_0^{i,N}$ for all $i\in\mathbb{S}_N$, where $n\in \mathbb{N}_+$, $\Delta B_{t_n}^i=B_{t_{n+1}}^i-B_{t_n}^i$, $\Delta \overline B_{t_n}^i=\overline B_{t_{n+1}}^i-\overline B_{t_n}^i$, $\Delta W_{t_n}=W_{t_{n+1}}-W_{t_n}$, $\Delta \overline W_{t_n}=\overline W_{t_{n+1}}-\overline W_{t_n}$, $\widetilde{\mu}_{t_n}^{\widetilde X,N,h}=\frac{1}{N}\sum_{j=1}^N \delta_{\widetilde X_{t_n}^{j,N,h}}$, $h_n=\delta\min\{h_n^1, \cdots, h_n^N\}$ with $\delta\in(0,1)$, $h_n^i=h(\widetilde X_{t_n}^{i,N},\widetilde{\mu}_{t_n}^{\widetilde X,N})$ and
%\begin{align*}
%h(x,\mu):=\frac{1}{1+|b(x,\mu)|^2+|\sigma(x,\mu)|^2+|\overline\sigma(x,\mu)|^2}.
%\end{align*}
\begin{align*}
h(x,\mu):=\frac{1}{1+|b(x,\mu)|^2}.
\end{align*}
 Further set $\underline t=\max\{t_n:t_n\le t\}$, the continuous version of \eqref{AEM-dis} can be formulated as
\begin{align}\label{AEM-cont}
{\rm d}\widetilde X_t^{i,N}=b(\widetilde X_{\underline t}^{i,N},\widetilde{\mu}_{\underline t}^{\widetilde X, N}){\rm d}t+\sigma_0{\rm d}B_t^i+\sigma(\widetilde X_{\underline t}^{i,N},\widetilde{\mu}_{\underline t}^{\widetilde X,N}){\rm d}\overline B_t^i+\sigma_1{\rm d} W_t+\overline\sigma(\widetilde X_{\underline t}^{i,N},\widetilde{\mu}_{\underline t}^{\widetilde X,N}){\rm d}\overline{W}_t.
\end{align}
Therefore, \eqref{IPS} can cover \eqref{AEM-cont} once we set $\widetilde{b}=b$ and $\theta_t=\overline \theta_t=\underline t$.

\begin{lemma}
Suppose that Assumptions \ref{A1}, \ref{S1} and \ref{N1} hold, then $\lim_{n\to \infty}t_n=\infty$ almost surely.
\end{lemma}

\begin{proof}
For any $i\in\mathbb{S}_N$ and $H>0$, we define the truncated adaptive EM method of \eqref{NIPS} as follows 
\begin{align*}%\label{AEM-dis-H}
\begin{cases}
t_0^H=0, ~~\widetilde X_{t_0^H}^{i,N,H}=\widetilde X_0^{i,N}, ~~t_{n+1}^H=t_n^H+h_n^H, \\
\widetilde X_{t_{n+1}^H}^{i,N,H}=\widetilde X_{t_n^H}^{i,N,H}+b_H(\widetilde X_{t_n^H}^{i,N,H},\widetilde{\mu}_{t_n^H}^{\widetilde X,N,H})h_n^H+\sigma_0\Delta B_{t_n^H}^i+\sigma(\widetilde X_{t_n^H}^{i,N,H},\widetilde{\mu}_{t_n^H}^{\widetilde X,N,H})\Delta\overline B_{t_n^H}^i\\
\quad\qquad\qquad+\sigma_1\Delta W_{t_n^H}+\overline\sigma(\widetilde X_{t_n^H}^{i,N,H},\widetilde{\mu}_{t_n^H}^{\widetilde X, N,H})\Delta \overline W_{t_n^H},
\end{cases}
\end{align*}
where $h_n^H=\delta\min\{h_n^{1,H}, \cdots, h_n^{N,H}\}$ with $h_n^{i,H}=h_H(\widetilde X_{t_n^H}^{i,N,H},\widetilde{\mu}_{t_n^H}^{\widetilde X,N,H})$,
\begin{align*}
h_H(x,\mu)=
\begin{cases}
       h(x,\mu), & {\text{if}} ~~|x|^{2(l+1)}+\mathbb{W}_1^2(\mu,\delta_0)\le H,\\
       \frac{1}{1+H}, & {\text{if}}~~ |x|^{2(l+1)}+\mathbb{W}_1^2(\mu,\delta_0)>H,
\end{cases}
\end{align*}
\begin{align*}
b_H(x,\mu)=
\begin{cases}
       b(x,\mu), & {\text{if}} ~~|x|^{2(l+1)}+\mathbb{W}_1^2(\mu,\delta_0)\le H,\\
       \frac{x}{1+|x|^2}, & {\text{if}}~~ |x|^{2(l+1)}+\mathbb{W}_1^2(\mu,\delta_0)>H,
\end{cases}
\end{align*}
for $x\in\mathbb{R}^d$ and $\mu\in\mathcal{P}(\mathbb{R}^d)$, where $l$ is defined in Assumption \ref{S1}. Then, according to Remark \ref{Remark-1.1}, it can be checked that for all $x\in\mathbb{R}^d$ and $\mu\in\mathcal{P}(\mathbb{R}^d)$,
\begin{align}
%&|f_H(x,\mu)|h_H(x,\mu)\le \delta,\label{AEM-eq1}\\
|b_H(x,\mu)|^2&h_H(x,\mu)\le 1,\label{AEM-eq2}\\
\<x, b_H(x,\mu)\>\le  &1+\frac{\lambda_2}{2}\mathbb{W}_1^2(\mu,\delta_0)+K_2.\label{AEM-eq3}
\end{align}
Moreover, according to the definition of $h_n^H$ and \eqref{b_growth}, we have $t_{n+1}^H-t_n^H= h_n^H\ge \frac{\delta}{1+3\hat K^2(1+H)}$. Therefore, $\lim_{n\to \infty}t_n^H=\infty~ a.s.$

Next, we define $\underline t^H=\max\{t_n^H:t_n^H\le t\}$, the continuous interpolant process is defined by
\begin{align}\label{AEM-cont-H}
{\rm d}\widetilde X_t^{i,N,H}=b_H(\widetilde X_{\underline t^H}^{i,N,H},\widetilde{\mu}_{\underline t^H}^{\widetilde X, N,H}){\rm d}t+\sigma_0{\rm d}B_t^i+\sigma(\widetilde X_{\underline t^H}^{i,N,H},\widetilde{\mu}_{\underline t^H}^{\widetilde X,N,H}){\rm d}\overline B_t^i+\sigma_1{\rm d} W_t+\overline\sigma(\widetilde X_{\underline t^H}^{i,N,H},\widetilde{\mu}_{\underline t^H}^{\widetilde X,N,H}){\rm d}\overline{W}_t.
\end{align}
Using It\^o's formula and \eqref{AEM-eq3}, we have
\begin{align}\label{Eq-AEM}
|\widetilde X_t^{i,N,H}|^2\le&|\widetilde X_0^{i,N}|^2+\int_0^t  2\left(1+\frac{\lambda_2}{2}\mathbb{W}_1^2\left(\widetilde{\mu}_{\underline s^H}^{\widetilde X, N,H},\delta_0\right)+K_2\right){\rm d}s\notag\\
&+\int_0^t  2|\widetilde X_s^{i,N,H}-\widetilde X_{\underline s^H}^{i,N,H}||b_H(\widetilde X_{\underline s^H}^{i,N,H},\widetilde{\mu}_{\underline s^H}^{\widetilde X, N,H})|{\rm d}s\notag\\
&+\int_0^t \left(\sigma_0^2d+\sigma_1^2d+|\sigma(\widetilde X_{\underline s^H}^{i,N,H},\widetilde{\mu}_{\underline s^H}^{\widetilde X,N,H})|^2+|\overline\sigma(\widetilde X_{\underline s^H}^{i,N,H},\widetilde{\mu}_{\underline s^H}^{\widetilde X,N,H})|^2\right){\rm d}s+\mathcal{M}_t.
\end{align}
where 
\begin{align*}
\mathcal{M}_t=&\int_0^t  2\left\<\widetilde X_s^{i,N,H},\sigma_0{\rm d}B_s^i\right\>+\int_0^t  2\left\<\widetilde X_s^{i,N,H},\sigma(\widetilde X_{\underline s^H}^{i,N,H},\widetilde{\mu}_{\underline s^H}^{\widetilde X,N,H})\right\>{\rm d}\overline B_s^i\notag\\
&+\int_0^t  2\left\<\widetilde X_s^{i,N,H},\sigma_1{\rm d}W_s\right\>+\int_0^t  2\left\<\widetilde X_s^{i,N,H},\overline\sigma(\widetilde X_{\underline s^H}^{i,N,H},\widetilde{\mu}_{\underline s^H}^{\widetilde X,N,H})\right\>{\rm d}\overline W_s.
\end{align*}
Define $\tau_R=\inf\{t>0:\max_{i\in\mathbb{S}_N}|\widetilde X_t^{i,N,H}|>R\}$ for each $R>0$, then%and $\tau=s\wedge \tau_R$, 
\begin{align}\label{AEM-eq4}
\mathbb{E}|\widetilde X_{t\wedge \tau_R}^{i,N,H}|^2\le&\mathbb{E}|\widetilde X_0^{i,N}|^2+2\left(1+K_2\right)t+\lambda_2\mathbb{E}\int_0^t  \mathbb{W}_1^2(\widetilde{\mu}_{\underline \tau^H}^{\widetilde X, N,H},\delta_0){\rm d}s\notag\\
&+\mathbb{E}\int_0^t  2|\widetilde X_\tau^{i,N,H}-\widetilde X_{\underline \tau^H}^{i,N,H}||b_H(\widetilde X_{\underline \tau^H}^{i,N,H},\widetilde{\mu}_{\underline \tau^H}^{\widetilde X, N,H})|{\rm d}s\notag\\
&+\mathbb{E}\int_0^t \left(\sigma_0^2d+\sigma_1^2d+|\sigma(\widetilde X_{\underline \tau^H}^{i,N,H},\widetilde{\mu}_{\underline \tau^H}^{\widetilde X,N,H})|^2+|\overline\sigma(\widetilde X_{\underline \tau^H}^{i,N,H},\widetilde{\mu}_{\underline \tau^H}^{\widetilde X,N,H})|^2\right){\rm d}s,
\end{align}
where $\tau=s\wedge \tau_R$, $\underline \tau^H=\underline{s}^H\wedge \tau_R$. Using Assumption \ref{N1}, the isometry property of stochastic integrals and \eqref{AEM-eq2}, we can deduce that
\begin{align}\label{AEM-eq5}
&\mathbb{E}\left[|\widetilde X_\tau^{i,N,H}-\widetilde X_{\underline\tau^H}^{i,N,H}||b_H(\widetilde X_{\underline\tau^H}^{i,N,H},\widetilde{\mu}_{\underline\tau^H}^{\widetilde X, N,H})|\big|\mathcal{F}_{\underline\tau^H}\right]\notag\\
&\le |b_H(\widetilde X_{\underline\tau^H}^{i,N,H},\widetilde{\mu}_{\underline\tau^H}^{\widetilde X, N,H})|^2(\tau-\underline\tau^H)+|\sigma_0||b_H(\widetilde X_{\underline\tau^H}^{i,N,H},\widetilde{\mu}_{\underline\tau^H}^{\widetilde X, N,H})|\mathbb{E}\left[|B_\tau^i-B_{\underline\tau^H}^i|\big|\mathcal{F}_{\underline\tau^H}\right]\notag\\
&\quad+|\sigma(\widetilde X_{\underline\tau^H}^{i,N,H},\widetilde{\mu}_{\underline\tau^H}^{\widetilde X,N,H})||b_H(\widetilde X_{\underline\tau^H}^{i,N,H},\widetilde{\mu}_{\underline\tau^H}^{\widetilde X, N,H})|\mathbb{E}\left[|\overline B_\tau^i-\overline B_{\underline\tau^H}^i|\big|\mathcal{F}_{\underline\tau^H}\right]\notag\\
&\quad+|\sigma_1||b_H(\widetilde X_{\underline\tau^H}^{i,N,H},\widetilde{\mu}_{\underline\tau^H}^{\widetilde X, N,H})|\mathbb{E}\left[|W_\tau-W_{\underline\tau^H}|\big|\mathcal{F}_{\underline\tau^H}\right]\notag\\
&\quad+|\overline\sigma(\widetilde X_{\underline\tau^H}^{i,N,H},\widetilde{\mu}_{\underline\tau^H}^{\widetilde X,N,H})||b_H(\widetilde X_{\underline\tau^H}^{i,N,H},\widetilde{\mu}_{\underline\tau^H}^{\widetilde X, N,H})|\mathbb{E}\left[|\overline W_\tau-\overline W_{\underline\tau^H}|\big|\mathcal{F}_{\underline\tau^H}\right]\notag\\
%&\le |b_H(\widetilde X_{\underline\tau^H}^{i,N,H},\widetilde{\mu}_{\underline\tau^H}^{\widetilde X, N,H})|^2\delta h_H(\widetilde X_{\underline\tau^H}^{i,N,H},\widetilde{\mu}_{\underline\tau^H}^{\widetilde X, N,H})+|\sigma_0||b_H(\widetilde X_{\underline\tau^H}^{i,N,H},\widetilde{\mu}_{\underline\tau^H}^{\widetilde X, N,H})|\sqrt{d\delta h_H(\widetilde X_{\underline\tau^H}^{i,N,H},\widetilde{\mu}_{\underline \tau^H}^{\widetilde X, N,H})}\notag\\
%&\quad+|\sigma(\widetilde X_{\underline\tau^H}^{i,N,H},\widetilde{\mu}_{\underline\tau^H}^{\widetilde X,N,H})||b_H(\widetilde X_{\underline\tau^H}^{i,N,H},\widetilde{\mu}_{\underline\tau^H}^{\widetilde X, N,H})|\sqrt{d\delta h_H(\widetilde X_{\underline\tau^H}^{i,N,H},\widetilde{\mu}_{\underline \tau^H}^{\widetilde X, N,H})}\notag\\
%&\quad+|\overline\sigma(\widetilde X_{\underline\tau^H}^{i,N,H},\widetilde{\mu}_{\underline\tau^H}^{\widetilde X,N,H})||b_H(\widetilde X_{\underline\tau^H}^{i,N,H},\widetilde{\mu}_{\underline\tau^H}^{\widetilde X, N,H})|\sqrt{m\delta h_H(\widetilde X_{\underline\tau^H}^{i,N,H},\widetilde{\mu}_{\underline\tau^H}^{\widetilde X, N,H})}\notag\\
%&\le \delta+\sqrt{\delta}\left(|\sigma_0|\sqrt{d}+|\sigma(\widetilde X_{\underline \tau^H}^{i,N,H},\widetilde{\mu}_{\underline \tau^H}^{\widetilde X,N,H})|+|\overline\sigma(\widetilde X_{\underline \tau^H}^{i,N,H},\widetilde{\mu}_{\underline\tau^H}^{\widetilde X,N,H})|\right)\notag\\
&\le C\sqrt{\delta}.
\end{align}
Substituting \eqref{AEM-eq5} into \eqref{AEM-eq4}, using Assumption \ref{N1} and the fact that $\mathbb{E}\mathbb{W}_1^2(\widetilde{\mu}_{\underline\tau^H}^{\widetilde X,N,H},\delta_0)\le \mathbb{E}|\widetilde X_{\underline\tau^H}^{i,N,H}|^2$, one can deduce that
\begin{align*}%\label{}
\mathbb{E}|\widetilde X_{t\wedge \tau_R}^{i,N,H}|^2\le&\mathbb{E}|\widetilde X_0^{i,N}|^2+Ct+C\int_0^t \mathbb{E}|\widetilde X_{\underline\tau^H}^{i,N,H}|^2{\rm d}s.
\end{align*}
Moreover, using equation \eqref{AEM-cont-H}, the isometry property of stochastic integrals and \eqref{AEM-eq2}, it holds that
\begin{align*}%\label{}
\mathbb{E}|\widetilde X_{\tau}^{i,N,H}-\widetilde X_{\underline\tau^H}^{i,N,H}|^2\le& C\mathbb{E}
|b_H(\widetilde X_{\underline\tau^H}^{i,N,H},\widetilde{\mu}_{\underline\tau^H}^{\widetilde X, N,H})|^2(\tau-\underline\tau^H)^2+C\mathbb{E}[\sigma_0^2|B_\tau^i-B_{\underline\tau^H}^i|^2]\notag\\
&+C\mathbb{E}[|\sigma(\widetilde X_{\underline\tau^H}^{i,N,H},\widetilde{\mu}_{\underline\tau^H}^{\widetilde X,N,H})|^2|\overline B_\tau^i-\overline B_{\underline\tau^H}^i|^2]+C\mathbb{E}[\sigma_1^2|W_\tau-W_{\underline\tau^H}|^2]\notag\\
&+C\mathbb{E}[|\overline\sigma(\widetilde X_{\underline t^H}^{i,N,H},\widetilde{\mu}_{\underline t^H}^{\widetilde X,N,H})|^2|\overline W_\tau-\overline W_{\underline\tau^H}|^2]\notag\\
\le& C\delta,
\end{align*}
Hence,
\begin{align*}%\label{}
\mathbb{E}|\widetilde X_{t\wedge \tau_R}^{i,N,H}|^2\le&\mathbb{E}|\widetilde X_0^{i,N}|^2+Ct+C\int_0^t \left(\mathbb{E}|\widetilde X_{\tau}^{i,N,H}|^2+\mathbb{E}|\widetilde X_{\tau}^{i,N,H}-\widetilde X_{\underline\tau^H}^{i,N,H}|^2\right){\rm d}s\notag\\
\le&\mathbb{E}|\widetilde X_0^{i,N}|^2+Ct+C\int_0^t \mathbb{E}|\widetilde X_{s\wedge \tau_R}^{i,N,H}|^2{\rm d}s,
\end{align*}
which, together with Gronwall's inequality, yields that for any $T>0$,
\begin{align*}%\label{sup_E_bound}
\max_{i\in\mathbb{S}_N}\sup_{0\le t\le T}\mathbb{E}|\widetilde X_{t\wedge \tau_R}^{i,N,H}|^2\le&(\mathbb{E}|\widetilde X_0^{1,N}|^2+CT){\rm e}^{CT}=:C(\mathbb{E}|\widetilde X_0^{1,N}|^2,T).
\end{align*}
Then, with the help of the Markov inequality, we obtain that 
\begin{align*}
\mathbb{P}(\tau_R< T)\le \sum_{i=1}^N\mathbb{P}(|\widetilde X_{T\wedge \tau_R}^{i,N,H}|>R)=N\mathbb{P}(|\widetilde X_{T\wedge \tau_R}^{1,N,H}|>R)\le N\frac{\mathbb{E}|\widetilde X_{T\wedge \tau_R}^{1,N,H}|^2}{R^2}\le N \frac{C(\mathbb{E}|\widetilde X_0^{1,N}|^2,T)}{R^2},
\end{align*}
which tends to zero as $R\to \infty$. Therefore, $\tau_R\to\infty$ as $R\to \infty$. Then due to Fatou's lemma, we get 
\begin{align}\label{sup_E_bound}
\max_{i\in\mathbb{S}_N}\sup_{0\le t\le T}\mathbb{E}|\widetilde X_{t}^{i,N,H}|^2\le&C(\mathbb{E}|\widetilde X_0^{1,N}|^2,T).
\end{align}

On the other hand, for equation \eqref{Eq-AEM}, applying the Burkholder-Davis-Gundy inequality, together with  Assumption \ref{N1}, it holds that
\begin{align}\label{AEM-eq6}
\mathbb{E}\left[\sup_{0\le t\le T}|\widetilde X_t^{i,N,H}|^2\right]\le&\mathbb{E}|\widetilde X_0^{i,N}|^2+C\int_0^T \mathbb{E}\mathbb{W}_1^2(\widetilde{\mu}_{\underline t^H}^{\widetilde X,N,H},\delta_0){\rm d}t+CT+\frac{1}{2}\mathbb{E}\left[\sup_{0\le t\le T}|\widetilde X_t^{i,N,H}|^2\right]\notag\\
&+\mathbb{E}\left[\int_0^T  2|\widetilde X_t^{i,N,H}-\widetilde X_{\underline t^H}^{i,N,H}||b_H(\widetilde X_{\underline t^H}^{i,N,H},\widetilde{\mu}_{\underline t^H}^{\widetilde X, N,H})|{\rm d}t\right].
\end{align}
Using \eqref{AEM-eq5} once more, combined with $\mathbb{E}\mathbb{W}_1^2(\widetilde{\mu}_{\underline t^H}^{\widetilde X,N,H},\delta_0))\le \mathbb{E}|\widetilde X_{\underline t^H}^{i,N,H}|^2$, $\delta\in(0,1)$ and \eqref{sup_E_bound}, it gives
\begin{align}\label{AEM-eq8}
\mathbb{E}\left[\sup_{0\le t\le T}|\widetilde X_t^{i,N,H}|^2\right]\le&2\mathbb{E}|\widetilde X_0^{i,N}|^2+C\int_0^T \mathbb{E}|\widetilde X_{\underline t^H}^{i,N,H}|^2{\rm d}t+CT+C\delta^{1/2}\notag\\
\le& 2\mathbb{E}|\widetilde X_0^{i,N}|^2+CT\sup_{0\le t\le T}\mathbb{E}|\widetilde X_{t}^{i,N,H}|^2+CT\notag\\
\le& C(\mathbb{E}|\widetilde X_0^{i,N}|^2,T).
\end{align}
Note that
\begin{align*}
\{t_n\le T\}\subset& \left\{t_n\le T,~\max_{i\in\mathbb{S}_N}\sup_{t\in[0,T]}\left(|\widetilde X_t^{i,N,H}|^{2(l+1)}+\mathbb{W}_1^2(\widetilde{\mu}_t^{\widetilde X,N,H},\delta_0)\right)\le \frac{H}{2}\right\}\notag\\
&\cup\left\{\max_{i\in\mathbb{S}_N}\sup_{t\in[0,T]}\left(|\widetilde X_t^{i,N,H}|^{2(l+1)}+\mathbb{W}_1^2(\widetilde{\mu}_t^{\widetilde X,N,H},\delta_0)\right)> \frac{H}{2}\right\}\notag\\
\subset& \left\{t_n^H\le T\right\}\cup\left\{\max_{i\in\mathbb{S}_N}\sup_{t\in[0,T]}|\widetilde X_t^{i,N,H}|^{2(l+1)}> \frac{H}{4}\right\}\cup\left\{\max_{i\in\mathbb{S}_N}\sup_{t\in[0,T]}\mathbb{W}_1^2(\widetilde{\mu}_t^{\widetilde X,N,H},\delta_0)> \frac{H}{4}\right\}.
\end{align*}
Then, using $\mathbb{E}\left[\sup_{t\in[0,T]}\mathbb{W}_1^2(\widetilde{\mu}_{\underline t^H}^{\widetilde X,N,H},\delta_0)\right]\le \mathbb{E}\left[\sup_{t\in[0,T]}|\widetilde X_{\underline t^H}^{i,N,H}|^2\right]$ for any $i\in\mathbb{S}_N$, Markov's inequality and \eqref{AEM-eq8}, we get that for any $H>0$,
\begin{align*}
\mathbb{P}(t_n\le T)\le& \mathbb{P}(t_n^H\le T)+\left(\frac{4}{H}\right)^{1/(l+1)}\sum_{i=1}^N\mathbb{E}\left[\sup_{t\in[0,T]}|\widetilde X_{\underline t^H}^{i,N,H}|^2\right]+\frac{4}{H}\sum_{i=1}^N\mathbb{E}\left[\sup_{t\in[0,T]}|\widetilde X_{\underline t^H}^{i,N,H}|^2\right]\notag\\
\le&\mathbb{P}(t_n^H\le T)+\left(\left(\frac{4}{H}\right)^{1/(l+1)}+\frac{4}{H}\right)NC(\mathbb{E}|X_0^i|^2,T).
\end{align*}
Then, let $n\to \infty$ and recall that $\lim_{n\to\infty} t_n^H=\infty ~a.s.,$ we have for any $H>0$,
\begin{align*}
\limsup_{n\to\infty}\mathbb{P}(t_n\le T)\le \left(\left(\frac{4}{H}\right)^{1/(l+1)}+\frac{4}{H}\right)NC(\mathbb{E}|\widetilde X_0^{i,N}|^2,T).
\end{align*}
Then, letting $H\to \infty$, we get $\lim_{n\to\infty}\mathbb{P}(t_n\le T)=0$. Therefore, $t_n\to\infty$ in probability as $n\to\infty$. Since $(t_n)_{n\ge 0}$ is an increasing sequence, we have $\lim_{n\to\infty}t_n\to\infty~ a.s.$ Thus, the result follows. 
\end{proof}

\begin{lemma}\label{lemma-AEM-boundedness}
Suppose that Assumptions \ref{A1}, \ref{S1} and \ref{N1} hold with $\lambda_1>2\lambda_2$, for any $p>0$, $i\in\mathbb{S}_N$ and $\delta\in(0,\frac{1}{2(\lambda_1-2\lambda_2)})$, there is a constant $C>0$ such that 
\begin{align}\label{AEM-boundedness}
\mathbb{E}|\widetilde X_t^{i,N}|^p\le \mathbb{E}|\widetilde X_0^{i,N}|^p+C, ~t\ge 0,
\end{align}
in case of $\mathbb{E}|\widetilde X_0^{i,N}|^p<\infty$.
\end{lemma}

\begin{proof}
By tracing the proof of Lemma \ref{boundedness of BEM}, it is sufficient to show that, for any integer $p\ge 1$, there exist a constant $C>0$ such that 
\begin{align*}
\mathbb{E}\left(|\widetilde X_t^{i,N}|^{2p}\vert\mathcal{F}_0\right)\le \mathbb{E}\left(|\widetilde X_0^{i,N}|^{2p}\vert\mathcal{F}_0\right)+C, ~t\ge 0
\end{align*}
for the sake of validity of \eqref{AEM-boundedness}.

For any $t>0$, recall that $\underline t=\max\{t_n:t_n\le t\}$, applying It\^o's formula, for any integer $p\ge 1$, it follows from \eqref{AEM-cont}, \eqref{B1} and Assumption \ref{N1} that
\begin{align*}
|\widetilde X_t^{i,N}|^{2}=& |\widetilde X_{\underline{t}}^{i,N}|^2+\int_{\underline{t}}^t 2\left\<\widetilde X_{\underline s}^{i,N},b(\widetilde X_{\underline s}^{i,N},\widetilde{\mu}_{\underline s}^{\widetilde X, N})\right\>{\rm d}s+\int_{\underline{t}}^t 2\left\<\widetilde X_s^{i,N}-\widetilde X_{\underline s}^{i,N},b(\widetilde X_{\underline s}^{i,N},\widetilde{\mu}_{\underline s}^{\widetilde X, N})\right\>{\rm d}s\notag\\
&+\int_{\underline{t}}^t \left(\sigma_0^2d+\sigma_1^2d+|\sigma(\widetilde X_{\underline t}^{i,N},\widetilde{\mu}_{\underline t}^{\widetilde X,N})|^2+|\overline\sigma(\widetilde X_{\underline t}^{i,N},\widetilde{\mu}_{\underline t}^{\widetilde X,N})|^2\right){\rm d}s+\widetilde M_t\notag\\
\le &|\widetilde X_{\underline{t}}^{i,N}|^2+2\delta\left(-K_1|\widetilde X_{\underline t}^{i,N}|^2+\frac{\lambda_2}{2}\mathbb{W}_1^2(\widetilde{\mu}_{\underline t}^{\widetilde X, N},\delta_0)+K_2+\sigma_0^2d+\sigma_1^2d+2\hat L^2\right)\notag\\
&+\int_{\underline{t}}^t 2\left\<\widetilde X_s^{i,N}-\widetilde X_{\underline s}^{i,N},b(\widetilde X_{\underline s}^{i,N},\widetilde{\mu}_{\underline s}^{\widetilde X, N})\right\>{\rm d}s+\widetilde M_t,
\end{align*}
where $K_1, K_2$ was defined in Remark \ref{Remark-1.1}, 
\begin{align*}
\widetilde M_t:=2\int_{\underline{t}}^t \left\<\widetilde X_s^{i,N}, \sigma_0{\rm d}B_s^i+\sigma(\widetilde X_{\underline s}^{i,N},\widetilde{\mu}_{\underline s}^{\widetilde X,N}){\rm d}\overline B_s^i+\sigma_1{\rm d}W_s+\overline\sigma(\widetilde X_{\underline s}^{i,N},\widetilde{\mu}_{\underline s}^{\widetilde X,N}){\rm d}\overline{W}_s\right\>.
\end{align*}
For any integer $p\ge 1$, using the binomial theorem, we have
\begin{align}\label{AEM-tilldeX}
|\widetilde X_t^{i,N}|^{2p}\le&\Bigg\{(1-2K_1\delta)|\widetilde X_{\underline{t}}^{i,N}|^2+\lambda_2\delta\mathbb{W}_1^2(\widetilde{\mu}_{\underline t}^{\widetilde X, N},\delta_0)+2(K_2+\sigma_0^2d+\sigma_1^2d+2\hat L^2)\delta\notag\\
&\quad+\int_{\underline{t}}^t 2\left\<\widetilde X_s^{i,N}-\widetilde X_{\underline s}^{i,N},b(\widetilde X_{\underline s}^{i,N},\widetilde{\mu}_{\underline s}^{\widetilde X, N})\right\>{\rm d}s+\widetilde M_t\Bigg\}^{p}\notag\\
=&\Bigg\{(1-2K_1\delta)|\widetilde X_{\underline{t}}^{i,N}|^2+\lambda_2\delta\mathbb{W}_1^2(\widetilde{\mu}_{\underline t}^{\widetilde X, N},\delta_0)\Bigg\}^{p}\notag\\
&\quad+p\Bigg\{(1-2K_1\delta)|\widetilde X_{\underline{t}}^{i,N}|^2+\lambda_2\delta\mathbb{W}_1^2(\widetilde{\mu}_{\underline t}^{\widetilde X, N},\delta_0)\Bigg\}^{p-1}\notag\\
&\qquad\times\left(\overline K\delta+\int_{\underline{t}}^t 2\left\<\widetilde X_s^{i,N}-\widetilde X_{\underline s}^{i,N},b(\widetilde X_{\underline s}^{i,N},\widetilde{\mu}_{\underline s}^{\widetilde X, N})\right\>{\rm d}s+\widetilde M_t\right)\notag\\
&\quad+\sum_{k=0}^{p-2}C_{p}^k\Bigg\{(1-2K_1\delta)|\widetilde X_{\underline{t}}^{i,N}|^2+\lambda_2\delta\mathbb{W}_1^2(\widetilde{\mu}_{\underline t}^{\widetilde X, N},\delta_0)\Bigg\}^{k}\notag\\
&\qquad\times\left(\overline K\delta+\int_{\underline{t}}^t 2\left\<\widetilde X_s^{i,N}-\widetilde X_{\underline s}^{i,N},b(\widetilde X_{\underline s}^{i,N},\widetilde{\mu}_{\underline s}^{\widetilde X, N})\right\>{\rm d}s+\widetilde M_t\right)^{p-k}\notag\\
=:&\Phi_i({\bf X}_{\underline t}^{N,\delta})+\widetilde\Phi_i({\bf X}_{\underline t}^{N,\delta})+\overline\Phi_i({\bf X}_{\underline t}^{N,\delta}),
\end{align}
where $\overline K=2(K_2+\sigma_0^2d+\sigma_1^2d+2\hat L^2)$. By means of the binomial theorem and the Young inequality, since $\widetilde X_{\underline{t}}^{1,N},\dots, \widetilde X_{\underline{t}}^{N,N}$ are identical distributed, $1-2(\lambda_1-2\lambda_2)\delta\in(0,1)$, we deduce that
\begin{align}\label{AEM-Phi}
\mathbb{E}\left(\Phi_i({\bf X}_{\underline t}^{N,\delta})\big\vert \mathcal{F}_0\right)=&\mathbb{E}\left(\sum_{k=0}^{p}C_{p}^k(1-2K_1\delta)^k|\widetilde X_{\underline{t}}^{i,N}|^{2k}(\lambda_2\delta)^{p-k}\mathbb{W}_1^{2(p-k)}(\widetilde{\mu}_{\underline t}^{\widetilde X, N},\delta_0)\big\vert \mathcal{F}_0\right)\notag\\
\le &\left(1-2(\lambda_1-2\lambda_2)\delta\right)^{2p}\mathbb{E}\left(|\widetilde X_{\underline{t}}^{i,N}|^{2p}\big\vert \mathcal{F}_0\right)\notag\\
\le &\left(1-2(\lambda_1-2\lambda_2)\delta\right)\mathbb{E}\left(|\widetilde X_{\underline{t}}^{i,N}|^{2p}\big\vert \mathcal{F}_0\right).
\end{align}
On the other hand, 
\begin{align*}%\label{AEM-tildePhi}
\mathbb{E}\left(\widetilde\Phi_i({\bf X}_{\underline t}^{N,\delta})\big\vert \mathcal{F}_0\right)=&p\mathbb{E}\Bigg(\Bigg\{(1-2K_1\delta)|\widetilde X_{\underline{t}}^{i,N}|^2+\lambda_2\delta\mathbb{W}_1^2(\widetilde{\mu}_{\underline t}^{\widetilde X, N},\delta_0)\Bigg\}^{p-1}\overline K\delta\Bigg\vert \mathcal{F}_0\Bigg)\notag\\
&+p\mathbb{E}\Bigg(\Bigg\{(1-2K_1\delta)|\widetilde X_{\underline{t}}^{i,N}|^2+\lambda_2\delta\mathbb{W}_1^2(\widetilde{\mu}_{\underline t}^{\widetilde X, N},\delta_0)\Bigg\}^{p-1}\notag\\
&\quad\qquad\times\int_{\underline{t}}^t 2\left\<\widetilde X_s^{i,N}-\widetilde X_{\underline s}^{i,N},b(\widetilde X_{\underline s}^{i,N},\widetilde{\mu}_{\underline s}^{\widetilde X, N})\right\>{\rm d}s\Bigg\vert \mathcal{F}_0\Bigg)\notag\\
&+p\mathbb{E}\Bigg(\Bigg\{(1-2K_1\delta)|\widetilde X_{\underline{t}}^{i,N}|^2+\lambda_2\delta\mathbb{W}_1^2(\widetilde{\mu}_{\underline t}^{\widetilde X, N},\delta_0)\Bigg\}^{p-1}\widetilde M_t\Bigg\vert \mathcal{F}_0\Bigg)\notag\\
=:&\widetilde\Phi_i^1({\bf X}_{\underline t}^{N,\delta})+\widetilde\Phi_i^2({\bf X}_{\underline t}^{N,\delta})+\widetilde\Phi_i^3({\bf X}_{\underline t}^{N,\delta}).
\end{align*}
By the Young inequality, one can arrive at
\begin{align*}%\label{AEM-tildePhi-1}
\widetilde\Phi_i^1({\bf X}_{\underline t}^{N,\delta})\le \frac{1}{4}(\lambda_1-2\lambda_2)\delta\mathbb{E}\left(|\widetilde X_{\underline{t}}^{i,N}|^{2p}\big\vert \mathcal{F}_0\right)+C\delta.
\end{align*}
According to Assumption \ref{N1} and the definition of the step-size, using the Young inequality once more, it holds that
\begin{align*}%\label{AEM-tildePhi-2}
\widetilde\Phi_i^2({\bf X}_{\underline t}^{N,\delta})=&2p\mathbb{E}\Bigg(\Bigg\{(1-2K_1\delta)|\widetilde X_{\underline{t}}^{i,N}|^2+\lambda_2\delta\mathbb{W}_1^2(\widetilde{\mu}_{\underline t}^{\widetilde X, N},\delta_0)\Bigg\}^{p-1}\int_{\underline{t}}^t \vert b(\widetilde X_{\underline s}^{i,N},\widetilde{\mu}_{\underline s}^{\widetilde X, N})\vert^2(s-\underline s){\rm d}s\Big\vert \mathcal{F}_0\Bigg)\notag\\
&+2p\mathbb{E}\Bigg(\Bigg\{(1-2K_1\delta)|\widetilde X_{\underline{t}}^{i,N}|^2+\lambda_2\delta\mathbb{W}_1^2(\widetilde{\mu}_{\underline t}^{\widetilde X, N},\delta_0)\Bigg\}^{p-1} \sigma_0\vert b(\widetilde X_{\underline t}^{i,N},\widetilde{\mu}_{\underline t}^{\widetilde X, N})\vert\int_{\underline{t}}^t |B_s^i-B_{\underline s}^i|{\rm d}s\Bigg\vert \mathcal{F}_0\Bigg)\notag\\
&+2p\mathbb{E}\Bigg(\Bigg\{(1-2K_1\delta)|\widetilde X_{\underline{t}}^{i,N}|^2+\lambda_2\delta\mathbb{W}_1^2(\widetilde{\mu}_{\underline t}^{\widetilde X, N},\delta_0)\Bigg\}^{p-1}\notag\\
&\qquad\times\vert\sigma(\widetilde X_{\underline t}^{i,N},\widetilde{\mu}_{\underline t}^{\widetilde X,N})\vert\vert b(\widetilde X_{\underline t}^{i,N},\widetilde{\mu}_{\underline t}^{\widetilde X, N})\vert\int_{\underline{t}}^t |\overline B_s^i-\overline B_{\underline s}^i|{\rm d}s\Bigg\vert \mathcal{F}_0\Bigg)\notag\\
&+2p\mathbb{E}\Bigg(\Bigg\{(1-2K_1\delta)|\widetilde X_{\underline{t}}^{i,N}|^2+\lambda_2\delta\mathbb{W}_1^2(\widetilde{\mu}_{\underline t}^{\widetilde X, N},\delta_0)\Bigg\}^{p-1} \sigma_1\vert b(\widetilde X_{\underline t}^{i,N},\widetilde{\mu}_{\underline t}^{\widetilde X, N})\vert\int_{\underline{t}}^t |W_s-W_{\underline s}|{\rm d}s\Bigg\vert \mathcal{F}_0\Bigg)\notag\\
&+2p\mathbb{E}\Bigg(\Bigg\{(1-2K_1\delta)|\widetilde X_{\underline{t}}^{i,N}|^2+\lambda_2\delta\mathbb{W}_1^2(\widetilde{\mu}_{\underline t}^{\widetilde X, N},\delta_0)\Bigg\}^{p-1}\notag\\
&\qquad\times\vert\overline\sigma(\widetilde X_{\underline t}^{i,N},\widetilde{\mu}_{\underline t}^{\widetilde X,N})\vert \vert b(\widetilde X_{\underline t}^{i,N},\widetilde{\mu}_{\underline t}^{\widetilde X, N})\vert\int_{\underline{t}}^t |\overline{W}_s-\overline{W}_{\underline s}|{\rm d}s\Big\vert \mathcal{F}_0\Bigg)\notag\\
\le &C\delta\mathbb{E}\Bigg(\Bigg\{(1-2K_1\delta)|\widetilde X_{\underline{t}}^{i,N}|^2+\lambda_2\delta\mathbb{W}_1^2(\widetilde{\mu}_{\underline t}^{\widetilde X, N},\delta_0)\Bigg\}^{p-1}\Big\vert \mathcal{F}_0\Bigg)\notag\\
\le &\frac{1}{4}(\lambda_1-2\lambda_2)\delta\mathbb{E}\left(|\widetilde X_{\underline{t}}^{i,N}|^{2p}\big\vert \mathcal{F}_0\right)+C\delta.
\end{align*}
 Note that $(\widetilde M_t)_{t\ge 0}$ is a martingale, then $\widetilde\Phi_i^3({\bf X}_{\underline t}^{N,\delta})=0$. Hence
\begin{align}\label{AEM-tildePhi-0}
\mathbb{E}\left(\widetilde\Phi_i({\bf X}_{\underline t}^{N,\delta})\big\vert \mathcal{F}_0\right)\le \frac{1}{2}(\lambda_1-2\lambda_2)\delta\mathbb{E}\left(|\widetilde X_{\underline{t}}^{i,N}|^{2p}\big\vert \mathcal{F}_0\right)+C\delta.
\end{align}
Finally, together with the Young inequality and the fact that the conditional expectation (given $\mathcal{F}_{\underline t}$) of the increments $|B_t^i-B_{\underline t}^i|^{p-k}$, $|\overline B_t^i-\overline B_{\underline t}^i|^{p-k}$, $|W_t-W_{\underline t}^1|^{p-k}$, and $|\overline{W}_t-W_{\underline t}^2|^{p-k}$ contribute at least the order $\delta$ for $k\le p-2$, leads to
\begin{align}\label{AEM-barPhi}
\mathbb{E}\left(\overline\Phi_i({\bf X}_{\underline t}^{N,\delta})\big\vert \mathcal{F}_0\right)\le \frac{1}{2}(\lambda_1-2\lambda_2)\delta\mathbb{E}\left(|\widetilde X_{\underline{t}}^{i,N}|^{2p}\big\vert \mathcal{F}_0\right)+C\delta.
\end{align}
Substituting \eqref{AEM-Phi}, \eqref{AEM-tildePhi-0} and \eqref{AEM-barPhi} into \eqref{AEM-tilldeX}, based on the continuity, via an inductive argument, since $1-(\lambda_1-2\lambda_2)\delta\in (0,1)$, we arrive at
\begin{align*}%\label{AEM-tilldeX}
\mathbb{E}\left(|\widetilde X_t^{i,N}|^{2p}\vert \mathcal{F}_0\right)\le& \left(1-(\lambda_1-2\lambda_2)\delta\right)\mathbb{E}\left(|\widetilde X_{\underline{t}}^{i,N}|^{2p}\big\vert \mathcal{F}_0\right)+C\delta\notag\\
\le& \left(1-(\lambda_1-2\lambda_2)\delta\right)\mathbb{E}\left(|\widetilde X_0^{i,N}|^{2p}\big\vert \mathcal{F}_0\right)+C\notag\\
\le& \mathbb{E}\left(|\widetilde X_0^{i,N}|^{2p}\big\vert \mathcal{F}_0\right)+C.
\end{align*}

\end{proof}

%\begin{lemma}\label{AEM-lemma2}
%Under Assumptions \ref{A1}, \ref{A12}, \ref{S1}, \ref{N1}, for any $i\in\mathbb{S}_N$ and $\delta\in(0,1)$, there is a constant $C>0$ such that 
%\begin{align*}
%\mathbb{E}|b(\widetilde X_t^{i,N},\widetilde{\mu}_t^{\widetilde X,N})-b(\widetilde X_{\underline t}^{i,N},\widetilde{\mu}_{\underline t}^{\widetilde X,N})|\le C\left(1+\mathbb{E}|\widetilde X_0^{i,N}|^{2l}\right)\delta^{1/2},
%\end{align*}
%in case of $\mathbb{E}|\widetilde X_0^{i,N}|^{2l}<\infty$.
%\end{lemma}

Finally, the uniform-in-time error estimate for the adaptive EM method can be obtained follows Theorem \ref{main theorem} and Lemma \ref{lemma-AEM-boundedness}.
\begin{theorem}\label{adaptive-error-bound}
Let $\sigma_0\sigma_1\neq 0$. Suppose that Assumptions \ref{A1}, \ref{A2}, \ref{S1} and \ref{N1} hold with $\lambda_1>2\lambda_2$. Then there exists $C>0$ such that for any $\lambda_2\in [0,\lambda_2^*)$, $L\in[0,L^*]$, $\delta\in(0,\frac{1}{2(\lambda_1-2\lambda_2)})$, $t>0$ and $i\in\mathbb{S}_N$,
\begin{align*}
\mathcal{W}_1\big(\mathcal{L}(\mu_t^i),\mathcal{L}(\widetilde\nu_t^{i,N})\big)\le C\left\{{\rm e}^{-\lambda_0t}\mathbb{W}_1(\mu,\nu)+\psi(N)+\left(1+\mathbb{E}|\widetilde X_0^{i,N}|^{2l}\right)\delta^{1/2}\right\}
\end{align*}
in the case of $\mathbb{E}|\widetilde X_0^i|^q<\infty$ for some $q>1$ and $\mathbb{E}|\widetilde X_0^{i,N}|^{2l}<\infty$, where $\mu_t^i=\mathcal{L}(X_t^i\vert \mathcal{F}_t^W)$ stands for the regular conditional distribution of $X_t$, determined by \eqref{NIPS}, with the initial distribution $\mathcal{L}(X_0^i)=\mu$, $\widetilde\nu_t^{i,N,h}=\mathcal{L}(\widetilde X_t^{i,N,h}\vert \mathcal{F}_t^W)$ stands for the regular conditional distribution of $\widetilde X_t^{i,N,h}$, determined by \eqref{AEM-cont}, with the initial distribution $\mathcal{L}(\widetilde X_0^{i,N,h})=\nu$, $\lambda_2^*, L^*$, and $\lambda_0$ were given in Theorem  \ref{main theorem}.
\end{theorem}

\begin{proof}
From Assumption \ref{S1}, based on the H\"older inequality, it is easy to get that
\begin{align*}
\mathbb{E}|b(\widetilde X_t^{i,N},\widetilde{\mu}_t^{\widetilde X,N})-b(\widetilde X_{\underline t}^{i,N},\widetilde{\mu}_{\underline t}^{\widetilde X,N})|\le& C\left(1+(\mathbb{E}|\widetilde X_t^{i,N}|^{2l})^{1/2}+(\mathbb{E}|\widetilde X_{\underline t}^{i,N}|^{2l})^{1/2}\right)(\mathbb{E}|\widetilde X_t^{i,N}-\widetilde X_{\underline t}^{i,N}|^2)^{1/2}.
\end{align*}
From \eqref{AEM-cont}, using the tower property of conditional expectation, together with the definition of the step-size, one has
\begin{align*}
\mathbb{E}|\widetilde X_t^{i,N}-\widetilde X_{\underline t}^{i,N}|^2\le& C\delta^2\mathbb{E}\left[\left|b(\widetilde X_{\underline t}^{i,N},\widetilde{\mu}_{\underline t}^{\widetilde X, N})\right|^2h(\widetilde X_{\underline t}^{i,N},\widetilde{\mu}_{\underline t}^{\widetilde X, N})^2\right]+C\delta\mathbb{E}\left[\left(\sigma_0^2+\left|\sigma(\widetilde X_{\underline t}^{i,N},\widetilde{\mu}_{\underline t}^{\widetilde X,N})\right|^2\right)h(\widetilde X_{\underline t}^{i,N},\widetilde{\mu}_{\underline t}^{\widetilde X, N})\right]\notag\\
&+C\delta\mathbb{E}\left[\left|\overline\sigma(\widetilde X_{\underline t}^{i,N},\widetilde{\mu}_{\underline t}^{\widetilde X,N})\right|^2h(\widetilde X_{\underline t}^{i,N},\widetilde{\mu}_{\underline t}^{\widetilde X, N})\right]\notag\\
\le& C\delta.
\end{align*}
Hence, it follows from Lemma \ref{lemma-AEM-boundedness} that
\begin{align}\label{AEM-b-b_tilde}
\mathbb{E}|b(\widetilde X_t^{i,N},\widetilde{\mu}_t^{\widetilde X,N})-b(\widetilde X_{\underline t}^{i,N},\widetilde{\mu}_{\underline t}^{\widetilde X,N})|\le& C\left(1+\mathbb{E}|\widetilde X_0^{i,N}|^{2l}\right)\delta^{1/2}.
\end{align}
Using Theorem \ref{main theorem}, take $X_t^{i,N}=\overline X_t^{i,N,h}$ satisfies \eqref{AEM-cont}, $\widetilde b=b$, $\theta_t=\overline \theta_t=t_h$, combining \eqref{AEM-b-b_tilde}, the assertion follows. 
\end{proof}

\section{Numerical example}\label{Numerical example}
\begin{example}\rm
Consider the scalar McKean-Vlasov SDEs with common noise:
\begin{align}\label{example1-eq}
{\rm d}X_t=(X_t-X_t^3+\mathbb{E}^B[X_t]){\rm d}t+(\sigma_0+0.1X_t)dB_t+(\sigma_1+\mathbb{E}^B[X_t])dW_t
\end{align}
where $X_0\sim N(0,1)$ is supported by $(\Omega^B,\mathcal{F}^B, \mathbb{P}^B)$. It is easy to verify that
\begin{align*}
\<x-y, b(x,\mu)-b(y,\mu)\>=&\<x-y, x-x^3-(y-y^3)\>\\
%=&(x-y)^2[1-(x^2+y^2+xy)]\\
=&(x-y)^2\left[1-\frac{3}{4}(x+y)^2-\frac{1}{4}(x-y)^2\right]\\
%\le &(x-y)^2\left[1-\frac{1}{4}(x-y)^2\right]\\
\le &2(x-y)^2\mathbbm{1}_{\{|x-y|\le 2\sqrt{2}\}}-(x-y)^2,
\end{align*}
hence the drift term is dissipative merely in the long distance. Moreover, we can also verify that other assumptions hold. 

To examine the long-time error, we simulate equation \eqref{example1-eq} using numerical methods. First, we fix the number of particles $N=500$ and test simulation errors of the backward Euler method(BEM), the tamed Euler method(TEM), and the adaptive Euler method(AEM). The parameters are chosen as follows:  the step size $h=0.01$ for BEM and TEM, while AEM uses a baseline step size $\delta=1$ with an adaptive scaling factor; the noise intensity $\sigma_0=0.1, \sigma_1=0.02$, total simulation time $T=40$. Reference solution is obtained by TEM with a much smaller step size $10^{-3}$. Figure  \ref{Fig.main1} shows the evolution of the $L^1$ error between the numerical solution and the reference solution in semi-logarithmic scale. All three error curves remain bounded at a low level (approximately $10^{-2}$ to $10^{-1}$) over the whole time interval. This experiment confirms Theorems  \ref{BEM-error-bound}, \ref{TEM-error-bound}, and \ref{adaptive-error-bound}: even though the drift is only dissipative in long distance, the three numerical schemes are stable in the long run and their errors do not blow up with time. Since the convergence rates of these three methods are standard and well-known, we do not display them here.

%%%-------------------------------------------------------------------------------------------------------------------------------------------------------
 \begin{figure}[H]
\centering  
\includegraphics[width=0.9\textwidth]{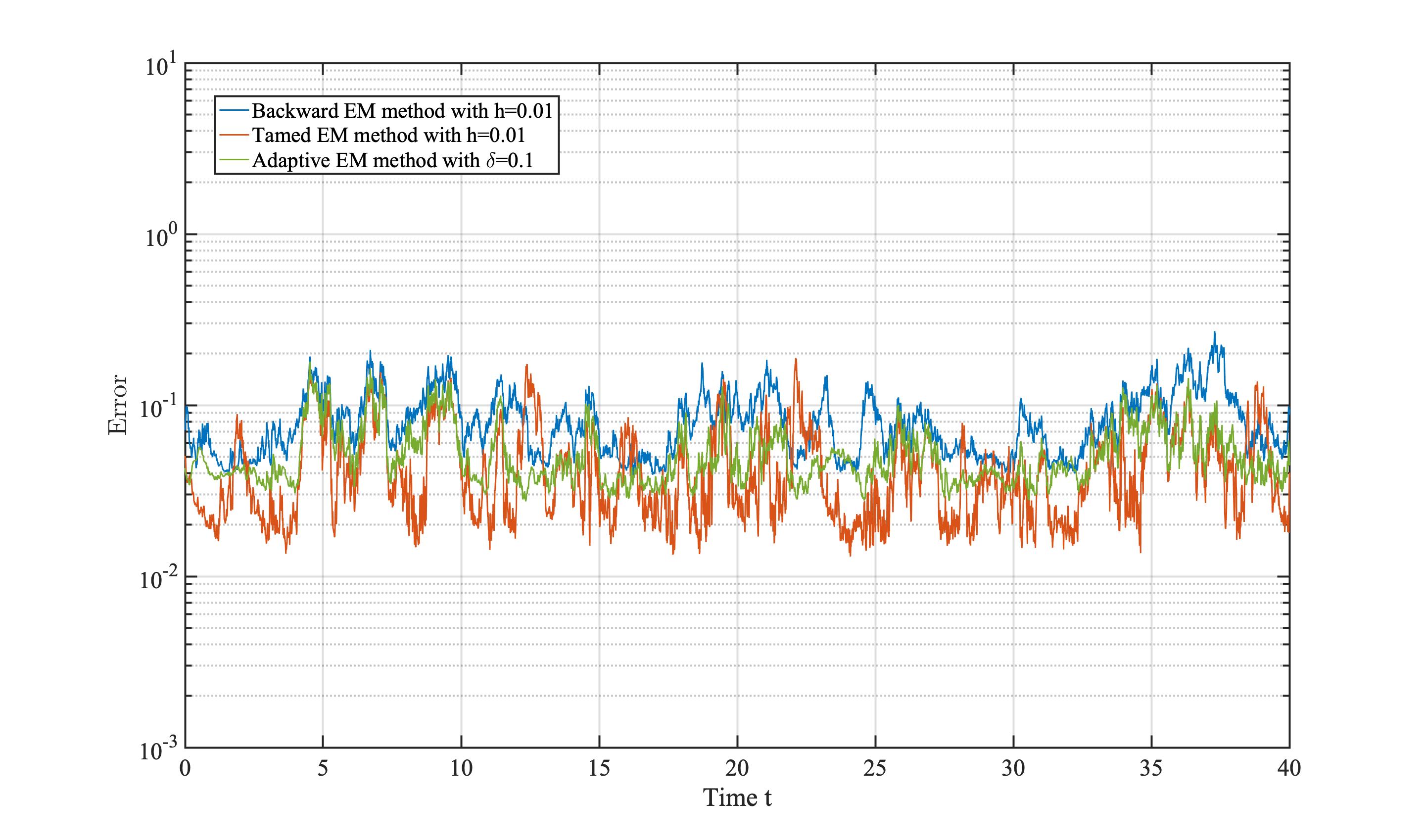}
\caption{Long-time error for the BEM, TEM, and AEM.}
\label{Fig.main1}
\end{figure}

Then we test the long-time errors for propagation of chaos. Solutions of the interacting particle systems are simulated using BEM with the step size $h=0.1$ and the numbers of particles $N=1000, 2000, 4000, 6000$. A reference solution is computed using BEM with the step size $h=0.1$ and a much larger number of particles $N_{ref}=10000$. The noise intensity $\sigma_0=0.1, \sigma_1=0.02$ and the simulation runs up to $T=40$. The $L^1$ propagation of chaos error is shown in Figure  \ref{Fig.main2}.
 \begin{figure}[H]
\centering  
\includegraphics[width=0.9\textwidth]{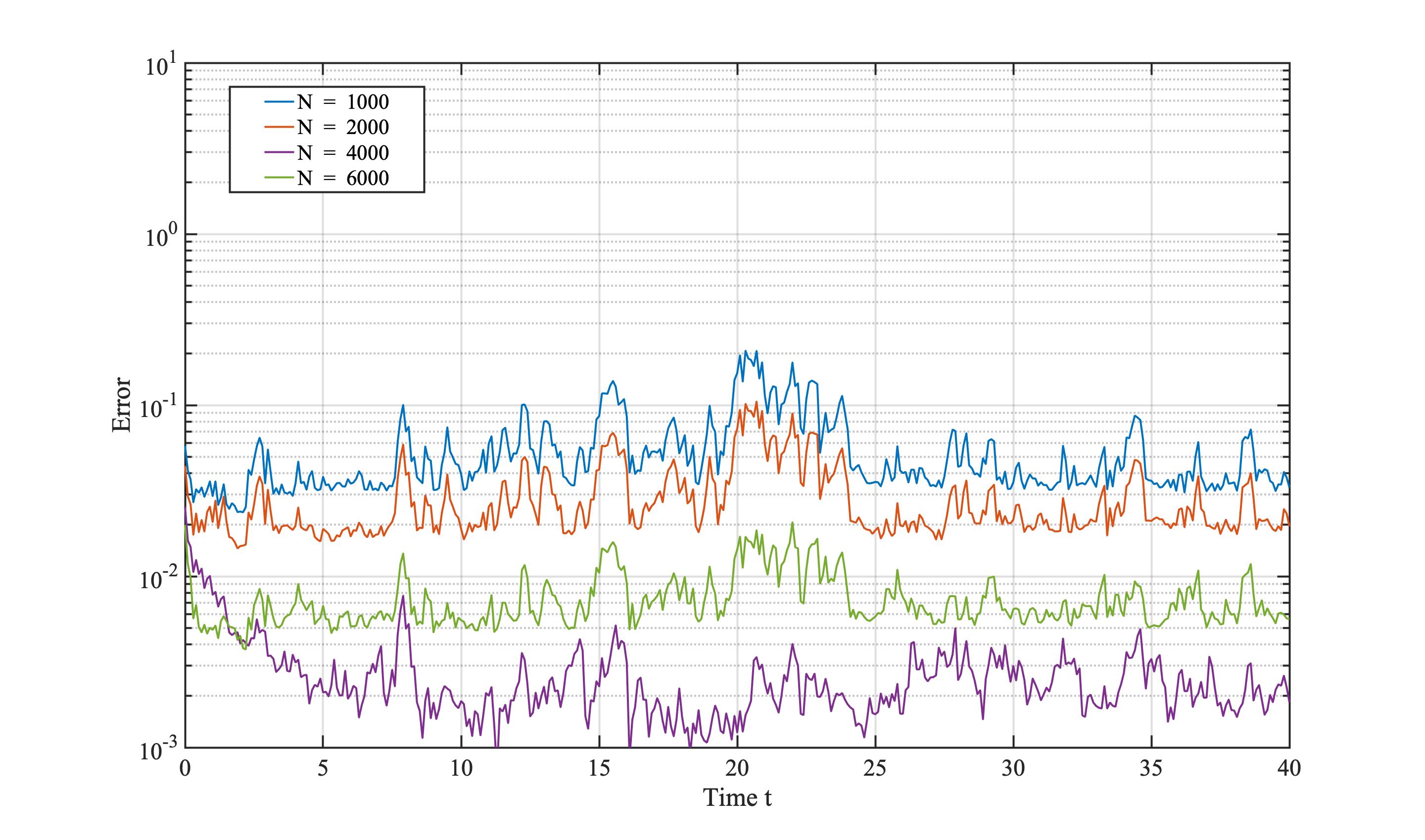}
\caption{Long-time error for propagation of chaos.}
\label{Fig.main2}
\end{figure}

Finally, we verify the exponential forgetting of initial conditions in Figure \ref{Fig.main3}. Both systems are simulated with the BEM using the step size $h=0.01$ and the number of particles $N=500$. In this test, we use moderate noise $\sigma_0=\sigma_1=2$ to make the exponential decay clearly visible within a reasonable simulation time. System  1 always starts from the standard normal distribution $N(0,1)$. System  2 starts from three different distributions: the uniform distribution $U(-1,1)$, $U(-2,2)$, and the normal distribution $N(0,4)$. 
 \begin{figure}[H]
\centering  
\includegraphics[width=0.9\textwidth]{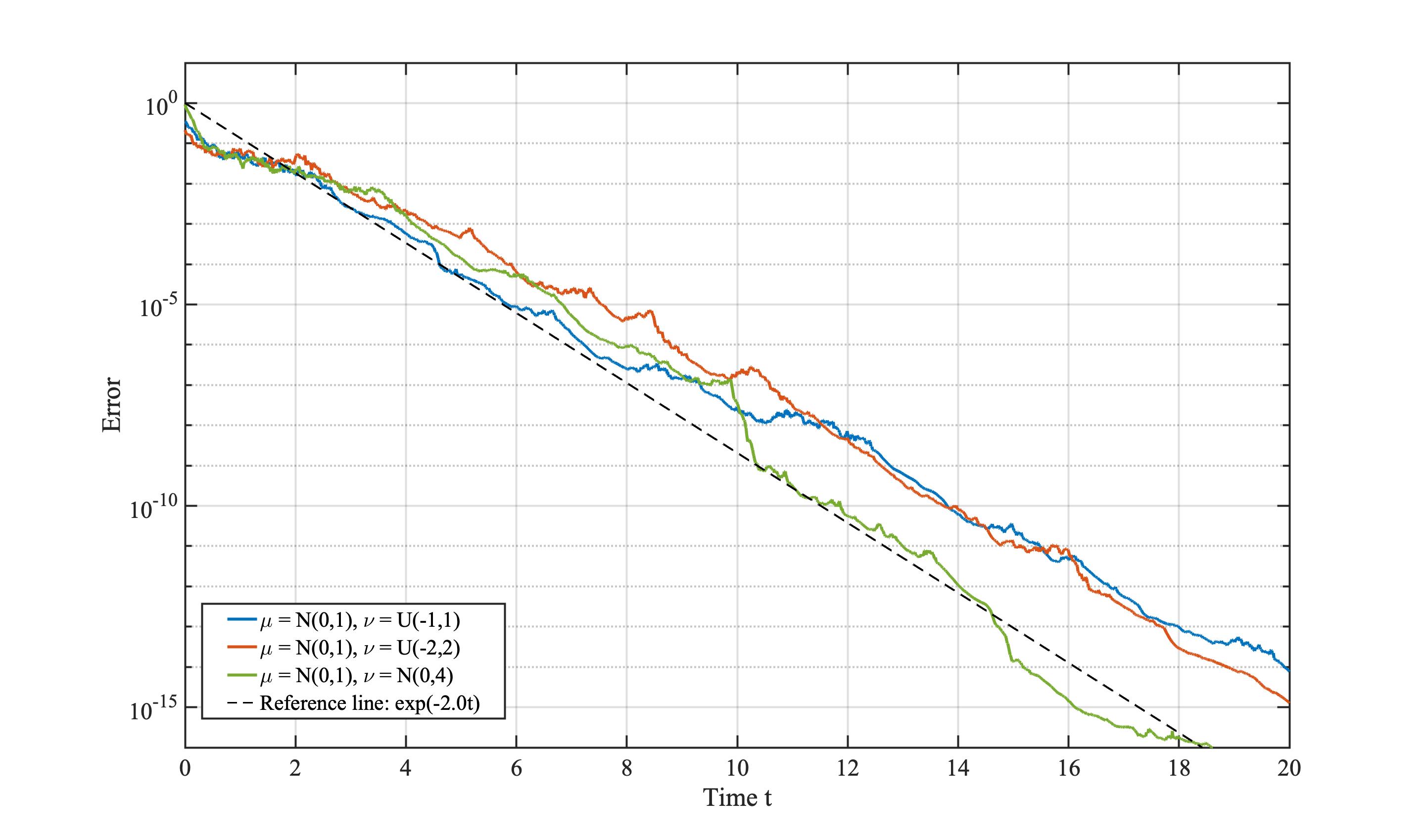}
\caption{Long-time error for different initial data.}
\label{Fig.main3}
\end{figure}

\end{example}

\section*{Acknowledgement}
This research was supported by the National Natural Science Foundation of China under Grant No. 12471372. The authors deeply grateful to Prof. Jianhai Bao for his insightful and helpful comments.

%%%%%%%%%%%%%%%%%%%%%%%%%%%%%

\end{document}